\documentclass[a4paper,11pt]{amsart}

\usepackage[utf8]{inputenc}
\DeclareUnicodeCharacter{00A0}{ }
\usepackage{amssymb}
\usepackage{MnSymbol}
\usepackage{enumitem}

\usepackage[normalem]{ulem}

\usepackage{tcolorbox}

\usepackage{ae}
\usepackage[utf8]{inputenc}
\numberwithin{figure}{section}

\usepackage{mathrsfs,a4wide}
\usepackage{esint}
\usepackage{color}

\usepackage{comment}

\usepackage[colorlinks=true, linkcolor=red, urlcolor=red,citecolor=red]{hyperref}

\usepackage{color}

\usepackage{import}

\usepackage[leqno]{amsmath}

\numberwithin{figure}{section}

\newtheorem{theorem}{Theorem}[section]

\newtheorem{lemma}[theorem]{Lemma}
\newtheorem{proposition}[theorem]{Proposition}
\newtheorem{corollary}[theorem]{Corollary}
\theoremstyle{definition}

\newtheorem{remark}[theorem]{Remark}
\numberwithin{equation}{section}

\addtocontents{toc}{\setcounter{tocdepth}{1}}

\newcommand{\id}{\mathrm{id}}

\newcommand{\R}{\mathbb{R}}
\newcommand{\N}{\mathbb{N}}

\newcommand{\Ha}{\mathcal{H}}

\newcommand{\beq}{\begin{equation}}
\newcommand{\eeq}{\end{equation}}
\newcommand{\dist}{{\rm dist}}
\newcommand{\eps}{\varepsilon}
 
\newcommand{\la}{\langle}
\newcommand{\op}{\mathrm{op}}
\newcommand{\ra}{\rangle}
\newcommand{\dia}{\mathrm{diam}}
\newcommand{\diver}{\operatorname{div}}
\newcommand{\pa}{\partial}

\newcommand{\spt}{\mathrm{spt}}
\newcommand{\medint}{-\kern -,375cm\int}
\newcommand{\medintinrigo}{-\kern -,315cm\int}

\newcommand{\Tr}{\text{Tr}}

\renewcommand{\d}{\mathrm{d}}

\begin{document}

\title[Consistency of the flat flow]{Consistency of the flat flow solution to  the volume preserving mean curvature flow}

\author{Vesa Julin}

\author{Joonas Niinikoski}


\begin{abstract} 
We consider the flat flow solution, obtained via discrete minimizing movement scheme, to  the volume preserving mean curvature flow starting from $C^{1,1}$-regular set. We prove the consistency principle which states that (any) flat flow solution agrees with the classical solution as long as the latter exists. In particular, flat flow solution is unique and smooth up to the first singular time. We obtain the result by proving the full regularity for  the discrete time approximation of the  flat flow such that the regularity estimates are stable with respect to the time discretization.   Our method can also be applied in the case of the mean curvature flow and thus it  provides an alternative proof,  not relying on comparison principle,  for the consistency between  the flat flow solution and the classical solution for  $C^{1,1}$-regular initial sets. 
\end{abstract}

\maketitle

\tableofcontents


\section{Introduction}

\subsection{Statement of the Main Theorem}
In this paper we consider the flat flow solution to the volume preserving mean curvature flow, which is a weak notion of solution obtained via discrete minimizing movement scheme. Our main goal is to prove the full regularity of the 
 flat flow up to the first singular time when the initial set is $C^{1,1}$-regular. As a corollary we obtain the consistency principle  between the flat flow and the classical solution.

Let us begin by recalling that a smooth family of sets $(E_t)_{t\in [0,T)}\subset \R^{n+1}$, for some $T>0$, is a solution to the volume preserving mean curvature flow  if it satisfies
\beq \label{eq:VMCF}
V_t = - (H_{E_t} - \bar H_{E_t}),
\eeq
where $V_t$ denotes the normal velocity, $H_{E_t}$ the mean curvature and  $\bar H_{E_t}:=\fint_{\pa E_t}H_{E_t}\ \d \Ha^n $ the integral average of the mean curvature of the evolving boundary $\pa E_t$. An important feature is that \eqref{eq:VMCF}  can be seen as a $L^2$-gradient flow of the surface area. Since it also preserves the volume, it can be regarded as the evolutionary counterpart to the isoperimetric problem. 

If the initial set $E_0$ is regular enough, e.g. it satisfies interior and exterior ball condition,  the equation \eqref{eq:VMCF}  has a unique smooth solution for a short interval of time \cite{ES}. The classical result by Huisken \cite{Hui2} states that for convex initial sets the classical solution exists for all times and converges exponentially fast to a sphere. Similarly it follows from \cite{ES, Joonas} that if the initial set is close to a local minimum of the isoperimetric problem, the equation \eqref{eq:VMCF} does not develop singularities and convergences exponentially fast.  However,  for generic initial sets the equation  \eqref{eq:VMCF} may develop singularities in finite time \cite{M, MaSim}.  In fact, unlike the standard mean curvature flow, \eqref{eq:VMCF} may develop singularities even in the plane and the boundary may also collapse such  that the curvature of the  evolving boundary stays uniformly bounded  up to the singular time.  It is therefore natural to find a proper notion of  weak solution for   \eqref{eq:VMCF}   which is defined for all times even if the flow develops singularities.  The crucial difference between \eqref{eq:VMCF} and the mean curvature flow is that the former is nonlocal and does not satisfy the comparison priciple. Therefore we cannot directly use the notion of viscosity solution to define the level-set solution via the methods  introduced by Chen-Giga-Goto \cite{CGG} and Evans-Spruck \cite{EvSp}, although in \cite{KK} Kim-Kwon are able to find a viscosity solution for \eqref{eq:VMCF} for  star-shaped sets.  Instead, we  may use the gradient flow structure to   obtain a weak solution called \emph{flat flow} via discrete minimizing movement scheme as first introduced by Almgren-Taylor-Wang \cite{ATW} and Luckhaus-St\"urzenhecker \cite{LS}   for the mean curvature flow, and then implemented to the volume preserving setting \eqref{eq:VMCF} by Mugnai-Seis-Spadaro \cite{MSS}. We give the precise definition in Section 3.  The existence of the flat flow solution of \eqref{eq:VMCF} is proven in \cite{MSS} and the recent results   \cite{DeGKu, FJM, JMPS, JN, MoPoSpa} indicate that it has the expected asymptotic behavior. Indeed, it is proven in \cite{JMPS} that in the plane any flat flow solution of \eqref{eq:VMCF}, starting from any set of finite perimeter,  converges exponentially fast to a union of equisize disks. 

One of the main issues with the flat flow solution  is that it has a priori very low regularity. The second issue is that it is not clear if the procedure provides a solution to the equation \eqref{eq:VMCF}   in some weak  sense. The first issue is related to the  regularity and the  second one is the problem of consistency, and it is rather clear that these are closely related to each other. Indeed, the flat flow is obtained as a limit of a discrete minimizing scheme, in the spirit of the Euler implicit method, where the time disretization is led to zero. If the flow remains smooth enough, as the time discretization goes to zero,  then one can show that the limiting flat flow provides a solution  to the  equation  \eqref{eq:VMCF}.   However, the only case when this seems to be known is the case when the initial set is convex. In this case the construction in \cite{BCCN}, which however is slightly different than \cite{MSS}, provides a flow of sets which remains convex and thus gives a solution to \eqref{eq:VMCF}.  One may also define a distributional solution to \eqref{eq:VMCF} (see \cite{MSS}) and in a recent work  Laux \cite{Lau} proves that this notion of solution, and in fact any gradient-ﬂow calibration,  agrees with the classical solution as long as the latter exists (see also  \cite{HL}). 

The issue with regularity and consistency is better understood in the case of the standard mean curvature flow. It is proven in \cite{ATW} that the flat flow for the mean curvature equation agrees with the classical solution as long as the latter exists. If we are in a situation where the level-set solution is unique, i.e., it  does not develop fattening,  then due to the result by Chambolle \cite{Cha} we know that the  flat flow coincides with the level-set solution, see also \cite{CMP, CMP2}. We may then  use the result in \cite{EvSp4} to conclude that  the flat flow is a 'subsolution' to the mean curvature flow in the sense of  Brakke and has  the partial regularity proven in \cite{Bra}.  Thus we have  the consistency and partial regularity for the mean curvature flow when the flow does not develop fattening. In addition, due to the recent result by DePhilippis-Laux \cite{DL} together with the classical result in \cite{LS}, we know that  the flat flow is a distributional solution to the mean curvature flow equation when the initial set is mean convex.

As we mentioned above, here we study the regularity of the flat flow solution of \eqref{eq:VMCF} when the initial set is $C^{1,1}$-regular, which is the same as to say that the set satisfies interior and exterior ball conditions. Throughout the paper we will say that an open set $E \subset \R^{n+1}$ satisfies \emph{uniform ball condition with radius } $r>0$ if  it satisfies  interior and exterior ball condition with radius $r>0$.  Our main theorem reads as follows. 
\begin{theorem}
\label{thm1}
Assume that $E_0 \subset \R^{n+1}$  is an open and bounded set which satisfies uniform ball condition (UBC) with radius $r_0$. There   is time $T_0>0$, which depends on $r_0$ and $n$, such that any flat flow solution  $(E_t)_{t \geq 0}$  of \eqref{eq:VMCF} starting from $E_0$  satisfies UBC  with radius $r_0/2$ for all $t \leq T_0$.  This condition is open in the sense that if $(E_t)_{t \geq 0}$ satisfies UBC with radius $r$ for all $t \leq T$, then there is $\delta>0$ such that it satisfies UBC with radius $r/2$ for all $t < T+\delta$.  

Moreover, the flat  flow $(E_t)_{t \geq 0}$  becomes instantaneously smooth and remains smooth as long  as it satisfies UBC. To be more precise, if $(E_t)_{t \geq 0}$ satisfies UBC with radius $r$  for all $t \leq T$, then for every $k \in \N$ it holds
\beq \label{eq:smoothing}
\sup_{t \in (0,T]} \big( t^k \|H_{E_t}\|_{H^{k}(\pa E_t)}^2 \big) \leq C_k,
\eeq
where $C_k$ depends on $T$, $n$, $k$, $r$ and $|E_0|$.
\end{theorem}
In fact, we obtain even stronger result since we prove the uniform ball condition and the estimate \eqref{eq:smoothing} directly for the discrete approximative flat flow $(E_t^h)_{t \geq 0}$ such that the estimates hold for all $h \leq h_0$ for constants independent of $h$. However, we choose to state the regularity result only for the limiting flow since the precise statement, which can be found in Theorem \ref{thm2} and Theorem \ref{thm3},   is rather technical. The first part of the theorem  is related  to the result by Swartz-Yip \cite{SY}, where the authors prove curvature bounds for the Merriman-Bence-Osher thresholding algorithm for the mean curvature flow.

It is well-known that we have uniqueness among  smooth solutions of \eqref{eq:VMCF}. Therefore an important consequence of Theorem \ref{thm1} is the consistency between the notion of flat flow solution and the classical solution of \eqref{eq:VMCF} when the initial set is $C^{1,1}$-regular.

\begin{corollary}\label{coro}
Assume that $E_0 \subset \R^{n+1}$  is an open and  bounded set which satisfies  uniform ball condition. Let $(\hat E_t)_{t\in [0,T)}\subset \R^{n+1}$ be the classical solution of \eqref{eq:VMCF} starting from $E_0$, where $T>0$ is the maximal time of existence, and let $(E_t)_{t\geq 0}\subset \R^{n+1}$ be a flat flow solution of \eqref{eq:VMCF} starting from $E_0$. Then 
\[
\hat E_t = E_t  \qquad \text{for all } \, t \in [0,T).
\]
\end{corollary}

Let us next briefly comment on the regularity estimate \eqref{eq:smoothing}. The first part of Theorem \ref{thm1} (see Theorem \ref{thm2} in Section 4) provides a bound for the uniform ball condition for a short time $[0,T_0]$ and the proof of Theorem \ref{thm2} also provides an estimate how the curvature grows in time for the approximative flat flow $(E_t^h)_{t \geq 0}$. However, without higher order regularity bounds we are not able to pass these growth-estimates to the limit as $h \to 0$ (see the discussion at the end of Section 5). Therefore our main motivation to prove \eqref{eq:smoothing} is to pass these curvature estimates to the limit as $h \to 0$ by Ascoli-Arzela theorem, and deduce that the uniform ball condition is, in fact, an open condition and therefore the flat flow agrees with the classical solution over the whole maximal time of existence. Of course, in addition to that, \eqref{eq:smoothing} quantifies the smoothing effect of the equation in a sharp way.

\subsection{An overview of the proof}

The proof of Theorem \ref{thm1} is divided in three sections and therefore we give here a short  overview.  We recall that in the minimizing movements scheme, for a fixed time discretization step $h>0$, we obtain a sequence of sets  $E_{k}^h$ such that $E_0^h = E_0$ is the initial set and $E_{k+1}^h$ is defined inductively as a minimizer of the functional 
\[
\mathcal{F}_h(E, E_{k}^h) = P(E) + \frac{1}{h} \int_E d_{E_{k}^h} \, dx  + \frac{1}{\sqrt{h}}\big| |E| - |E_0|\big|,
\]
where $d_{E_{k}^h}$ denotes the signed distance function. A flat flow is then defined as any cluster point of the discrete flow as $h \to 0$.  We first prove in Proposition \ref{prop:distance-bound} via energy comparison argument, that if $E_{k}^h$ is smooth and satisfies UBC with radius $r_0$ then the subsequent set $E_{k+1}^h$ satisfies the following distance estimate 
\[
|d_{E_{k}^h}| \leq \frac{C}{r_0} h \qquad \text{on } \, \pa E_{k+1}^h. 
\]
The above estimate is crucial as it implies that the speed of the discrete flow is sublinear. It also implies a bound for the mean curvature and the regularity of $E_{k+1}^h$ by  applying the Allard's regularity theory \cite{Allard}. The most crucial part of the proof of the main theorem is then to show that the subsequent set $ E_{k+1}^h$ also satisfies UBC with a quantified radius. 

We solve this problem by adopting the two-point function method due to Huisken \cite{Hui} to the discrete setting (see also the works by Andrews \cite{And} and Brendle \cite{Bre} for an overview of the topic).  The idea  is to  double the variables and to study the maximum and minimum values of the function 
\[
S_{E_{k}^h}(x,y) = \frac{(x-y)\cdot \nu(x)}{|x-y|^2}
\]
for $x\neq y \in \pa E_{k}^h$. The point is that the extremal values of $S_{E_{k}^h}$ are related to the uniform ball condition radius of the set $E_{k}^h$ (see Lemma \ref{lem:S-E}). We use the maximum principle to prove the following familiar inequality (see Lemma \ref{lem:2-point-arg})
\[
\frac{\|S_{E_{k+1}^h}\|_{L^\infty} -\|S_{E_{k}^h}\|_{L^\infty}}{h} \leq C \|S_{E_{k}^h}\|_{L^\infty}^3. 
\]
By iterating the above estimate, we obtain that  the sets $E_{k}^h$ satisfy UBC for  all $k \leq T_0 h^{-1}$, where the constant $T_0$ is related to the UBC of the initial set.   This implies the first part of Theorem \ref{thm1} (see Theorem \ref{thm2}). An important technical  part in this argument is the discrete version of the formula for $\frac{d}{dt} \nu_{E_t}$ which we derive in Lemma  \ref{lem:maaginen}. 

The formula  in Lemma \ref{lem:maaginen} is, in fact, so simple that we are able to differentiate it multiple times and obtain in Proposition \ref{prop:maaginen-deri-2} a discrete analog for the formula 
\beq \label{eq:continuous-deri-high}
\frac{\d}{\d t} \Delta^k H_{E_t} = \Delta^{k+1} H_{E_t} + \text{lower order terms}, 
\eeq
where $\Delta$ denotes the Laplace-Beltrami operator (see e.g.  \cite{Mantegazza2002}). The lower order terms are due to the nonlinearity of the equation \eqref{eq:VMCF} and we need the notation and tools from differential geometry  in order to control them. We stress that this is the only part in the paper where we need  to introduce higher order covariant derivatives.   After we have obtained the discrete version of the formula \eqref{eq:continuous-deri-high} and bounded the lower order error terms, we may adopt the argument from \cite{FJM3D} to the discrete setting and obtain the full regularity of the flow. Finally we point out that the argument can be adopted to the case of the mean curvature flow essentially without any modifications.

\section{Notation and preliminary results}
Throughout this paper, $C_n \in \R_+$ stands for a generic dimensional constant which may change from line to line.
We denote the open ball with radius $r$ centered at $x$ by $B_r(x) \subset \R^{n+1}$ and by $B_r$ if it is centered at the origin. We denote by $\textbf{C}(x,r,R) \subset \R^{n+1}$  the open cylinder 
\[
\textbf{C}(x,r,R) := B_r^n(x') \times (-R+ x_{n+1},R+x_{n+1}),
\]
where $B_r^n \subset \R^n$ denotes  the $n$-dimensional ball and $x = (x',x_{n+1}) \in \R^n \times \R$.
For a given set $E \subset \R^{n+1}$ and a radius $r \in \R_+$ we set 
its $r$-enlargement $\mathcal N_r(E) = \{ x \in \R^{n+1} : \dist(x,E) < r\}$. Note that we may alternatively write this
as the Minkowski sum $E + B_r$.
The notation $\nabla^k F$ stands for $k$:th order differential of a vector field $F : \R^{n+1} \to \R^m$. For a matrix $\mathcal A \in \R^k \otimes \R^k$ we denote by $|\mathcal A|$ its Frobenius norm $\sqrt{\Tr(\mathcal A^\text{T} \mathcal A)}$ and by $|\mathcal A|_{\op}$  its operator norm $\max\{|\mathcal A \, \xi| : \xi \in \R^k, |\xi|=1\}$.

If a set $S \subset \R^{k}$ is Lebesgue-measurable, we denote its $k$-dimensional Lebesgue measure (or volume) by $|S|$. Given a non-empty set $E \subset \R^{n+1}$ we denote the distance function by $\text{dist}_E(x) := \inf_{y \in E}|x-y|$ and the \emph{signed distance function} by $d_E:\R^{n+1} \rightarrow \R$ , which is defined as
\beq
\label{sdf}
d_E(x) := \begin{cases} \text{dist}_E(x) , \,\,&\text{for }\, x \in  \R^{n+1} \setminus E\\
- \text{dist}_{\R^n \setminus E}(x)  , \,\, &\text{for }\, x \in  E.  
\end{cases} 
\eeq
Then clearly it holds $\text{dist}_{\pa E} = |d_E|$. If for a given  point $x \in \R^{n+1}$ there is a unique distance minimizer $y_x$ on 
$\pa E$ (that is $|x-y_x|=\text{dist}_{\pa E}(x)$), we denote $y_x$ by $\pi_{\pa E}(x)$ and call it the projection of $x$ onto $\pa E$.   
For a set of finite perimeter $E \subset \R^{n+1}$ we denote its reduced boundary by $\pa^* E$. Then $P(E;F)= \Ha^n(\pa^* E \cap F)$ for every Borel set $F \subset \R^{n+1}$ and $P(E) = \Ha^n(\pa^* E)$.

\subsection{Regular sets and tangential differentiation}
We will mostly deal with regular and bounded sets $E \subset \R^{n+1}$. 
As usual, a bounded set $E \subset \R^{n+1}$ is said to be $C^{k,\alpha}$-regular, with $k \geq 1$ and $0\leq \alpha \leq 1$, if for every $x \in \pa E$ we find a cylinder $\textbf{C}(x,r,R)$ and a function $f \in C^{k,\alpha}(B^n_r(x'))$ with $|f-x_{n+1}|<R$
such that, up to rotating the coordinates, we may write
\[
\mathrm{int}(E) \cap \textbf{C}(x,r,R) = \{y \in \textbf{C}(x,r,R): y_{n+1} < f(y')\}.
\]
In particular, $\pa E$ is a compact and embedded $C^{k,\alpha}$-hypersurface. Again, if $\alpha = 0$, we say that $E$ is $C^k$-regular and if $k=\infty$, we say that $E$ is smooth. If $r$ and $R$ are independent of the choice of $x$ and the $C^{k,\alpha}$-norm of $g$ has a bound, also independent of $x$, then we say that $E$ is uniformly $C^{k,\alpha}$-regular. 
We denote the outer unit normal by $\nu_E$, or simply $\nu$ if the meaning is clear from the context.
Note that  $\nu_E \in C^{k-1,\alpha}(\pa E; \pa B_1)$. We always assume that the orientation of $\pa E$ is induced by $\nu_E$.
We define the matrix field $P_{\pa E} : \pa E \rightarrow \R^{n+1} \otimes \R^{n+1}$ by setting $P_{\pa E} = I - \nu_E\otimes
\nu_E$. For a given  point $x \in \pa E$ the map $P_{\pa E}(x)$ is the orthogonal projection onto the \emph{geometric tangent plane} $G_x \pa E := \langle \nu_E (x) \rangle^\perp$. 

For given a vector field $F \in C^l(\R^{n+1};\R^m)$ with $1 \leq l \leq k$ we define its \emph{tangential differential} along $\Sigma = \pa E$ as
a matrix field $\nabla_{\tau_E} F : \pa E \rightarrow \R^m \otimes \R^{n+1}$ by setting
\beq
\label{eq:tangdiff}
\nabla_{\tau_E} F =\nabla F P_{\pa E} = \nabla F - (\nabla F   \nu_E) \otimes  \nu_E.
\eeq
When the meaning is clear from the context, we abbreviate $E$ from the notation and write simply $\nabla_\tau F$. 
In the case $m=n+1$, the \emph{tangential divergence} of $F$  is defined as $\diver_\tau F = \Tr(\nabla_\tau F)$
and the \emph{tangential Jacobian} $J_\tau F$ of $F$ is defined on $\pa E$ as
\beq
\label{def:tJacob}
J_\tau F = \sqrt{\det\left((\nabla_\tau F \circ \iota_\tau)^T(\nabla_\tau F \circ \iota_\tau)\right)},
\eeq
where $\iota_\tau (x)$ at $x \in \pa E$ is the inclusion $G_x \pa E \hookrightarrow \R^{n+1}$. In the case $m=1$, the notation $\nabla_\tau F$ also stands for the \emph{tangential gradient}  $P_{\pa E} \nabla F$. 
Note that $\nabla_\tau F$ is $C^{l-1}$-regular and independent of how $F$ is extended beyond $\pa E$. 
On the other hand, every $ G \in C^l(\pa E,\R^m)$, with $1 \leq l \leq k$, admits a $C^k$-extension $ F : \R^n \rightarrow \R^m$
so we may extend the concept of tangential differential to concern $G$ simply by setting $\nabla_\tau G = \nabla_\tau F$ and further define the other introduced concepts in similar manner.

If $E$ is $C^k$-regular for $k \geq 2$, we may define its \emph{second fundamental form}, with respect to the orientation 
$\nu_E$, as a matrix field $B_E : \pa E\rightarrow \R^{n+1}\otimes\R^{n+1}$ given by
\[
B_E (x) = \sum_i \lambda_i (x) \kappa_i (x) \otimes \kappa_i (x),
\]
where the (unit)  principal directions $\kappa_1 (x), \ldots, \kappa_n (x) \in \langle \nu_E(x) \rangle^\perp$
and the principal curvatures $\lambda_1 (x), \ldots, \lambda_n (x)$ at $x \in \pa E$ are given by the 
orientation $\nu_E$. The corresponding (scalar) mean curvature field $H_E$ is then given pointwise as the sum of the principal curvatures, i.e., $H_E = \Tr(B_E)$. Note that  we may simply write 
\beq
\label{eq:BH}
B_E = \nabla_\tau \nu_E \quad \text{and} \quad H_E = \diver_\tau \nu_E.  
\eeq
Finally, we define the \emph{tangential Hessian} for given $u \in C^2(\pa E)$ as $\nabla_\tau^2 u = \nabla_\tau (\nabla_\tau u)$ and further the 
\emph{tangential Laplacian} or the \emph{Laplace-Beltrami} of $u$ as
\[
\Delta_{\tau} u = \diver_\tau(\nabla_\tau u) = \text{Tr} (\nabla_\tau^2 u). 
\]
The tangential Laplacian $\Delta_\tau F$ for $F \in C^2(\pa E;\R^{n+1})$ is defined as $\sum_i \Delta_\tau (F \cdot e_i) e_i$.
We will need the following identities on $\pa E$
\beq
\label{eq:Delta-id-normal}
\Delta_\tau \id  = - H_E \nu_E \quad \text{and} \quad  \Delta_\tau \nu_E= - |B_E|^2 \nu_E + \nabla_\tau H_E \ \ \text{if $E$ is $C^3$-regular.}
\eeq
The importance of the mean curvature  $H_E$ lies in the surface divergence theorem
which states that for every $G \in C^1(\pa E;\R^{n+1})$ it holds
\beq
\int_{\pa E} \diver_{\tau} G \, \d \Ha^n = \int_{\pa E} H_E (G \cdot \nu_E) \, \d \Ha^n.
\eeq

The concept of mean curvature can be generalized to the setting of bounded sets of finite perimeter in the varifold sense. 
Indeed, for a set of finite perimeter $E \subset \R^{n+1}$, we may define
the tangential divergence $\diver_\tau F$ of $F \in C^1(\R^{n+1};\R^{n+1})$ along $\pa^*E$ in the 
same way as in the regular case by replacing the outer unit normal field with the measure theoretic
normal field $\pa^* E \rightarrow \pa B_1$ which we also denote by $\nu_E$. Then, if $E$ is a bounded set of finite perimeter and there is $g \in L^1(\pa^*E,\Ha^n|_{\pa^* E})$ such that 
\beq
\label{weakcurvature}
\int_{\pa^*E} \diver_{\tau} F \, \d \Ha^n = \int_{\pa^*E} g (F \cdot \nu_E) \, \d \Ha^n 
\eeq
for every $F \in C^1(\R^{n+1};\R^{n+1})$, we say that $g$ is a \emph{generalized mean curvature} of $E$ and denote it by $H_E$.
As mentioned, this is a concept from the context of varifold theory for which we refer to  \cite{Sim} as a standard introduction.  
Since $\pa^* E$ is $\Ha^n$-rectifiable set, one may treat the pair $(\pa^* E, \Ha^n|_{\pa^*E})$ as an rectifiable integral
varifold of multiplicity one.

\subsection{Riemannian geometry}
We need the notation related to Riemannian geometry and as an introduction to the topic we refer to \cite{Lee}. Let us assume that $E \subset \R^{n+1}$ is a smooth and bounded set and denote $\Sigma = \pa E$. Since $\Sigma$ is  embedded in $\R^{n+1}$  it has natural metric $g$ induced by the Euclidian metric. Then $(\Sigma, g)$ is a Riemannian manifold and we denote the inner product on each tangent space $X, Y \in T_x \Sigma$ by $\la X, Y \ra$, which we may write in local coordinates as 
\[
\la X,Y \ra = g(X,Y) = g_{ij} X^iY^j.
\]
We extend the inner product in a natural way for tensors. Note that  $x \cdot y$ denotes the inner product of two vectors in $\R^{n+1}$. We denote smooth vector fields on $\Sigma$ by $\mathscr{T}(\Sigma)$ and by  a slight abuse of notation we denote smooth $k$:th order tensor fields on $\Sigma$  by $\mathscr{T}^k(\Sigma)$. We write $X^i$ for vectors and $Z_i$ for covectors in local coordinates. We denote the Riemannian connection on  $\Sigma$ by $\tilde \nabla$ and recall that  for a function $u \in C^\infty(\Sigma)$ the covariant derivative  $\tilde \nabla u $ is a $1$-tensor field defined for  $X  \in \mathscr{T}( \Sigma)$  as
\[
\tilde \nabla u(X)  = \tilde \nabla_X u = X u,
\]
i.e., the derivative  of $u$ in the direction of $X$. The  covariant derivative  of  a smooth $k$-tensor field $F \in \mathscr{T}^k( \Sigma)$, denoted  by  $\tilde \nabla F$, is a $(k+1)$-tensor field    and  for $ Y_1, \dots, Y_k, X \in \mathscr{T}( \Sigma)$  we have the recursive formula
\begin{equation} \label{eq:recursive}
\tilde \nabla F(Y_1, \dots, Y_k, X) = (\tilde \nabla_X F)(Y_1, \dots, Y_k) ,
\end{equation}
where
\[
(\tilde \nabla_X F)(Y_1, \dots, Y_k) = X F(Y_1, \dots, Y_k) - \sum_{i=1}^k F(Y_1, \dots,  \tilde \nabla_X Y_i ,\dots, Y_k).
\]
Here $\tilde \nabla_X Y$ is  the covariant derivative of $Y$ in the direction of $X$ (see \cite{Lee}) and since $\tilde \nabla$ is the Riemannian connection it holds  $\tilde \nabla_X Y = \tilde \nabla_Y X  + [X,Y]$ for every $X, Y \in \mathscr{T}( \Sigma)$. We denote the $k$:th order covariant derivative of a function $u$ on $\Sigma$ by $\tilde \nabla^k u \in \mathscr{T}^k( \Sigma)$ and the Laplace-Beltrami operator by $\Delta $. Note that for functions it holds  $\Delta u = \Delta_\tau u$. The notation $\tilde \nabla_{i_k} \cdots \tilde \nabla_{i_1} u$ means a  coefficient of $\tilde \nabla^k u$ in local coordinates. We may raise the index of $\tilde \nabla_i u$ by using the inverse of the metric tensor  $g^{ij}$  as $\tilde \nabla^i u = g^{ij}\tilde \nabla_j u$. We note that the tangential gradient of $u : \Sigma \to \R$ is equivalent to its covariant derivative in the sense that for every vector field $X \in \mathscr{T}(\Sigma)$ we find a unique vector field  $\tilde{X} : \Sigma \to \R^{n+1}$ which satisfies $\tilde{X}\cdot \nu_E = 0$  and 
\[
\tilde \nabla_X u = \nabla_\tau u \cdot \tilde{X}.
\]
Similarly it holds $\tilde \nabla^2 u(X,Y) = \nabla_\tau^2 u \tilde X \cdot \tilde Y$. Finally we recall that the notation $\nabla^k$ always stands for the standard Euclidian $k$:th order differential for an ambient function.

 We define the Riemann curvature tensor $R \in \mathscr{T}^4(\Sigma)$ \cite{Lee, MantegazzaBook}  via interchange of covariant  derivatives of a vector field  $Y^i$ and a covector field  $Z_i$ as   
\begin{equation}
\label{eq:curv-tensor}
\begin{split}
&\tilde \nabla_i \tilde \nabla_j Y^s - \tilde \nabla_j \tilde \nabla_i Y^s = R_{ijkl} g^{ks} Y^l,\\
&\tilde \nabla_i \tilde \nabla_j Z_k - \tilde \nabla_j \tilde \nabla_i Z_k =  R_{ijkl} g^{ls} Z_s,
\end{split}
\end{equation}
where we have used the Einstein summation convention. We may write the Riemann tensor in local coordinates by using the second fundamental form $B$, which in the Riemannian setting is understood to be 2-form,  as
\begin{equation}
\label{eq:curv-tensor2}
R_{ijkl} = B_{ik}B_{jl} - B_{il}B_{jk}.
\end{equation}
We will also need the Simon's identity which reads as 
\begin{equation}
    \label{eq:Simon}
    \Delta B_{ij} =  \tilde \nabla_i \tilde \nabla_j  H + H B_{il} g^{ls} B_{sj} - |B|^2 B_{ij}.
\end{equation}

Let us next fix our notation for the function spaces.  We define  the Sobolev space $W^{l,p}(\Sigma)$ in  a standard way for $p \in [1,\infty]$, see e.g. \cite{AubinBook2}, denote the Hilbert space $H^l(\Sigma) = W^{l,2}(\Sigma)$ and  define the associated norm for $u \in W^{l,p}(\Sigma)$ as
\[
\| u\|_{W^{l,p}(\Sigma)}^p = \sum_{k = 0}^l \int_\Sigma |\tilde \nabla^k u|^p\, d \Ha^n
\]
and for $p = \infty$
\[
\| u\|_{W^{l,\infty}(\Sigma)} = \sum_{k = 0}^l \sup_{x \in \Sigma} |\tilde \nabla^k u|.
\]
The above definition extends naturally for tensor fields.  We adopt the convention that $\| u\|_{H^0(\Sigma) }  = \| u\|_{L^2(\Sigma)}$ and  denote $\| u\|_{C^{m}(\Sigma)} = \| u\|_{W^{m,\infty}(\Sigma)}$. We remark that we may define the $k$:th order covariant derivative of a function $u \in C^k(\Sigma)$ and the space $W^{k,p}(\Sigma)$ for $k \geq 2$ as above assuming only that $\Sigma$ (i.e. the set $E$ for which $\Sigma = \pa E$)  is $C^k$-regular.


Finally  we adopt the notation $S \star T$ from \cite{Ham, Mantegazza2002} to denote a tensor formed by contracting  some indexes of  tensors $S$ and $T$ using the coefficients of the metric tensor $g_{ij}$. This notation is useful as it implies 
\[
|S \star T| \leq C |S||T|,
\]
where the constant $C$ depends on the 'structure' of $S \star T$.

\subsection{Functional and geometric inequalities}

We will need standard interpolation inequalities on smooth hypersurfaces. Since we will apply them on the moving boundary given by the flow, we need to control the constants in the inequalities. We begin with a simple interpolation on H\"older norms. 

\begin{lemma}
\label{lem:inter-holder}
Let $\Omega \subset \R^{k}$ be an open set and let $u \in C^1(\Omega)$, then for every $\alpha \in (0,1)$
\[
\|u\|_{C^{0,\alpha}(\Omega)} \leq 3 \|u\|_{L^\infty(\Omega)}^{1-\alpha}\|u\|_{C^1(\Omega)}^{\alpha}. 
\]
\end{lemma}

\begin{proof}

The inequality follows from 
\[
 \frac{|u(y) - u(x)|}{|y-x|^\alpha} \leq  |u(y) - u(x)|^{1-\alpha}  \left(\frac{|u(y) - u(x)|}{|y-x|}\right)^{\alpha} \leq 2 \|u\|_{L^\infty(\Omega)}^{1-\alpha}  \|u\|_{C^1(\Omega)}^{\alpha}. 
\]
\end{proof}

We continue to  introduce  functional and geometric  inequalities that we need in order to prove the higher order regularity estimates stated at the end of Theorem \ref{thm1}. As we already mentioned  we do not need any deep results from differential geometry in order to prove the estimate for the uniform ball condition stated in the beginning of Theorem \ref{thm1}. It is only when we deal with higher order derivatives, i.e., higher than two, we need the notation of covariant derivatives. Recall that we always assume that $\Sigma = \pa E$ for a bounded set $E \subset \R^{n+1}$. 

Let us first recall the interpolation inequality with Sobolev-norms on embedded surfaces. We use the result from \cite[Proposition 6.5]{Mantegazza2002} which states that under curvature bound the standard interpolation inequality holds for a uniform constant.   
\begin{proposition}
\label{prop:interpolation}
Assume $\|B_{\Sigma}\|_{L^{\infty}}, \Ha^n(\Sigma)\leq C_0$ and $\Sigma$ is $C^{m}$-regular for $m \geq 2$.  Then for integers $0\leq k \leq l \leq m$ and numbers $p,q,r \in [1,\infty)$, there is $\theta \in [k/l,1]$ such that for every $C^l$-regular covariant tensor field $T$ on $\Sigma$ it holds
\[
\|\tilde \nabla^k T\|_{L^p(\Sigma)} \leq C \| T\|_{W^{l,q}(\Sigma)}^\theta \| T\|_{L^{r}(\Sigma)}^{1-\theta}
\]
for a constant $C=C(k,l,n,p,q,r,\theta,C_0) \in \R_+$ provided that the following compatibility condition is satisfied
\[
\frac{1}{p} = \frac{k}{n} + \theta \left( \frac{1}{q} - \frac{l}{n} \right) + \frac{1}{r}(1 - \theta).
\]
\end{proposition} 

We denote an index vector by $\alpha \in \N^k$, i.e.,  $\alpha = (\alpha_1, \dots, \alpha_k)$ where  $\alpha_i \in \N$, and define its norm by
\[
|\alpha| = \sum_{i=1}^k \alpha_i.
\]
The following inequality is well-known but we prove it for the reader's convenience. 
\begin{proposition}
\label{prop:kato-ponce}
Assume $\|B_{\Sigma}\|_{L^{\infty}}, \Ha^n(\Sigma) \leq C$ and $\Sigma$ is $C^{m}$-regular for $m \geq 2$. Assume $u_1, \dots, u_l$ are $C^m$-regular functions such that $\|u_i\|_{L^\infty} \leq C$. Then for an index vector $\alpha \in \N^l$ with $|\alpha| \leq k \leq m$   and $p \in (1,\infty)$ it holds 
\[
\| |\tilde \nabla^{\alpha_1} u_1 | \cdots  | \tilde \nabla^{\alpha_l} u_l | \|_{L^p(\Sigma)} \leq C_k \sum_{i = 1}^k  \| u_{i}\|_{W^{k,p}(\Sigma)}.
\]

\end{proposition}

\begin{proof}
Without loss of generality we may assume that $|\alpha| = k$. We first use H\"older's inequality
\[
\| |\tilde \nabla^{\alpha_1} u_1 | \cdots  | \tilde \nabla^{\alpha_l} u_l | \|_{L^p(\Sigma)}  \leq \|\tilde \nabla^{\alpha_1} u_1\|_{L^{\frac{pk}{\alpha_1}}} \cdots  \|\tilde \nabla^{\alpha_l} u_l\|_{L^{\frac{pk}{\alpha_l}}}.
\] 
By the interpolation inequality in Proposition \ref{prop:interpolation} and by $\|u_i\|_{L^\infty} \leq C$ it holds 
\[
\|\tilde \nabla^{\alpha_i} u_1\|_{L^{\frac{pk}{\alpha_i}}} \leq C \| u_i\|_{W^{k,p}}^{\frac{\alpha_i}{k}} \| u_i\|_{L^{\infty}}^{1-\frac{\alpha_i}{k}} \leq C \| u_i\|_{W^{k,p}}^{\frac{\alpha_i}{k}} .
\]
Hence we have 
\[
\| |\tilde \nabla^{\alpha_1} u_1 | \cdots  | \tilde \nabla^{\alpha_l} u_l | \|_{L^p(\Sigma)}  \leq   C_k \| u_1\|_{W^{k,p}}^{\frac{\alpha_1}{k}} \cdots  \| u_l\|_{W^{k,p}}^{\frac{\alpha_l}{k}}.
\]
Since $\alpha_1 + \dots + \alpha_l = |\alpha| = k$ the claim follows from Young's inequality. 
\end{proof}

If $u : \R^{n+1} \to \R$ is  a regular function then its restriction on $\Sigma$ is also regular. In the next lemma we bound the covariant derivatives of $u$ on $\Sigma$ with the Euclidian ones.  The statement of the lemma is not optimal but it is sharp enough for our purpose.   In the proof   we will repeatedly use the fact that the $k$:th order derivative of the composition $f \circ h$ and the product $f \cdot g$ of  functions $f,g : \R^m \to \R^k$  and $h : \R^n \to \R^m$ can be written as 
\beq\label{eq:leibniz}
\begin{split}
\nabla^k (f \circ h) &= \sum_{|\alpha|\leq k-1} \nabla^{1+ \alpha_1} h \star \cdots \star  \nabla^{1+ \alpha_k} h \star  \nabla^{1+ \alpha_{k+1}} f\\
\nabla^k (f \cdot  g) &= \sum_{i+j = k } \nabla^{i} f \star   \nabla^{j} g.
\end{split}
\eeq  
\begin{lemma}
\label{lem:cova-eucl}
Assume $\Sigma $ is $C^{k+2}$-regular and $u \in  C^{k+1}(\R^{n+1})$. Then it  holds for all  $x \in \Sigma$
\[
|\tilde \nabla^{k+1} \,  u(x)| \leq C_k \sum_{|\alpha|\leq k} \big(1+|\tilde \nabla^{\alpha_1} B_{E}(x) |   \cdots |\tilde \nabla^{\alpha_{k}} B_{E}(x) |\big)  \, |\nabla^{1+\alpha_{k+1}} u(x) |. 
\] 
Recall that  $\tilde \nabla^k $ denotes the $k$:th order covariant derivative on $\Sigma$ while $\nabla^k $ is the $k$:th order Euclidian derivative. 
\end{lemma}

\begin{proof}
The proof follows from  basic theory of differential geometry and we merely sketch it. Let us fix  $x \in \Sigma$ and choose the coordinates such that $x = 0$ and $\nu_E(0) = e_{n+1}$. Since $\Sigma$ is $C^{k+2}$-regular hypersurface we may write it locally as a graph of $f \in C^{k+2}(\R^n)$, i.e., $\Sigma \cap B_r(0) \subset \{ (x,f(x)) : x  \in \R^n\}$. Note that since  $\nu_E(0) = e_{n+1}$ then  $\nabla_{\R^n} f(0)= 0 $.  

We consider the graph coordinates  $\Phi^{-1}: B_r^n \to \Phi^{-1}(B_r^n) \subset \Sigma$, $\Phi^{-1}(x) =  (x,f(x))$. We denote the points on $\R^n$ by $x$, the points on $\Sigma$ by $p$,    $\Phi(p) = \big(x^1(p), \dots, x^n(p)\big)$ and $U = \Phi^{-1}(B_r^n) $. Then the chart $\big(U,(x^i) \big)$ determines coordinate vector fields which we denote by $\frac{\pa}{\pa x^i}\Big {|}_{p}$ and recall that they act on smooth functions $v : U \to \R$ at  $p = \Phi(x)$ as 
\[
\frac{\pa}{\pa x^i}\Big {|}_{p} v =  \tilde \nabla u \left(  \frac{\pa}{\pa x^i}\right) (p) = \pa_i ( v \circ \Phi^{-1}) (x),
\]
where $\pa_i$ denotes the standard partial derivative in $\R^n$. It holds for the metric tensor and for the Christoffel symbol  $\Gamma_{jk}^i $ (see \cite{Lee})  for $x \in B_r^n$
\[
g_{ij}(x) = \delta_{ij} + \pa_i f(x) \pa_j f(x)\quad \text{and} \quad \Gamma_{jk}^i(x) = g^{il}(x)\,  \pa_{jk}^2 f(x) \pa_l f(x).
\]
Moreover by the recursive formula \eqref{eq:recursive} we may write the $(k+1)$:th order covariant derivative of  $u$  iteratively (see \cite[Lemma 4.8]{Lee}) as 
\begin{equation}
\begin{split}
\label{eq:cova-eucl-1}
\tilde \nabla^{k+1} u \left(\frac{\pa}{\pa x^{i_1}}, \dots, \frac{\pa}{\pa x^{i_k}}, \frac{\pa}{\pa x^j}\right) = &\pa_j \left(\tilde \nabla^{k} u \left(\frac{\pa}{\pa x^{i_1}}, \dots, \frac{\pa}{\pa x^{i_k}}\right)\right) \\
&- \sum_{m=1}^k \tilde \nabla^{k} u \left(\frac{\pa}{\pa x^{i_1}}, \dots, \frac{\pa}{\pa x^{l}}, \dots , \frac{\pa}{\pa x^{i_k}}\right)\,   \Gamma_{j i_m}^l. 
\end{split}
\end{equation}
Recall that $\tilde \nabla u \left(\frac{\pa}{\pa x^{i}}\right)(p) = \frac{\pa}{\pa x^i}\Big {|}_{p} u $.

Using \eqref{eq:leibniz}  we have  
\[
|\nabla_{\R^n}^{k+1} \,  (u \circ \Phi^{-1})(0)| \leq C_k \sum_{|\alpha|\leq k }\big(1+ |\nabla_{\R^n}^{1+\alpha_1}  f(0)|\cdots  |\nabla_{\R^n}^{1+\alpha_k}  f(0)|\big)|\nabla^{1+\alpha_{k+1}} \,  u(0)|.
\]
We use \eqref{eq:cova-eucl-1} and \eqref{eq:leibniz}, and obtain after long but straightforward calculation that 
\[
|\tilde \nabla^{k+1} \,  u(0)| \leq C_k  \sum_{|\alpha|\leq k }\big(1+ |\nabla_{\R^n}^{1+\alpha_1}  f(0)|\cdots  |\nabla_{\R^n}^{1+\alpha_k}  f(0)|\big)||\nabla^{1+\alpha_{k+1}} \,  u(0)|.
\]
Note that $\nu_{E} \circ \Phi^{-1} = \frac{(-\nabla_{\R^n} f,1)}{\sqrt{1 + |\nabla_{\R^n}f|^2}}$. We thus obtain by \eqref{eq:leibniz}  that 
\beq \label{eq:cova-eucl}
\begin{split}
 |\nabla_{\R^n}^{l+1}  f(0)| &\leq  C_l  \sum_{|\beta|\leq l }\big( 1+| \nabla^{\beta_1}(\nu_{E} \circ \Phi^{-1}) |\cdots  |\nabla^{\beta_l}  (\nu_{E} \circ \Phi^{-1})|\big) \\
&\leq C_l  \sum_{|\beta|\leq l -1}\big( 1+|\tilde \nabla^{\beta_1}  B_{E} |\cdots  |\tilde \nabla^{\beta_l}  B_{E}|\big)
\end{split}
\eeq
and the claim follows. 
\end{proof}

Next we turn our focus on geometric inequalities on compact hypersufaces. Recall that by classical  results e.g. from \cite{AubinBook2}  it holds $\|u\|_{H^2(\Sigma)}\leq C(\|\Delta u\|_{L^2(\Sigma)} + \|u\|_{L^2(\Sigma)})$ and e.g. in \cite{FJM3D} it is proven that $\|u\|_{H^{2k}(\Sigma)}\leq C(\|\Delta u\|_{H^k(\Sigma)} + \|u\|_{L^2(\Sigma)})$. We need these results  with a quantitative control on the constant. 
\begin{lemma}
\label{lem:laplace}
Assume $\Sigma$ is $C^{2k+2}$-regular and $\|B_{\Sigma}\|_{L^{\infty}}, \Ha^n(\Sigma)\leq C$. Then for all  $u \in C^{2k+1}(\Sigma)$ it holds 
\[
\|u\|_{H^{2k}(\Sigma)} \leq C_k(\|\Delta^k u\|_{L^2(\Sigma)} +(1+ \|B_\Sigma\|_{H^{2k-1}(\Sigma)})\|u\|_{L^\infty(\Sigma)})
\]
and
\[
\|u\|_{H^{2k+1}(\Sigma)} \leq C_k(\|\tilde \nabla \Delta^k u\|_{L^2(\Sigma)} +(1+\|B_\Sigma\|_{H^{2k}(\Sigma)})\|u\|_{L^\infty(\Sigma)}).
\]
\end{lemma}
\begin{proof}
We only prove the first inequality since the second follows from the same argument. The proof is similar to \cite[Proposition 2.11]{JL} but we sketch it for the reader's convenience.  Denote $l = 2k$. We begin by noticing that we may interchange the derivatives of the $(l+1)$:th order covariant derivative of $u$ by using \eqref{eq:curv-tensor}, \eqref{eq:curv-tensor2}, \eqref{eq:cova-eucl-1} and  the curvature bound $\|B_{\Sigma}\|_{L^{\infty}}\leq C$    (see also \cite[Proof of Lemma 7.3]{Mantegazza2002}) 
\[
\begin{split}
|\tilde \nabla_{i_{l+1}} \cdots \tilde \nabla_{i_{m+1}} \tilde \nabla_{i_{m}}  \cdots  \tilde \nabla_{i_{1}} u - \tilde \nabla_{i_{l+1}}\cdots  \tilde \nabla_{i_{m}} \tilde \nabla_{i_{m+1}} \cdots  \tilde \nabla_{i_{1}} u| \leq C_l \sum_{|\alpha|\leq l-1}(1+ |\tilde \nabla^{\alpha_1} B_\Sigma| \cdots |\tilde \nabla^{\alpha_{l-1}} B_\Sigma|) |\tilde \nabla^{\alpha_{l}} u|.
\end{split}
\]
We leave the details for the reader. This holds pointwise on $\Sigma$ and we use it without further mentioning. Let us denote $F = \tilde \nabla^{2k-2} u$ and denote its components simply by $F_{\beta}$, where $\beta = (i_1, \dots, i_{2k-2})$. Then it holds by divergence theorem,  by interchanging the derivatives and by Proposition \ref{prop:kato-ponce} 
\[
\begin{split}
\int_{\Sigma} &|\tilde \nabla^{2k} u|^2 \, d \Ha^{n} = \int_{\Sigma} |\tilde \nabla^2 F|^2 \, d \Ha^{n} = \int_{\Sigma}\tilde \nabla_i \tilde \nabla_j F_\beta \tilde \nabla^i \tilde \nabla^j F^\beta \, d \Ha^{n} = - \int_{\Sigma}  \tilde \nabla_j F_\beta  \tilde \nabla_i \tilde \nabla^i \tilde \nabla^j F^\beta \, d \Ha^{n}\\
&\leq - \int_{\Sigma}  \tilde \nabla_j F_\beta \tilde \nabla^j \tilde \nabla_i\tilde \nabla^i  F^\beta \, d \Ha^{n} + C_k\sum_{|\alpha|\leq l-1} \int_{\Sigma}(1+ |\tilde \nabla^{\alpha_1} B_\Sigma|^2 \cdots |\tilde \nabla^{\alpha_{l-1}} B_\Sigma|^2) |\tilde \nabla^{\alpha_{l}} u|^2 \,  d \Ha^{n}\\
&\leq \int_{\Sigma} \tilde \nabla^j \tilde \nabla_j F_\beta \tilde \nabla_i \tilde \nabla^i  F^\beta \, d \Ha^{n} + C_k( \|u\|_{H^{l-1}(\Sigma)}^2 + \|u\|_{L^{\infty}(\Sigma)}^2\|B_\Sigma\|_{H^{l-1}(\Sigma)}^2 )\\
&=  \int_{\Sigma} |\Delta  \tilde  \nabla^{2k-2} u|^2 \, d \Ha^{n} + C_k(\|u\|_{H^{2k-1}(\Sigma)}^2 + \|u\|_{L^{\infty}(\Sigma)}^2\|B_\Sigma\|_{H^{2k-1}(\Sigma)}^2 ).
\end{split}
\]
By interchanging the derivatives and arguing as above  we obtain 
\[
 \int_{\Sigma} |\Delta   \tilde \nabla^{2k-2} u|^2 \, d \Ha^{n} \leq  \int_{\Sigma} | \tilde \nabla^{2k-2} \Delta   u|^2 \, d \Ha^{n}+ C_k( \|u\|_{H^{2k-1}(\Sigma)}^2 + \|u\|_{L^{\infty}(\Sigma)}^2\|B_\Sigma\|_{H^{2k-1}(\Sigma)}^2 ).
\]
By repeating the argument by replacing $u$ with   $\Delta^j u$, for $j = 1, \dots, k-1$,  we deduce 
\[
\int_{\Sigma} |\tilde \nabla^{2k} u|^2 \, d \Ha^{n} \leq  \int_{\Sigma} |\Delta^k u|^2 \, d \Ha^{n} + C_k( \|u\|_{H^{2k-1}(\Sigma)}^2 + \|u\|_{L^{\infty}(\Sigma)}^2\|B_\Sigma\|_{H^{2k-1}(\Sigma)}^2 ).
\]
The claim follows from interpolation inequality (Proposition \ref{prop:interpolation}) as for $\theta \in (0,1)$ it holds  
\[
 \|u\|_{H^{2k-1}(\Sigma)}^2 \leq  \|u\|_{H^{2k}(\Sigma)}^{2 \theta} \|u\|_{L^{\infty}(\Sigma)}^{2(1-\theta)} \leq \eps  \|u\|_{H^{2k}(\Sigma)}^2 + C_\eps \|u\|_{L^{\infty}(\Sigma)},
\]
where the last inequality follows from Young's inequality. 
\end{proof}

 Lemma \ref{lem:laplace} together with Simon's identity \eqref{eq:Simon} imply the following inequality.   
\begin{proposition}
\label{prop:mean-curv}
Assume $\Sigma$ is $C^{2k+3}$-regular and  $\|B_{\Sigma}\|_{L^{\infty}}, \Ha^n(\Sigma)\leq C$.  Then it holds
\[
\|B_{\Sigma}\|_{H^{2k}(\Sigma)} \leq C_k(1+ \| \Delta^k H_\Sigma \|_{L^2(\Sigma)})
\]
and
\[
\|B_{\Sigma}\|_{H^{2k+1}(\Sigma)} \leq C_k(1+ \|\tilde \nabla \Delta^k  H_\Sigma \|_{L^2(\Sigma)}).
\]
\end{proposition}

\subsection{Uniform ball condition and signed distance function}
In this subsection we recall some properties related to sets which satisfy uniform ball condition as well as properties of signed distance function defined in \eqref{sdf}. Most of them can be found e.g. in  \cite{AD, Bel} while others are more difficult to find. 
 We recall that a set $E \subset \R^{n+1}$ satisfies uniform ball condition with a radius $r \in \R_+$, 
if it simultaneously satisfies the exterior and interior ball condition with  radius $r$ at every boundary point. 
That is, for every 
$x \in \pa E$ there are balls $B_r(x_+)$ and $B_r(x_-)$  such that 
\[
B_r(x_+) \subset \R^{n+1}\setminus E, \quad B_r(x_-) \subset E \quad \text{and} \quad x \in \pa B_r(x_+) \cap \pa B_r(x_-).  
\]
It is well known, for the experts at least, that the uniform ball condition of a set implies  that its boundary is uniformly $C^{1,1}$-regular hypersurface.  We need this property in a quantitative form which states that if $E \subset \R^{n+1}$ satisfies uniform ball condition with radius $r$, then it can be written locally in a cylinder of width $r/2$ as a graph of $C^{1,1}$-function. 
Since this result is not easy to find in the literature we state it and provide a proof here. 
\begin{proposition}
\label{prop:unifball-reg}
Assume $E \subset \R^{n+1}$ satisfies uniform ball condition with radius $r>0$. Then for every point $x \in \pa E$ we may, by rotating the coordinates, write the interior of the set locally as a subgraph of a function $g: B_{r/2}^n(x') \to \R$, i.e., 
\[
\begin{split}
\mathrm{int}(E) \cap \textbf{C}(x',r/2, r) &= \{ (y',y_{n+1})\in  \textbf{C}(x',r/2, r/2) : y_{n+1} < g(y')   \} \ \ \text{and} \\
\pa E \cap \textbf{C}(x',r/2, r) &= \{ (y',g(y')) : y' \in B^n_{r/2}\}.
\end{split}
\]   
The function $g$ is $C^{1,1}$-regular and it holds  for all $y' \in B_{r/2}^n(x')$ and $s \in (0,r/2]$
\begin{gather*}
|g(y')| \leq \frac{|y'-x'|^2}{r+\sqrt{r^2-|y'-x'|^2}}, \quad |\nabla g(y')| \leq \frac{|y'-x'|}{r} \left(1- \left(\frac{|y'-x'|}{r}\right)^2\right)^{-\frac12} \quad{and} \\
 \sup_{\substack{y_1',y_2' \in B_s^n(x') \\ y_1' \neq y_2'}} \frac{|\nabla g(y'_2)- \nabla g(y'_1)|}{|y_2'-y_1'|} \leq \frac{1}{r}\left(1- \left(\frac{s}{r}\right)^2\right)^{-\frac32}.
\end{gather*}
Moreover, the outer unit normal $\nu_E$ is $1/r$-Lipschitz continuous in Euclidean metric.
\end{proposition}  

\begin{remark} \label{rem:C11-UCB}
We remark, that the converse of Proposition \ref{prop:unifball-reg} also holds true. That is, if $E \subset \R^{n+1}$ is a set such that for every $x \in \pa E$, we may write 
its boundary locally, by rotating and translating the coordinates, as $\pa E \cap \mathbf C(x,r,2r) \subset \{(x',g(x')): x' \in B_r^n \}$ with $\|g\|_{C^{1,1}(B_r^n)} \leq C/r$, then
$E$ satisfies  uniform ball condition with  radius $c \, r$, for a constant $c>0$ which  depends on $n$ and $C$. This is fairly straightforward to show  and we 
leave it to the reader.
\end{remark}

\begin{proof}[Proof of Proposition \ref{prop:unifball-reg}]
We remark that  the uniform ball condition implies that  for every $x \in \pa E$ there exists a unique unit vector $\nu_E(x)$
such that $B_r(x-r\nu_E(x)) \subset E$ and $B_r(x+r\nu_E(x)) \subset \R^{n+1} \setminus E$. Therefore, we have a vector field
$\nu_E : \pa E \rightarrow \pa B_1$ which later turns out to be the outer unit normal field of $E$. 
We first show that $\nu_E$ is $1/r$-Lipschitz continuous with respect to Euclidean distance. To this end, fix $x,y \in \pa E$.
By the previous remark $B_r(x + r\nu_E(x)) \subset \R^{n+1} \setminus E$ and $B_r(y-r\nu_E(x)) \subset E$ so the balls
are disjoint. Similarly, the balls $B_r(x - r\nu_E(x))$ and $B_r(y + r\nu_E(y))$ are disjoint.  Hence the distances between the corresponding centerpoints are at least $2r$ and we obtain the inequalities
\begin{align*}
4r^2 \leq |x-y + r(\nu_E(x)+\nu_E(y))|^2 \ \ \text{and} \\
4r^2 \leq |x-y - r(\nu_E(x)+\nu_E(y))|^2.
\end{align*} 
By summing the above inequalities gives us
$
8r^2 \leq  2|x-y|^2 + 4r^2\left(1+ \nu_E(x) \cdot \nu_E(y)\right)
$
and, again, by subtracting and dividing terms we further obtain
\beq
\label{eq:unifball-reg-1}
1 - \frac{|x-y|^2}{2r^2} \leq \nu_E(x) \cdot \nu_E(y) \quad \text{or equivalently} \quad |\nu_E(x) - \nu_E(y)|^2 \leq \frac{|x-y|^2}{r^2}.
\eeq
In particular, $\nu_E$ is $1/r$-Lipschitz. 

For given a point $x \in \pa E$, we show the existence of $g$ as claimed. Without loss of generality we may assume $x = 0$ and  $\nu_E(0)= e_{n+1}$. Then it holds $B_r(-re_{n+1}) \subset E$ and $B_r(re_{n+1}) \subset \R^{n+1} \setminus E$. Thus, for every $y' \in B_{r/2}^n$ there is a number $t_{y'}$ such that 
$(y',t_{y' }) \in \pa E$ and 
\beq \label{eq:unifball-reg-2}
|t_{y'}| \leq r - \sqrt{r^2 - |y'|^2} = \frac{|y'|^2}{ r + \sqrt{r^2 - |y'|^2} }.
\eeq
In particular, $|t_{y'}|<|y'|$. Combining \eqref{eq:unifball-reg-1} and \eqref{eq:unifball-reg-2} yields
\beq \label{eq:unifball-reg-3}
 \nu_E(y',t_{y'}) \cdot e_{n+1} \geq \sqrt{1 - \Big(\frac{|y'|}{r}\Big)^2}.
\eeq
Let us show that such a number $t_{y'}$ is unique. 

We suppose by contradiction there is $s_{y'} \in (-r,r) \setminus \{t_{y'}\}$ such that $(y',s_{y'}) \in \pa E$. We may assume $s_{y'} > t_{y'}$. Since $B_r\big( (y',t_{y'}) + r \nu_E(y',t_{y'}) \big)\subset \R^{n+1} \setminus E$  and $(y',s_{y'})  \in \pa E$, then the point $(y',s_{y'}) $ is not in the ball $B_r\big( (y',t_{y'}) + r \nu_E(y',t_{y'}) \big)$. Hence, we obtain 
\[
\begin{split}
r^2 &\leq |(y',s_{y'}) - \big((y',t_{y'}) +r \nu_E(y',t_{y'}) \big)|^2\\
&= (s_{y'} -t_{y'})^2 -2 r(s_{y'} -t_{y'}) \nu_E(y',t_{y'}) \cdot e_{n+1} + r^2.
\end{split}
\]
Therefore, using first the above, then \eqref{eq:unifball-reg-2}, \eqref{eq:unifball-reg-3} and $|y'|<r/2$ we deduce
\[
s_{y'} \geq t_{y'} + 2r \nu_E(x',t_{y'}) \cdot e_{n+1}  \geq -  r + 3 \sqrt{r^2 - |y'|^2} >   r -  \sqrt{r^2 - |y'|^2}.
\]
This implies together with $s_{y'}<r$ that $(y',s_{y'}) \subset B_r(re_{n+1}) \subset \R^{n+1} \setminus E$ which, in turn, contradicts $(y',s_{y'}) \in \pa E$. Thus, $t_{y'}$ is a unique value in $(-r,r)$ satisfying $(y',t_{y'}) \in \pa E$.  
Thus, the function $g: B_{r/2}^n \to \R$, given by the relation $g(y')=t_{y'}$, satisfies
\beq
\label{eq:unifball-reg-4}
\begin{split}
\mathrm{int}(E) \cap \textbf{C}(0,r/2, r/2) &= \{ (y',y_{n+1})\in  \textbf{C}(0,r/2, r/2) : y_{n+1} < g(y')   \} \ \ \text{and} \\
\pa E \cap \textbf{C}(0,r/2, r/2) &= \{ (y',g(y')) : y' \in B^n_{r/2}\}.
\end{split}
\eeq
Again, \eqref{eq:unifball-reg-2} gives us the bound on $|g(y')|$ as claimed. The condition \eqref{eq:unifball-reg-3}
implies that for every $y' \in  B_{r/2}^n(0)$ there are open sets $y' \in V \subset  B^n_r$, $(y',g(y')) \in U \subset  \textbf{C}(0,r/2, r/2)$ and functions
$\psi_+,\psi_- \in C^\infty(V)$ such that $\pa B_r((y' ,g(y' )) \pm r \nu_E(y',g(y')) \cap U$ are the graphs of $\psi_\pm$
respectively. Then $\psi_- \leq g \leq \psi_+$ in $V$ and $\psi_-(w)=g(w)=\psi_+(w)$
implying the differentiability of $g$ at $y'$ with $\nabla g (y') = \nabla \psi_\pm (y')$. Moreover, we deduce that 
$\nu_E(y',g(y'))$ is the outer unit normal of $\{ (z',z_{n+1}) \in V \times \R : z_{n+1} > \psi_+(z')\}$ at $(y',g(y'))$ and thus
\beq
\label{eq:unifball-reg-5}
\nu_E(y',g(y')) = \frac{(-\nabla \psi_+(y'),1)}{\sqrt{1+|\nabla \psi_\pm (y')|^2})} = \frac{(-\nabla g (y'),1)}{\sqrt{1+|\nabla g (y')|^2}}.
\eeq
Since now $g$ and $\nu_E$ are continuous, \eqref{eq:unifball-reg-5} implies that $\nabla g$ is continuous too.
Thus, $E$ is $C^1$-regular and $\nu_E$ is the actual outer unit  normal of $E$.
We combine  \eqref{eq:unifball-reg-3} and  \eqref{eq:unifball-reg-5} to observe
\beq
\label{eq:unifball-reg-6}
|\nabla g (y')| \leq \frac{|y'|}{r} \left(1- \left(\frac{|y'|}{r}\right)^2\right)^{-\frac12}.
\eeq

To conclude the Lipschitz estimate, we fix $s \in (0,r/2]$. If $y'_1,y'_2 \in B_s^n$, then 
the uniform ball condition implies that $(y'_1,g(y'_1)) \notin B_r((y'_2,g(y_2'))\substack{+ \\ (-)} r\nu_E(y'_2,g(y'_2))$
and $(y'_2,g(y'_2)) \notin B_r((y'_1,g(y'_1))\substack{+ \\ (-)} r\nu_E(y'_1,g(y'_1))$. Hence using \eqref{eq:unifball-reg-5} we obtain the estimates
\begin{align*}
r^2 &\leq \Big| (y'_2,g(y'_2))\substack{+ \\ (-)} r\frac{(-\nabla g (y'_2),1)}{\sqrt{1+|\nabla g (y'_2)|^2}}- (y'_1,g(y'_1))\Big|^2 \quad \text{and} \\
r^2 &\leq \Big| (y'_1,g(y'_1))\substack{+ \\ (-)} r\frac{(-\nabla g (y'_1),1)}{\sqrt{1+|\nabla g (y'_1)|^2}}- (y'_2,g(y'_2))\Big|^2.
\end{align*}
By summing these inequalities and simplifying we have 
\[
\substack{+ \\ (-)} (y'_2- y'_1) \cdot (\nabla g(y'_2)- \nabla g(y'_1)) \leq \frac{\sqrt{1+|\nabla g (y'_2)|^2}+\sqrt{1+|\nabla g (y'_1)|^2}}{2r}\left(|y'_2-y'_1|^2 + (g(y'_2)-g(y'_1))^2\right).
\]
Thus, by recalling \eqref{eq:unifball-reg-6} we further estimate
\begin{align}
\notag
|(y'_2- y'_1) \cdot (\nabla g(y'_2)- \nabla g(y'_1)) | 
&\leq  \frac{\sqrt{1+|\nabla g (y'_2)|^2}+\sqrt{1+|\nabla g (y'_1)|^2}}{2r}\left(|y'_2-y'_1|^2 + (g(y'_2)-g(y'_1))^2\right) \\
\label{eq:unifball-reg-7}
&\leq  \frac{\sqrt{1+ \sup_{B^n_s} |\nabla g|^2}}{r}\left(1 + \sup_{B^n_s} |\nabla g|^2\right)|y'_2-y'_1|^2 \\
\notag
&\leq  \frac{1}{r}\left(1- \left(\frac{s}{r}\right)^2\right)^{-\frac32}|y'_2-y'_1|^2.
\end{align}
The desired estimate then follows from \eqref{eq:unifball-reg-7} via a standard mollification argument.
\end{proof}

A signed distance function $d_E$ of non-empty set $E \subset \R^{n+1}$  is always $1$-Lipschitz.
Imposing more regularity on $E$ also improves the regularity of the signed distance function. We begin by observing that uniform ball condition is closely related to differentibility of signed distance function in a tubular neighborhood of the boundary. Indeed, one may show that for a non-empty open set $E \subset \R^{n+1}$ and $r \in \R_+$ the conditions
\begin{itemize}
\item[(i)] $d_E$ is differentiable in $\mathcal N_r(\pa E)$ and
\item[(ii)] $E$ satisfies uniform ball condition with radius $r$
\end{itemize}
 are equivalent. In such a case, the projection $\pi_{\pa E}$ onto $\pa E$ is defined in $\mathcal N_r(\pa E)$ as a continuous 
map and the following fundamental identities hold in $\mathcal N_r(\pa E)$
\beq
 \label{def:proj}
\pi_{\pa E} = \id - d_{E} \nabla d_E \ \ \text{and} \ \ \nabla d_E = \nu_E \circ \pi_{\pa E}.
\eeq
In particular, $d_E \in C^1(\mathcal N_r(\pa E))$. Further, it is fairly simple to conclude, that for every
$t \in (-r,r)$ the sublevel set $E_t = \{ x \in \R^{n+1}: d_E(x) <t\}$ has the level set $\{x \in \R^{n+1} : d_E(x)=t\}$ as the boundary and  satisfies uniform ball condition with radius $r-|t|$. Moreover, it holds
\beq
\label{levelproj}
d_{E_t} = d_E - t \quad \text{and} \quad \pi_{\pa E_t} = \pi_{\pa E} + t \nu_E \circ \pi_{\pa E} \ \ \text{in} \ \ \mathcal N_{r-|t|}(\pa E_t).
\eeq

We may then improve the regularity by showing  $\nabla d_E$ and $\pi_{\pa E}$ are locally Lipschitz continuous in $\mathcal N_r(\pa E)$ and obtain quantitative estimates for the Lipschitz constants in smaller tubes.

\begin{lemma}
\label{lemma:globalLip}
Assume $E \subset \R^{n+1}$ satisfies  uniform ball condition with radius $r>0$. 
Then for every $0<\rho<r$ and  $x,y \in \overline{\mathcal N_\rho(\pa E)}$ it holds
\[
|\pi_{\pa E}(x)-\pi_{\pa E}(y)| \leq \frac{r}{r-\rho} |x-y| \ \ \text{and} \ \ \
|\nabla d_E(x) - \nabla d_E(y)| \leq \frac{1}{r-\rho} |x-y|.
\]
\end{lemma}

\begin{proof} It is enough to prove  the first estimate, since the second estimate follows 
from the first  via Proposition \ref{prop:unifball-reg} and the second identity of \eqref{def:proj}. We first show that  the estimate hold locally, i.e., for every  $x \in \mathcal N_r(\pa E)$
\beq
\label{localLip0}
\mathrm{Lip}\left(\pi_{\pa E},x\right) \leq \frac{r}{r-|d_E(x)|} 
\eeq
To this end, we show that for every $x \in \pa E$ and $y \in B_{r/4}(x)$ it holds
\beq
\label{localLip1}
|\pi_{\pa E}(y) - x|^2 \leq \left(1 + \frac{4}{r-|d_E(y)|} |d_E(y)|\right) |y-x|^2.
\eeq
We may assume that $x=0$, $\nu_E(0)=e_{n+1}$ and $y \notin E$. 
Let $g : B^n_{r/2} \rightarrow \R$ be as in Proposition 
\ref{prop:unifball-reg}. Since $|y|<r/4$, then $y \in \mathbf C(r/2,r/2,0)$ implying $|d_E(y)| \leq |y_{n+1}-g_n(y')|$
and, hence, we make a technical observation
\beq
\label{localLip2}
d^2_E(y) \leq 2 d_E(y)(y_{n+1} - g(y')).
\eeq
Thus, using Proposition 
\ref{prop:unifball-reg}, \eqref{def:proj}, \eqref{localLip2} and Young's inequality we estimate
\begin{align*}
|\pi_{\pa E}(y)|^2
&=|y|^2 -2 d_E (y) \  y \cdot \nabla d_E(y) + d_E^2(y) \\
&=|y|^2 -2 d_E (y) y_{n+1} + d_E^2(y) -2 d_E (y) \  y \cdot (\nabla d_E(y)- e_{n+1}) \\
&\leq |y|^2 - 2 d_E(y) g(y') -2 d_E (y) \  y \cdot (\nu_E (\pi_{\pa E}(y))- \nu_E(0)) \\
&\leq |y|^2  + 2\frac{|d_E(y)|}{r} |y'|^2  + 2\frac{d_E(y)}{r} |y| |\pi_{\pa E}(y)| \\
&\leq  |y|^2 +2\frac{|d_E(y)|}{r} |y|^2 + \frac{|d_E(y)|}{r} |y|^2 + \frac{|d_E(y)|}{r}|\pi_{\pa E}(y)|^2
\end{align*}
and \eqref{localLip1} follows. Suppose next $y_1,y_2 \in B_\rho (x)$ for given $x \in \pa E$ and $0<\rho<r/9$.
The sublevel set $E_t$, for $t=d_E(y_2)$, satisfies  uniform ball condition with radius $r-\rho$ and $y_2 \in \pa E_t$. 
Since $|y_1 - y_2| < 2 \rho \leq (r-\rho)/4$, then by applying \eqref{localLip1} for $\pa E_t$ we have
\beq
\label{localLip3}
|\pi_{\pa E_t}(y_1)-y_2| \leq \left(1 + \frac{8\rho}{r-2\rho}\right)^\frac12 |y_1- y_2|.
\eeq
On the other hand, first recalling the second identity in \eqref{levelproj} and then applying Proposition 
\ref{prop:unifball-reg} gives us
\[
|\pi_{\pa E_t}(y_1)-y_2| = |\pi_{\pa E_t}(y_1)-\pi_{\pa E_t}(y_2)| \geq \left(1- \frac{\rho}{r}\right)|\pi_{\pa E}(y_1)-\pi_{\pa E}(y_2)|
\]
so by combining the estimate above with \eqref{localLip3} yields $\mathrm{Lip}(x, \pi_{\pa E}) = 1$. Hence, we deduce
\beq
\label{localLip4}
\mathrm{Lip}(x, \pi_{\pa E_t}) = 1
\eeq
 for every $t \in (-r,r)$ and $x \in \pa E_t$. By using \eqref{levelproj} and Proposition \ref{prop:unifball-reg} similarly as previous, we infer \eqref{localLip0} from  \eqref{localLip4}. 

Finally, for the first estimate of the claim, we may assume $x,y \in \mathcal N_\rho(\pa E)$. Let $J_{yx}:=\{tx+(1-t)y: t \in [0,1]\}$
be the line segment between them. If $J_{yx} \subset \mathcal N_\rho(\pa E)$, then the first estimate of the claim follows
from   \eqref{localLip0}. Otherwise, there are $0<t_1\leq t_2 < 1$ such that
$tx+(1-t)y \in \mathcal N_\rho(\pa E)$ for every $t \in [0,t_1) \cup (t_2,1]$ and $z_i = t_i x +(1-t_i)y \in \pa \mathcal N_{\rho}(\pa E)$ for $i=1,2$. Since $d_E(z_1)=\rho=d_E(z_2)$, then Proposition \ref{prop:unifball-reg} and \eqref{def:proj} imply
\[
|\pi_{\pa E}(z_1)-\pi_{\pa E}(z_2)| \leq   \frac{r}{r-\rho} |z_1-z_2|.
\]
On the other hand, due to \eqref{localLip0} we have
\[
|\pi_{\pa E}(x)-\pi_{\pa E}(z_1)| \leq  \frac{r}{r-\rho} |x-z_1| \quad \text{and} \quad |\pi_{\pa E}(z_2)-\pi_{\pa E}(y)| \leq  \frac{r}{r-\rho} |z_2-y|
\]
and we conclude the proof.
\end{proof}

If $E$ is $C^{k,\alpha}$-regular, with $k \geq 2$ and $0 \leq \alpha \leq 1$, then $d_E \in C^{k,\alpha}(\mathcal N_r(\pa E))$ and $\pi_{\pa E} \in C^{k-1,\alpha}(\mathcal N_r(\pa E);\R^{n+1})$. In particular, \eqref{def:proj} holds everywhere in
$\mathcal N_r(\pa E)$. Then it holds 
\beq \label{eq:2ndorder}
\nabla^2 d_E = B_E \quad \text{and} \quad \Delta d_E = H_E \ \ \text{on} \ \ \pa E.
\eeq
In particular, we deduce from Lemma \ref{lemma:globalLip} and \eqref{eq:2ndorder} that 
\beq
\label{eq:curva-bound-tri}
\| H_E\|_{L^\infty(\pa E)} \leq \frac{n}{r} \quad \text{and} \quad \sup_{\pa E} |B_E|_\op \leq \frac{1}{r}.  
\eeq
Differentiating  $\nabla d_E \cdot \nabla d_E = 1$ yields $\nabla^2 d_E \nabla d_E = 0$ in $\mathcal N_r(\pa E)$. Again, by differentiating the first identity in \eqref{def:proj} we obtain
\beq 
\label{eq:pro-1}
\nabla \pi_{\pa E} = I - \nabla d_{E} \otimes \nabla d_E  - d_{E} \nabla^2 d_E \ \ \text{in} \ \ \mathcal N_r(\pa E).
\eeq
The second identity in \eqref{def:proj} says that $\nabla d_E = \nabla d_E \circ \pi_{\pa E}$ in $\mathcal N_r(\pa E)$. Thus, by differentiating this and  by using the properties of the distance function mentioned before we have 
\beq \label{eq:dist-22}
\nabla^2 d_E  =(\nabla^2 d_E)^T = \nabla \pi_{\pa E} (\nabla^2 d_E \circ \pi_{\pa E}) = \big( I -  d_{E} \nabla^2 d_E \big) (B_E \circ \pi_{\pa E}) \ \ \text{in} \ \ \mathcal{N}_r(\pa E).
\eeq
We  write this  as 
\[
\nabla^2 d_E \big( I +  d_{E} (B_E \circ \pi_{\pa E}) \big)  =  B_E \circ \pi_{\pa E}.
\]
It follows from  \eqref{eq:curva-bound-tri} that the matrix field $ I + d_{E} ( B_E \circ \pi_{\pa E})$ is invertible in $\mathcal{N}_r(\pa E)$. Therefore, we have 
 \beq \label{eq:dist-3}
\nabla^2 d_E =(B_E \circ \pi_{\pa E}) \big( I +  d_{E}  (B_E \circ \pi_{\pa E}) \big)^{-1} \ \ \text{in} \ \ \mathcal{N}_r(\pa E).
\eeq
By combining \eqref{def:proj}, \eqref{eq:pro-1}, \eqref{eq:2ndorder} and \eqref{eq:dist-3} we may decompose $\nabla \pi_{\pa E}$  as
\beq \label{eq:pro-2}
\nabla \pi_{\pa E} = I - \nu_{E} \circ \pi_{\pa E} \otimes \nu_{E} \circ \pi_{\pa E}  - d_{E} (B_E \circ \pi_{\pa E}) \big( I +  d_{E} 
(B_E \circ \pi_{\pa E}) \big)^{-1} \ \ \text{in} \ \ \mathcal{N}_r(\pa E).
\eeq

By using a fairly standard calibration argument  (see e.g. \cite[Lemma 4.1]{AFM}) we conclude that uniform ball condition implies so called $\Lambda$-\emph{minimizer} condition.
\begin{lemma}
\label{lemma:calibration}
Assume that  $E \subset \R^{n+1}$  is an open and  bounded set which satisfies  uniform ball condition with radius $r>0$. Then
for every set of finite perimeter $F$ it holds
\[
\begin{split}
&P(E \cap F) \leq P(F) + \frac{C_n}{r}|F \setminus E| \quad  \text{and} \\
&P(E \cup F) \leq P(F) + \frac{C_n}{r}|E \setminus F|.
\end{split}
\]
In particular, $P(E) \leq \frac{C_n}{r} |E|$.
\end{lemma}
\begin{proof}
The argument is  a quantitative version of \cite[Lemma 4.1]{AFM}.  We will prove that for every set of finite perimeter $F$ it holds 
\begin{equation}
\label{lem:calibration1}
P(E) \leq P(F) + \frac{C_n}{r}|F \Delta E|. 
\end{equation}
Then the two  inequalities in the statement follow by using \eqref{lem:calibration1} with $E \cup F$ and $E \cap F$ in place of $F$ and  using the fact \cite[Lemma 12.22]{Ma}
\[
P(E\cup F) + P(E \cap F) \leq P(E) + P(F).
\] 
The third inequality follows by using  \eqref{lem:calibration1}  with $F = \emptyset$.

By standard  approximation argument  for the sets of finite perimeter \cite[Thm 13.8 ]{Ma} we may assume that $F$ is smooth. In turn, we may approximate also $E$ by a sequence of smooth sets $E_k$ in the $C^1$-sense such that $ E_k$ satisfies uniform ball condition with radius $r_k$ such that $r_k \rightarrow r$. Therefore by simplicity we assume that also $E$ is smooth.

We construct a vector-field $X \in C_0^1(\R^{n+1}; \R^{n+1})$ such that 
\begin{itemize}
\item[(i)] $X = \nu_E$ on $\pa E$,  
\item[(ii)] $|X| \leq 1$ in $\R^{n+1}$  and 
\item[(iii)] $\|\diver X\|_{L^\infty(\R^{n+1})} \leq C_n/r$.
\end{itemize} 
To this aim, let $\eta \in C^1(\R)$ be a cut-off function such that $0 \leq \eta \leq 1$,  $\eta(0) = 1$, $\eta(t) = 0$ for $|t| \geq r/2$ and $|\eta'| \leq 4/r$. 
We set $X = (\eta \circ d_E) \nabla d_E$. Then  $\spt X \subset \pa E + B_{r/2}(0)$ and
\[
\diver X = (\eta \circ  d_E) \Delta  d_E + \eta' \circ  d_E  \, \nabla d_E \otimes \nabla d_E\ \ \text{in} \ \ \mathcal N_{r/2}(\pa E).
\]
 It follows from Lemma \ref{lemma:globalLip} and $\nabla^2 d_E \nabla d_E = 0$ in $\mathcal N_r(\pa E)$  that $|\Delta d_E| \leq 2n/r$ in $\mathcal N_{r/2}(\pa E)$.
Thus $X$ satisfies the conditions (i)--(iii). 

The inequality \eqref{lem:calibration1} then follows from  divergence theorem as
\begin{align*}
P(E) - P(F) &\leq \int_{\pa E }  X \cdot \nu_E \, \d \Ha^n - \int_{\pa  F}   X \cdot \nu_F \, \d \Ha^n \\
&=  \int_{E} \diver X \, \d x   - \int_{F} \diver X \, \d x \leq  \int_{E \Delta F} |\diver X| \, \d x  \leq \frac{C_n}{r}|F \Delta E|. 
\end{align*}
\end{proof}

Suppose  that $E'$ is a connected component  of a set $E$ which satisfies uniform ball condition with $r$. We claim that then $E'$ is bounded. Indeed, by the above approximation we may 
assume that $E$ is smooth. Then by \eqref{eq:2ndorder} we have $|H_{E}| \leq n /r$ on $\pa E$.
Thus, by using so called Topping's inequality \cite{Top} as well as Lemma  \ref{lemma:calibration}
we have the estimate
 \begin{equation} \label{Topping}
\dia(E') \leq C_n \int_{\pa E'} |H_{E'}|^{n-1} \ \d \Ha^n \leq \frac{C_n}{r^n} |E'|.
\end{equation}

Finally, we need the following interpolation result.  
\begin{lemma}
\label{lemma:simpleinterpolation}
Assume $E \subset \R^{n+1}$   is an open and  bounded set which satisfies  uniform ball condition  with radius $r >0$. If $U$ is an open set containing $\pa E$ and $u \in C^2(U)$, then
\[
\|\nabla_\tau u\|^2_{L^\infty(\pa E)} \leq 4\|u\|_{L^\infty(\pa E)}\left(\sup_{\pa E} |\nabla^2 u|_{\op} + \frac{\|\nabla_\tau u\|_{L^\infty(\pa E)}}{r}\right).
\]
\end{lemma}
\begin{proof} By the above approximation argument we may assume that  $E$ is smooth. 

We first observe that for a bounded function $f \in C^2(\R)$ it holds 
\beq 
\label{eq:simpleinterpolation1}
\|f'\|_{L^\infty(\R)}^2 \leq 4 \|f\|_{L^\infty(\R)}\|f''\|_{L^\infty(\R)}.
\eeq
Indeed, let us fix a $t \in \R_+$. We may assume that $f'(t)> 0$, since otherwise we consider the function $-f$ instead of $f$. Let $I$ be a maximal open interval containing $t$ such that $f'>0$ in $I$ so $f$ is strictly increasing there. Then there is  a decreasing sequence $(\tilde t_i)_i \in (\inf I, t)$
 converging to $\inf I$ such that  $f'(\tilde t_i) \rightarrow 0$ as $i \rightarrow \infty$. Since $f$ is strictly increasing in $I$, it is invertible there.
Hence, we may compute for every $i \in \N$
\begin{align*}
|f'(t)|^2 - |f'(\tilde t_i)|^2  
&= \int_{\tilde t_i}^t \frac{\d}{\d s} (f'(s))^2  \d s = 2 \int_{\tilde t_i}^t f''(s) f'(s)  \d s 
= 2 \int_{\tilde t_i}^t f''(f^{-1}(f(s))) f'(s) \d s \\
&= 2 \int_{f(\tilde t_i)}^{f(t)} f''(f^{-1}(\tau)) \d \tau \leq 4 \|f\|_{L^\infty(\R)}\|f''\|_{L^\infty(\R)}
\end{align*}
and, thus, by letting $i \rightarrow \infty$ we obtain 
$|f'(t)|^2 \leq  4 \|f\|_{L^\infty(\R)}\|f''\|_{L^\infty(\R)}$
and \eqref{eq:simpleinterpolation1} follows.  

Since $\pa E$ is compact we find $x \in \pa E$ such that $|\nabla_\tau u (x)| = \|\nabla_\tau u \|_{L^\infty(\pa E)}$.
We may assume that $|\nabla_\tau u (x)| > 0$. The connected component of $\pa E$ containing $x$ is geodesically complete and, hence, we find a smooth unit speed geodesic curve $\gamma : \R \to \pa E$ satisfying $\gamma(0)=x$ and $\gamma'(0) = \nabla_\tau u (x)/|\nabla_\tau u (x)|$. Then we define a $C^2$-regular function $f = u \circ \gamma$. Note $f'(0)= \|\nabla_\tau u\|_{L^\infty(\pa E)}$ and
\beq
\label{eq:simpleinterpolation2} 
f''=  \gamma' \cdot (\nabla^2 u \circ \gamma) \gamma' + \gamma'' \cdot (\nabla_ \tau  u \circ \gamma) .
\eeq
By differentiating the identity $0 = d_E \circ \gamma$ twice and recalling the identities  \eqref{def:proj} and \eqref{eq:2ndorder} we obtain 
$0 =  \gamma' \cdot (B_E \circ \gamma) \gamma' + \gamma'' \cdot (\nu_E \circ \gamma)$. Since $\gamma$ is a geodesic curve, then $|\gamma'' \cdot (\nu_E \circ \gamma)|= |\gamma''|$ and hence we infer from the previous that $|\gamma''| \leq |B_E \circ \gamma |_\op$. By combing this with \eqref{eq:simpleinterpolation2} and using \eqref{eq:curva-bound-tri} gives us
\[
|f''| \leq \left(|\nabla^2 u \circ \gamma|_\op + |B_E \circ \gamma |_\op |\nabla_\tau u \circ \gamma| \right) 
\leq \left(\sup_{\pa E} |\nabla^2 u|_\op + \frac{\|\nabla_\tau u\|_{L^\infty(\pa E)}}{r} \right).
\]
Thus, by observing $\|f\|_{L^\infty(\R)} \leq \|u\|_{L^\infty(\pa E)}$, the claim follows from \eqref{eq:simpleinterpolation1}.
\end{proof}

\section{Definition of the flat flow and the first regularity estimates}

Let us begin by recalling the definition of the minimizing movements scheme and the flat flow solution of \eqref{eq:VMCF} from \cite{MSS}. Assume that $E_0 \subset \R^{n+1}$ is a bounded  set of finite perimeter.  For given a time step $h \in \R_+$ we construct a parametrized family $(E_t^h)_{t\geq 0}^\infty$ of sets of finite perimeter by an iterative minimizing procedure called minimizing movements, where 
\beq
\label{app flow}
\begin{split}
&E_t^h = E_0 \ \ \text{for every $0\leq t < h$ and} \\
&E^h_t = E_{h \lfloor t/h \rfloor}^h \ \text{is a minimizer of the functional $\mathcal{F}_h( \ \cdot \ , E^h_{t-h})$ for every $t \geq h$}.
\end{split}
\eeq
Here for a generic bounded set of finite perimeter $E \subset \R^{n+1}$ the functional  $\mathcal{F}_h( \ \cdot \ , E)$,  in the class of the bounded
set of finite perimeter, is defined as
\beq
\label{min mov}
\mathcal{F}_h(F, E) = P(F) + \frac{1}{h} \int_F d_{E} \, \d x  + \frac{1}{\sqrt{h}}\big| |F| - m_0\big|,
\eeq
for $m_0 = |E_0|$. We call the family $(E_t^h)_{t\geq 0}^\infty$ defined in \eqref{app flow} \emph{an approximative flat flow
solution} of  \eqref{eq:VMCF}  starting from $E_0$. We note that there is always a minimizer for \eqref{min mov} 
but it might not be unique. By \cite{MSS} we know that there is a subsequence of approximative flat flows $(E^{h_l}_t)_{t \geq 0}$ which converges to a
parametrized family $(E_t)_{t \geq  0}$ for a.e. $t$ in the $L^1$-sense, where for every $t >0$ the set $E_t$ is a set of finite perimeter with $|E_t| = |E_0|$. Any such limit is called a flat flow solution of \eqref{eq:VMCF} starting from $E_0$.   

Let us turn our focus back on a generic minimizer of \eqref{min mov}, where we assume that $|E|=m_0$. 
We then simply denote any minimizer for $\mathcal{F}_h( \ \cdot \ , E)$ by $E^h_{\min}$.
One has to be careful in the definition of the functional in \eqref{min mov}, since the sets of finite perimeter are only defined up to measure zero.  We avoid this issue we recall that, up to modifying a set of finite perimeter in a $L^1(\R^{n+1})$ -negligible
set, its topological boundary agrees with the closure of its measure theoretical boundary. Thus, we always
use the convention $\pa F = \overline{\pa^* F}$ for the initial set and the minimizers. We also remark that if $E$  is empty, then we use the convention $d_E = \infty$ everywhere to ensure that $E^h_{\min}$  is empty too. Next, we recall some basic properties regarding the minimizers. First, it is easy to conclude $P(E^h_{\min}) \leq P(E)$. Moreover, 
$E^h_{\min}$ satisfies the following distance property
\beq
\label{dist0}
\sup_{E^h_{\min} \Delta E} |d_E|  \leq \gamma_n \sqrt h
\eeq
for a dimensional constant $\gamma_n \in \R_+$, see \cite[Prop 3.2]{MSS}. Second, $E^h_{\min}$ has a generalized mean curvature satisfying 
the Euler-Lagrange equation
\beq
\label{euler}
\frac{d_E}{h} = -H_{E_{\min}^h} + \lambda^h
\eeq
 in the distributional sense \eqref{weakcurvature} on $\pa^* E^h_{\min}$,
where the Lagrange multiplier satisfies $|\lambda^h| = 1 / \sqrt h$ in the case $|E^h_{\min}| \neq m_0$, see  \cite[Lemma 3.7]{MSS}. Third, it is easy to see that $E^h_{\min}$ is always a so called $(\Lambda,r_0)$ -\emph{minimizer} with a suitable $\Lambda, r_0 \in \R_+$ satisfying $\Lambda r_0 \leq 1$ (see \cite{Ma} for the definition). Thus, by the standard regularity theory \cite[Thm 26.5 and Thm 28.1]{Ma} the reduced boundary  $\pa^* E^h_{\min}$ is relatively open in $\pa  E^h_{\min}$ and an embedded $C^{1,\alpha}$-regular hypersurface with any 
$0<\alpha <1/2$, and the Hausdorff dimension of the singular part $\pa E^h_{\min} \setminus \pa^* E^h_{\min}$ is at most $n-7$. Thus, by standard Schauder estimates one may show that $\pa^*E^h_{\min}$ is in fact $C^{2,\alpha}$-regular and \eqref{euler} holds
in the classical sense on $\pa^* E^h_{\min}$. Consequently, we may always consider $E^h_{\min}$ as an open set.  

We may improve the distance estimate \eqref{dist0} as well as regularity properties of $E^h_{\min}$, if we impose more regularity on $E$. We divide our approach into two steps. The first result states that if $E$ is  bounded and satisfies uniform ball condition with radius $r_0>0$ and $h$ is sufficiently small, then the left hand side of \eqref{dist0} is bounded linearly in $h$, the Lagrange multiplier $\lambda^h$ is bounded,  the generalized mean curvature $H_{E^h_{\min}}$ is bounded in the $L^\infty$-sense and $E^h_{\min}$ has the volume $m_0$.

\begin{proposition}
\label{prop:distance-bound}
Assume  $E \subset \R^{n+1}$ is an open and bounded set  of volume $m_0$ which satisfies   uniform ball condition with radius $r_0$. 
There are positive numbers $h_0 = h_0(n,m_0,r_0)$ and $C_0=C_0(n,m_0,r_0)$ and a positive dimensional constant $C_n$ such that
if $ h \leq h_0$, then
\[
\sup_{E^h_{\min} \Delta E} |d_E|  \leq \frac{C_n}{r_0} h, \quad \|H_{E_{\min}^h}\|_{L^\infty} + |\lambda^h| \leq C_0  \quad \text{and} \quad |E^h_{\min}| = m_0.
\]
\end{proposition}

\begin{proof}  In the proof, $C$ denotes a generic positive constant which may change its value from the line to line
but it depends only on $n,m_0$ and $r_0$, i.e., $C=C(n,m_0,r_0)$. We fix a number $K_n \in \R_+$ depending only on the dimension such that
$K_n$ exceeds the dimensional constants in Lemma \ref{lemma:calibration} and in \eqref{dist0}. Recall that by Proposition \ref{prop:unifball-reg} $E$ is uniformly $C^{1,1}$-regular and we may assume $E$ to be an open set.
If $|E^h_{\min} \Delta E| = 0$, then it follows from the openness of $E^h_{\min}$ and $E$ as well as the
property $\pa E^h_{\min} = \overline{\pa^* E^h_{\min}}$ and  $\pa E= \overline{\pa^* E}$ that $E^h_{\min} \Delta E = \emptyset$ and there is nothing to prove. Thus, we may assume that $|E^h_{\min} \Delta E|>0$ and further set
\[
d_+ = \sup_{E^h_{\min} \Delta E} d_E \ \ \ \text{and} \ \ \ d_- = \inf_{E^h_{\min} \Delta E} d_E.
\]
To conclude the first estimate, we show that
if $\sqrt h \leq r_0 / (8 K_n)$, then
\beq
\label{prop:distance-bound1}
d_-<0<d_+ \  \ \text{and} \ \ d_+ - d_-\leq \frac{C_n}{r_0} h.
\eeq
We suppose by contradiction that $d_-\geq 0$ which implies $E \subset E^h_{\min}$ due to the openness of $E$. Since $E$ satisfies uniform ball condition
with radius $r_0$ and $E \subset E^h_{\min}$, then by Lemma \ref{lemma:calibration}
\[
P(E) \leq P(E^h_{\min}) + \frac{C_n}{r_0} |E^h_{\min} \setminus E|.
\] 
Again, $| E^h_{\min}\setminus E| = |E^h_{\min}\Delta E| >0$ so the previous estimate together with the assumption $\sqrt h \leq r_0 / (8 K_n)$ and the choice of $K_n$ implies 
\[
\mathcal{F}_h(E,E) < \mathcal{F}_h(E^h_{\min},E) +\left(\frac{C_n}{r_0} - \frac1{\sqrt h}\right)|E^h_{\min} \setminus E|
\leq \mathcal{F}_h(E^h_{\min},E)
\]
contradicting the minimality of $E^h_{\min}$.
Thus, we conclude $d_- < 0$. By using a similar argument and recalling  $\pa E= \overline{\pa^* E}$ we also have that $d_+ > 0$. 

On the other hand, $\sqrt h \leq r_0 / (8 K_n)$ implies via \eqref{dist0} that $E^h_{\min} \Delta E \subset \subset \mathcal N_{r_0/4}(\pa E)$
so $-r_0/2 < d_-<0<d_+<r_0/2$. Then for every $t \in (d_-,d_+)$ 
the sublevel set $E_t = \{ x : d_E(x) < t\}$ satisfies uniform ball condition with $r_0/2$ and
$|E^h_{\min} \setminus E_t|, |E_t \setminus E^h_{\min}|>0$. By using a suitable
continuity argument, we infer from the previous that for every $r_+ < d_+$, sufficiently close to $d_+$, there is
$r_-\in (d_-, r_+)$ such that $|E^h_{\min} \setminus E_{r_+}|=|E_{r_-} \setminus E^h_{\min}| > 0$ and
$r_- \rightarrow d_-$ as $r_+ \rightarrow d_+$.  For such a pair $(r_+,r_-)$ we set
\[
\tilde E (r_\pm,h) = (E_{r_+} \cap E^h_{\min}) \cup  E_{r_-}.
\]
Clearly, $\tilde E (r_\pm,h)$ is a bounded set of finite perimeter and
$
|\tilde E (r_\pm,h)| = |E^h_{\min}|.
$
Thus, using $\tilde E (r_\pm,h)$ as a competitor against  $E^h_{\min}$ with respect to $\mathcal F_h( \ \cdot \ , E)$ we obtain
\begin{align}
\notag
P(E^h_{\min}) 
&\leq P(\tilde E (r_\pm,h)) + \frac{1}{h}\int_{E_{r_-} \setminus E^h_{\min}} d_E \, \d x
- \frac{1}{h}\int_{E^h_{\min} \setminus E_{r_+}} d_E \, \d x \\
\label{prop:distance-bound2}
&\leq P(\tilde E (r_\pm,h))
+ \frac{r_-}{h} |E_{r_-} \setminus E^h_{\min}| -  \frac{r_+}{h}|E^h_{\min} \setminus E_{r_+}|\\
\notag
&\leq P(\tilde E (r_\pm,h))
+ \frac{r_- - r_+}{h} |E_{r_-} \setminus E^h_{\min}|.
\end{align}
Applying Lemma \ref{lemma:calibration} to $E_{r_+}$ and $E_{r_-}$ gives us
\begin{align}
 P(\tilde E (r_\pm,h)) 
\notag
&=   P((E_{r_+} \cap E^h_{\min})\cup E_{r_-}) \\ 
\notag
&\leq P(E_{r_+} \cap E^h_{\min}) + \frac{C_n}{r_0/2}|E_{r_-} \setminus E^h_{\min}| \\
\label{prop:distance-bound3}
&\leq  P( E^h_{\min}) + \frac{C_n}{r_0/2}|E^h_{\min} \setminus E_{r_+}| +  \frac{C_n}{r_0/2}|E_{r_-} \setminus E^h_{\min}| \\
\notag
& =  P( E^h_{\min}) + \frac{C_n}{r_0}|E_{r_-} \setminus E^h_{\min}| 
\end{align}
We combine \eqref{prop:distance-bound2} and \eqref{prop:distance-bound3} and recall 
$|E_{r_-} \setminus E^h_{\min}|>0$ to observe 
\[
\frac{r_+ - r_-}{h} \leq \frac{C_n}{r_0}.
\]
Thus, by letting $r_+ \rightarrow d_+$, we obtain the second estimate in \eqref{prop:distance-bound1}.

Let us next bound the Lagrange multiplier. The argument is standard but we include it for the sake of completeness. 
We assume that $\sqrt h \leq r_0 / (8 K_n)$ and fix any connected component $E^i$ of $E$. By Lemma \ref{lemma:calibration} and \eqref{Topping}  we know that $\dia(E^i) \leq  C$ and $P(E) \leq  C$.  Then we also have $P(E^h_{\min}) \leq P(E) \leq C$. 
If $E^j$ is a connected component of $E$ distinct to $E^i$, then  uniform ball condition guarantees $\dist(E^i,E^j) \geq r_0$.
On the other hand, we have $E^h_{\min} \Delta E \subset \subset \mathcal N_{r_0/4}(\pa E)$. Thus, we infer from the previous observations that for the intersection 
$\tilde E^i = E^h_{\min} \cap (E^i + B_{r_0/4})$ it holds $\pa^* \tilde E^i = \pa^* E^h_{\min} \cap (E^i + B_{r_0/4})$,
$H_{\tilde E^i} = H_{E^h_{\min}}|_{\pa^* \tilde E^i}$, $\dia(\tilde E^i) \leq C + r_0/2 \leq C$ and $|\tilde E^i|\geq |B_{r_0/2}|$.
By translating the coordinates, we may also assume $0 \in \tilde E^i$. Therefore, using the divergence theorems and the Euler-Lagrange equation \eqref{euler}, which holds in the sense of \eqref{weakcurvature} on
$\pa^* \tilde E^i$, we compute
\[
\begin{split}
\lambda^h (n+1) |\tilde E^i| =\int_{\pa^* E^i}  \lambda^h   ( \id \cdot \nu_{\tilde E^i}) \, \d \Ha^n &= \int_{\pa^* \tilde E^i} \left(H_{\tilde E^i} + \frac{d_E}{h} \right)(\id \cdot \nu_{\tilde E^i}) \, \d \Ha^n \\
&= n P(\tilde E^i) + \int_{\pa^* \tilde E^i} \frac{d_{E}}{h}  \, (\id \cdot \nu_{\tilde E^i}) \, \d \Ha^n.
\end{split}
\] 
Hence, recalling the first inequality, the bounds on $P(E^h_{\min})$ and $\dia(\tilde E^i)$ and the lower bound for $|\tilde E^i|$
we find $C_0=C_0(n,m_0,r_0)$ such that $|\lambda^h| \leq C_0$. Therefore using  the Euler-Lagrange equation \eqref{euler}
and  the first estimate again we have, by possibly increasing $C_0$, that $\|H_{E_{\min}^h}\|_{L^\infty\left(\pa^*E_{\min}^h\right)} + |\lambda^h| \leq C_0$. Finally, if $|E^h_{\min}| \neq m_0$, then $|\lambda^h| = 1 /\sqrt{h}$. Thus, assuming $h \leq (2 C_0)^{-2}$
excludes this possibility and hence it must hold $|E^h_{\min}| = m_0$. 
\end{proof}

Proposition \ref{prop:distance-bound}, allows us, via Allard's regularity theorem, to deduce 
that the singular set of minimizer is in fact empty. Further, a standard Schauder estimate gives
us a quantitative, albeit non-sharp,  uniform ball condition for a minimizer. 
 
\begin{lemma}
\label{lemma:improved-regularity}
Assume  $E \subset \R^{n+1}$ is an open and bounded set  of volume $m_0$ which satisfies   uniform ball condition with radius $r_0$. There are positive numbers $h_0 = h_0(n,m_0,r_0)$ and $c_0=c_0(n,m_0,r_0)$ such that
if $ h \leq h_0$, then $\pa E \setminus \pa E^* = \varnothing$, $E^h_{\min}$ is  $C^{3,\alpha}$-regular with any $0<\alpha < 1$ and $E^h_{\min}$ satisfies uniform ball
condition with radius $c_0 h^{1/3}$. In particular, \eqref{euler} is satisfied in the classical sense on $\pa E^h_{\min}$. 
Moreover, if in addition $E$ is  $C^k$-regular, with  $k \geq 2$, then 
$E^h_{\min}$ is  $C^{k+2}$-regular.
\end{lemma}

\begin{proof} We divide the proof into two steps. Recall that we  may assume $E^h_{\min}$ to be open. In the proof, $C$ denotes a generic positive constant which may change its value from the line to line but it depends only on $n,m_0$ and $r_0$.   
\\
\\
\textbf{Step 1:} By using Allard's regularity theorem we show that
the topological boundary  $\pa E^h_{\min}$ agrees with the  reduced boundary $\pa^*E^h_{\min}$ when $h$ is sufficiently small.
To be more precise, we show that there exist positive numbers   $\rho=\rho(n,m_0,r_0)$   and $h_1 = h_1(n,m_0,r_0, \rho)$ such that
if $h \leq h_1$ and $x \in \pa E^h_{\min}$, then, by possibly rotating the coordinates, there is a function
$f \in C^{1,1/3} \left(B^n_{\rho}(x')\right)$ such that
\beq
\label{lemma:improved-regularity1}
 \mathbf C(\rho,2\rho,x) \cap E^h_{\min} = \{y \in  \mathbf C(\rho,2\rho,x)  : y_{n+1} < f(y)\}
\eeq
and $f$ satisfies the estimates
\beq
\label{lemma:improved-regularity2}
\|\nabla f\|_{L^\infty(B^n_{\rho}(x'))} \leq 1 \quad \text{and} \quad \|\nabla f\|_{C^{0,\frac13}(B^n_{\rho}(x'))} \leq C.
\eeq
In particular, \eqref{lemma:improved-regularity1} implies that $\pa^*E = \pa E$ and hence, by our earlier discussion,
we conclude that $E^h_{\min}$ is $C^{2,\alpha}$-regular with any $0<\alpha<1/2$. We may assume that $h_1$ is
chosen so small that via Proposition \ref{prop:distance-bound} the boundary $\pa E^h_{\min}$ is contained in $\mathcal N_{r_0/2}(\pa E)$. Since $d_E \in C^{1,1}(\mathcal N_{r_0/2}(\pa E))$, then recalling the Euler-Lagrange equation \eqref{euler} we may write the generalized mean curvature of $E^h_{\min}$ as a restriction of a $C^{1,1}$-function to $\pa E^h_{\min}$. Therefore, by using standard Schauder estimates,
one may show that $E^h_{\min}$ is actually $C^{3,\alpha}$-regular with any $0<\alpha<1$. Also, the same method gives
us $C^{k+2,\alpha}$-regularity for any $k \geq 2$, if $E$ is already known to be $C^{k,\alpha}$-regular. This is well-known procedure and
we leave it to the reader.  

The claim of Step 1 follows essentially from \cite[Thm 2.5.2]{Sim}, if we prove that for every $x \in \pa E^h_{\min}$ and $\eps \in \R_+$ there are positive numbers 
 $\rho=\rho(n,m_0,r_0,\eps)$   and  $\tilde h = \tilde h(n,m_0,r_0,\rho,\eps)$  such that if $h \leq \tilde h$, then
\begin{align}
\label{lemma:improved-regularity3}
\frac{\Ha^n(B_{\rho}(x) \cap \pa^* E^h_{\min})}{|B^n_{\rho}|} &\leq 1+ \eps \ \  \text{and}  \\
\label{lemma:improved-regularity4}
{\rho}^{\frac13}\left( \int_{B_{\rho}(x) \cap \pa^* E^h_{\min}} |H_{E^h_{\min}}|^{\frac{3n}{2}} \ \d \Ha^n \right)^{\frac{2}{3n}}
&\leq \eps.
\end{align}

We fix $\eps >0$ and initially assume $h \leq h_0$, where $h_0$ is from Proposition \ref{prop:distance-bound}. 
It follows from  Proposition \ref{prop:distance-bound} and the fact $\pa E^h_{\min}=\overline{\pa^* E^h_{\min}}$ that
\beq
\label{lemma:improved-regularity5}
(\overline {E^h_{\min}} \cup \overline E) \setminus (E^h_{\min} \cap E) \subset \mathcal N_{Ch}(\pa E).
\eeq
Thus, we may assume that $(\overline {E^h_{\min}} \cup \overline E) \setminus (E^h_{\min} \cap E) \subset \mathcal N_{r_0/2}(\pa E)$ and thus the projection $\pi_{\pa E}$ is well-defined there. Proposition  \ref{prop:distance-bound} also gives us $|E^h_{\min}|=m_0$.
Next, we fix $x \in \pa E^h_{\min}$. Without loss of generality, we may assume $\pi_{\pa E}(x)=0$ and $\nu_E(0)=e_{n+1}$. 
Then it follows from Proposition \ref{prop:unifball-reg} that there is $g \in C^{1,1}(B^n_{r_0/2})$ such that 
$|g(y')| < |y'|^2/r_0$,  $|\nabla g(y')|< 2|y'|/r_0$
 for every $y' \in B^n_{r_0/2}$ and
\[
\mathbf C (0,r_0/2,r_0/2) \cap E = \{ y \in \mathbf C (r_0/2,r_0/2,0) : y_{n+1} < g(y')\}.
\]
Then for $0<\rho<r_0/4$ we have the density bound
\beq
\label{lemma:improved-regularity6}
P(E; \mathbf C (0,\rho ,r_0/2)) = \int_{B^n_\rho} \sqrt{1+|\nabla g|^2} \, \d y' \leq (1+C\rho^2) |B^n_\rho|.
\eeq

Suppose that $y \in \mathbf C (0,\rho ,r_0/2) \cap ((\overline {E^h_{\min}} \cup \overline E) \setminus (E^h_{\min} \cap E))$ for 
$0<\rho<r_0/4$. Recalling \eqref{lemma:improved-regularity5}, we may assume that $\pi_{\pa E} (y) \in \mathbf C (0, r_0/2,r_0/2)$ and since $|\nabla g| \leq C$ in $B^n_{r_0/2}$ 
we estimate
\begin{align*}
|y_{n+1} - g(y')| 
&\leq |y-\pi_{\pa E} (y)| + |\pi_{\pa E} (y)-(y',g(y'))| \\
&\leq |y-\pi_{\pa E} (y)| + C |(\pi_{\pa E} (y))'-y'| \leq Ch.
\end{align*}
It follows then from Fubini's theorem 
\begin{align}
\label{lemma:improved-regularity7}
\left|\mathbf C (0,\rho ,r_0/2) \cap \left(\overline {E^h_{\min}} \cup \overline E \setminus (E^h_{\min} \cap E)\right)\right| &\leq C \rho^n  h \ \ \text{and} \\
\label{lemma:improved-regularity8}
\Ha^n\left(\pa \mathbf C (0,\rho ,r_0/2) \cap \left(\overline {E^h_{\min}} \cup \overline E \setminus (E^h_{\min} \cap E)\right)\right) &\leq C \rho^{n-1} h
\end{align}
for $0<\rho<r_0/4$. We define for such $\rho$ a comparison set $F_\rho$ by setting
\[
F_\rho = (E^h_{\min} \setminus  \mathbf C (0,\rho,r_0/2)) \cup (E \cap  \mathbf C (0,\rho,r_0/2))
\]
and make the following technical observations.
First, since $E^h_{\min} \cap E$ is open and contained in $F_\rho$, then $\Ha^n(\pa^* F_\rho \cap (E^h_{\min} \cap E))=0$.
Second, $\pa^* F_\rho \subset \overline {E^h_{\min}} \cup \overline E$.
With help of these, \eqref{lemma:improved-regularity6} and  \eqref{lemma:improved-regularity8} we estimate   
\begin{align*}
P(F_\rho )
&= P(F_\rho;\mathbf C (0,\rho ,r_0/2)) +  P(F_\rho; \pa \mathbf C (0,\rho ,r_0/2)) +  P(F_\rho; \R^{n+1} \setminus \overline{\mathbf C (0,\rho ,r_0/2)}) \\
&= P(E;\mathbf C(0,\rho ,r_0/2)) +  \Ha^n(\pa^* F_\rho \cap \pa \mathbf C (0,\rho ,r_0/2)) +  P(E^h_{\min}; \R^{n+1} \setminus \overline{\mathbf C(0,\rho ,r_0/2)}) \\
&\leq P(E;\mathbf C(0,\rho ,r_0/2)) +  P(E^h_{\min}, \R^n \setminus \overline{\mathbf C(0,\rho ,r_0/2)}) \\
&+ \Ha^n\left(\pa \mathbf C(0,\rho ,r_0/2) \cap \left(\overline {E^h_{\min}} \cup \overline E \setminus (E^h_{\min} \cap E)\right)\right) \\
&\leq (1 +C \rho^2)|B^n_\rho|+  P(E^h_{\min}; \R^{n+1} \setminus \overline{\mathbf C(0,\rho ,r_0/2)}) + C \rho^{n-1}h.
\end{align*}
Thus, the inequality $\mathcal F_h(E^h_{\min},E) \leq \mathcal F_h(F_\rho,E)$, \eqref{lemma:improved-regularity7}, $|E^h_{\min}|=m_0$ and the definition of $F_\rho$ yield
\begin{align*}
P(E^h_{\min};\mathbf C(0,\rho ,r_0/2)) 
&\leq (1 +C \rho^2) |B^n_\rho|_n + \frac1h \int_{\mathbf C(0,\rho ,r_0/2)) \cap (E^h_{\min} \Delta E)} |d_E| \, \d x + \frac{1}{\sqrt h} ||F_\rho|-m_0|
+ C \rho^{n-1}h  \\
&\leq  (1 + C\rho^2) |B^n_\rho| + C\left(1+\frac{1}{\sqrt h}\right)|\mathbf C (0,\rho ,r_0/2)\cap (E^h_{\min} \Delta E)|
+ C\rho^{n-1}h  \\
&\leq  (1 + C\rho^2)|B^n_\rho| + C(\rho^n \sqrt h + \rho^{n-1}h).
\end{align*}

Recall that for the fixed point  $x \in \pa E^h_{\min}$ it holds  $x=d_E(x)e_{n+1}$ with $|d_E(x)| \leq Ch$. Thus we may assume $B_\rho (x) \subset  \mathbf C(0,\rho ,r_0/2)$
for  $0<\rho<r_0/4$. Hence, the above  estimate yields
\beq
\label{lemma:improved-regularity9}
P(E^h_{\min};B_\rho (x)) \leq (1 + C\rho^2)|B^n_\rho| + C(\rho^n \sqrt h + \rho^{n-1}h).
\eeq
Moreover, it holds  $\|H_{E^h_{\min}}\|_{L^\infty(\pa^* E_{\min}^h)} \leq C$ by Proposition \ref{prop:distance-bound}, $P(E^h_{\min}) \leq P(E)$ and $P(E) \leq C$ by Lemma \ref{lemma:calibration} and therefore
\[
{\rho}^{\frac13}\left( \int_{B_{\rho}(x) \cap \pa^* E^h_{\min}} |H_{E^h_{\min}}|^{\frac{3n}{2}} \ \d \Ha^n \right)^{\frac{2}{3n}} \leq C{\rho}^{\frac13}.
\]
Hence,  we infer from the previous estimate and \eqref{lemma:improved-regularity9} the existence of numbers $\tilde h$ and $ \rho$ satisfying 
\eqref{lemma:improved-regularity3} and \eqref{lemma:improved-regularity4}.
\newline
\newline
\textbf{Step 2:} We assume that $h \leq h_1$ and fix $x \in \pa E^h_{\min}$. We may assume that $x = 0$ and $\nu_{E^h_{\min}}(0) = e_{n+1}$. 
According to Step 1, up to a possible rotation of the coordinates, there is $f \in C^3(B^n_{\rho_1}(x'))$ with $f(0) = \nabla f(0)= 0$ satisfying \eqref{lemma:improved-regularity1}
and \eqref{lemma:improved-regularity2}. 
We use  Schauder estimate in a quantitative manner to prove there is a positive $h_0=h_0(n,m_0,r_0) \leq h_1$ such that $h \leq h_0$ implies
\beq
\label{lemma:improved-regularity10}
\|\nabla^2 f\|_{L^\infty( B^n_{\rho/2})} \leq C h^{-\frac13}.
\eeq
Once we have proven \eqref{lemma:improved-regularity10} then the claim  that $E^h_{\min}$  satisfies uniform ball condition with radius $c_0 h^{1/3}$ follows 
in a straightforward manner as we discussed in Remark \ref{rem:C11-UCB}.


Thus, we are left to prove \eqref{lemma:improved-regularity10}. We may write
$H_{E^h_{\min}}$  in  local coordinates as the mean curvature of the subgraph $\{(y',y_{n+1}: y' \in B^n_{\rho}, \ y_{n+1} < f(y')\}$, that is,
\beq
\label{lemma:improved-regularity11}
H_{E^h_{\min}} (y',f(y')) = - \diver \left(\frac{\nabla f}{\sqrt{1+|\nabla f|^2}}\right) (y') = -\Tr\left(\mathcal A (y') \nabla^2 f (y')\right).
\eeq
It follows from $\eqref{lemma:improved-regularity2}$ that $\mathcal A$ is uniformly elliptic and bounded in the $C^{0,1/3}$-sense. 
To be more precise, we have 
 \[
 \inf_{y' \in B^n_{\rho}} \min_{\xi \in \pa B^n_1} \mathcal A (y') \xi \cdot \xi \geq 1/C \ \ \text{and} \ \ \max_{ij} \|[\mathcal A]_{ij} \|_{C^{0,\frac13}( B^n_{\rho} )} \leq C.
\]
Thus, by using standard Schauder interior estimate \cite{GT}, \eqref{lemma:improved-regularity2} and \eqref{lemma:improved-regularity11}, we obtain
 \begin{align}
\notag
\|\nabla^2 f\|_{C^{0,\frac13}( B^n_{\rho/2} )} &\leq C\left(\|u\|_{C^{0,\frac13}( B^n_{\rho} )} 
+ \|f\|_{L^\infty(B^n_{\rho})}\right) \\
\label{lemma:improved-regularity12}
 &\leq C\left(\|u\|_{C^{0,\frac13}( B^n_{\rho} )} +1\right),
\end{align}
where $u :  B^n_{\rho_1} \rightarrow \R^n$ is given by $u(y') =H_{E^h_{\min}} (y',f(y'))$. We may
assume $h$ is chosen sufficiently small so that via Proposition \ref{prop:distance-bound} we have $\|u\|_{L^\infty( B^n_{\rho} )} \leq C$. 
Again, \eqref{lemma:improved-regularity2} implies $|\nabla u (y')| \leq C |\nabla_\tau H_{E^h_{\min}} (y',f(y'))|$ 
for every $y' \in B^n_{\rho}$. On the other 
hand, by (tangentially) differentiating the Euler-Lagrange equality \eqref{euler} we obtain $|\nabla_\tau H_{E^h_{\min}} (y',f(y'))| \leq 1/h$ for every $y' \in B^n_{\rho}$. Hence, $\|\nabla u\|_{L^\infty(B^n_{\rho})} \leq C/h$ 
and since $\|u\|_{L^\infty( B^n_{\rho} )} \leq C$, assuming $h \leq 1$ yields $\|u\|_{C^1( B^n_{\rho})} \leq C/h$. Again, Lemma  \ref{lem:inter-holder} yields
$\|u\|_{C^{0,1/3}( B^n_{\rho} )} \leq C h^{-1/3}$ and hence, by recalling \eqref{lemma:improved-regularity12},
we conclude the existence of $h_0 = h_0(n,m_0,r_0)$ satisfying \eqref{lemma:improved-regularity10} for all $h \leq h_0$.
\end{proof}

\begin{remark}
\label{rem:improved-regularity}
We may replace the exponent $1/3$ with a generic $0<\alpha<1$ in the proof of Lemma \ref{lemma:improved-regularity}.
Then, naturally, $h_0$ and $c_0$ also depend on $\alpha$.
The uniform ball conditions with radius $r_0$ for $E$ and with radius $c_0 h^{1/3}$ for $E^h_{\min}$ imply together with the distance estimate of Proposition \ref{prop:distance-bound} and \eqref{def:proj} that 
there is $h_0 = h_0(n,m_0,r_0)$ such that if $h \leq h_0$, then $\nabla d_E \cdot \nu_{E^h_{\min}}> 0$ on $\pa E^h_{\min}$ and the projection $\pi_{\pa E}$ is injective on $\pa E^h_{\min}$.
\end{remark}

\section{Uniform ball condition for short-time}

In this section, we adopt the two-point function method to prove that if the initial set 
$E_0$ satisfies uniform ball condition with radius $r_0$, then there are positive numbers $h_0$ and $T_0$
such that 
\beq
\label{eq:rt}
h \leq h_0 \implies  E^h_t \  \text{satisfies uniform ball condition with radius $r_0/2$ for $0\leq t \leq T_0$},
\eeq
where the approximative flow $(E_t^h)_{t \geq 0}$ starting from $E_0$ is defined as in \eqref{app flow}.
 For more precise statement see Theorem \ref{thm2} 
at the end of the section. As we have seen in Lemma
\ref{lemma:improved-regularity}, uniform ball condition for an initial set is crucial as it guarantees that the corresponding minimizer of the energy \eqref{min mov} has improved regularity and an initial quantitative bound on the uniform ball condition
although the latter is highly dependent of $h$. In  this section, we improve the previous non-sharp estimate on the uniform ball condition for the minimizer by showing the minimizer satisfies almost the same uniform ball condition as the initial set.

The original idea of the two-point function  goes back  to \cite{Hui}, where it is used to study the regularity of the classical solution to the mean curvature flow. We refer to \cite{Bre} for a comprehensive overview of the topic and mention also the works \cite{And, Bre2, DG} which have inspired us.  Here we will show that the method can be  applied to the approximative flat flow at the level of discrete time scale. We will assume that the approximative flat flow is related to the volume preserving mean curvature flow but the arguments hold with essentially no modifications also in the case of the mean curvature flow. 

\subsection{Two-point function method}
The main idea is to double the variables and, given a set $E \subset \R^{n+1}$ satisfying uniform ball condition, to study the function $S_E$ defined for $(x,y) \in \pa E \times \pa E$ with $x \neq y$ as       
\beq \label{def:S-E}
S_E(x,y) := \frac{(x-y) \cdot \nu_E(x)}{|x-y|^2}.
\eeq
It is known, but we will include the proof below, that the maximum value of  $|S_E|$ is explicitly related to the  uniform ball condition. In other words, doubling the variables allows us to quantify the   uniform ball condition via the function $S_E$. It is interesting that the idea of  doubling the variables is also used  in \cite{IL} to study  regularity of solutions of nonlinear PDEs.

For the next lemma we note that if a set $E$ satisfies uniform ball condition with radius $r$, then it satisfies it also for every $\rho < r$. We define $r_E$ to be the supremum of such radii  and recalling our previous discussion we may write this as
\beq \label{def:r-E}
r_E = \sup \{ r >0 : d_E \ \ \text{is differentiable in} \ \ \mathcal{N}_r(\pa E)\}. 
\eeq
Note that  $r_E>0$. We use the abbreviation $\|S_E\|_{L^\infty} := \sup\{ |S_E(x,y)| : x,y \in \pa E, \ x \neq y \}$.
\begin{lemma}
\label{lem:S-E}
Let $E \subset \R^{n+1}$ be an open and  bounded set satisfying uniform ball condition. Then it holds 
\[
2\| S_E\|_{L^\infty} = \frac{1}{r_E} \quad \text{and} \quad \frac{|\nu(x) - \nu(y)|}{|x-y|} \leq  2 \| S_E\|_{L^\infty}  \quad \text{for every $x,y \in \pa E$ with $x\neq y$}.
\] 
where $r_E$ is defined in \eqref{def:r-E}. In the case $E$ is $C^2$-regular, we also have 
$|H_E|, |B_E| \leq 2n \| S_E\|_{L^\infty}$ on $\pa E$.
\end{lemma}

\begin{proof}
Let us first show $2\| S_E\|_{L^\infty} \geq 1/r_E$. We infer from the boundedness of $E$ that $r_E < \infty$ and, hence, it follows from the definition of $r_E$ that there is a sequence of points $z_i$ such that $d_E(z_i)$ is not differentiable at $z_i$ and $|d_E(z_i)| \to r_E$. Since the signed distance function is not differentiable at $z_i$, then there are two distinct points $x_i, y_i \in \pa B_{|d_E(z_i)|}(z_i) \cap  \pa E$. Since the intersection $ B_{|d_E(z_i)|}(z_i) \cap \pa E$ is empty it holds 
\[
\nu_E(x_i) = \pm \frac{x_i - z_i}{|x_i - z_i|}.
\]
Therefore, recalling also that $|z_i -y_i|= |x_i - z_i| = |d_{E}(z_i)|$ we have
\beq \label{eq:S-E-1}
\begin{split}
|S_E(x_i,y_i)| = \frac{|(x_i-y_i) \cdot \nu_E(x_i)|}{|x_i-y_i|^2} &= \frac{|(x_i-y_i) \cdot (x_i-z_i)|}{|d_{E}(z_i)| \, |x_i-y_i|^2}\\
&= \frac{\big|\big(x_i -z_i)-(z_i-y_i)\big) \cdot (x_i-z_i)\big|}{|d_{E}(z_i)| \, |x_i-y_i|^2}\\
&= \frac{|x_i -z_i|^2 -(z_i-y_i)\cdot (x_i-z_i)}{|d_{E}(z_i)| \, |x_i-y_i|^2}\\
&= \frac{|x_i -z_i|^2 -2(z_i-y_i)\cdot (x_i-z_i) + |z_i -y_i|^2}{2|d_{E}(z_i)| \, |x_i-y_i|^2}\\
&= \frac{\big| (x_i -z_i) - (z_i-y_i)\big|^2}{2|d_{E}(z_i)| \, |x_i-y_i|^2} \\
&= \frac{1}{2|d_{E}(z_i)| }. 
\end{split}
\eeq
Since, $|d_E(z_i)| \to r_E$ we obtain $2\| S_E\|_{L^\infty} \geq 1/r_E$. 

Let us then show $2\| S_E\|_{L^\infty} \leq 1 /r_E$. To this end, we fix  $x,y \in \pa E$ with $x \neq y$. 
Recall that $G_x \pa E$ denotes the geometric tangent plane of $\pa E$ at $x$. If $x-y  \in G_x\pa E$ then $S_E(x,y)= 0$. If $x-y \notin G_x \pa E$, we find a point $z$ on the line $\{x + t \nu_E(x) : t \in \R\}$ such that $|x-z|= |y-z|$. In other words,  there is $z \in \R^{n+1}$ such that 
\[
\nu_E(x) = \pm \frac{x - z}{|x - z|} \qquad \text{and} \qquad |x-z|= |y-z| =: R.
\]
By repeating the calculations in \eqref{eq:S-E-1} for $x,y$ and $z$ we deduce
$|S_E(x_i,y_i)| =1/(2R)$. Since $|x-z|= |y-z|=R$ and $x,y \in \pa E$, then the signed distance function $d_E$ is not differentiable at $z$. Thus, by the definition of $r_E$ in \eqref{def:r-E} it holds $R \geq r_E$ and we have the inequality $2\| S_E\|_{L^\infty} \leq 1/r_E$.  The rest of the claim is now a direct consequence of $2\| S_E\|_{L^\infty} = 1/r_E$, \eqref{eq:curva-bound-tri} 
and  Proposition \ref{prop:unifball-reg}. 
\end{proof}

An obvious consequence of Lemma \ref{lem:S-E} is that for every open and bounded set $E\subset \R^{n+1}$ it holds
\beq \label{eq:S-E-low}
\| S_E\|_{L^\infty} \geq c_0
\eeq
for a positive constant $c_0 = c_0(n,|E|)$.

We will also use the regularized version of $S_E$, which we define for any $\eps \in \R_+$  as  $S_{E,\eps}: \pa E \times \pa E \to \R$,
\beq \label{def:S-E-eps}
S_{E,\eps}(x,y) := \frac{(x-y) \cdot \nu_E(x)}{|x-y|^2+ \eps}.
\eeq
As in the case of $S_E$, we use the abbreviation $\|S_{E,\epsilon}\|_{L^\infty} = \max \{ |S_{E,\eps}(x,y)| : (x,y) \in \pa E \times \pa E\}$. 
The idea behind considering $S_{E,\eps}$ instead of $S_E$ is that, on the one hand, $S_{E,\epsilon} \rightarrow S$ pointwise
in $\pa E \times \pa E \setminus \{(x,x): x \in \pa E\}$ as $\epsilon$ tends to zero (in particular, $\|S_{E,\epsilon}\|_{L^\infty} \uparrow
\|S_E\|_{L^\infty}$) and, on the other hand, we may differentiate $S_{E,\epsilon}$ on the product $\pa E \times \pa E$
provided that $E$ is sufficiently regular. The followings calculations are similar to \cite{And, DL}  but we give them  in order to be self-consistent.

 Let us first differentiate $S_{E,\eps}$ in the case $E$ is $C^2$-regular. In the computations, the notations $\nabla^x_\tau$ and $\nabla^y_\tau$ stand for the tangential differentiation along $\pa E$ with respect to $x$ and $y$ -variables respectively. Recalling the basic identities \eqref{eq:BH} as well as  observing $B_E \nu_E = 0$ and $\nabla_\tau \id = P_{\pa E}$ on $\pa E$ we compute 
\beq \label{eq:S-E-diff-x}
\begin{split}
\nabla_\tau^x  S_{E,\eps}(x,y) &= \frac{\nabla_\tau^x \big((x-y) \cdot \nu_E(x)\big)}{|x-y|^2+ \eps} - \frac{(x-y) \cdot \nu_E(x)}{(|x-y|^2+ \eps)^2}  \nabla_\tau^x |x-y|^2\\
&= \frac{B_E(x)(x-y) -  2 S_{E,\eps}(x,y)\,  P_{\pa E}(x)(x-y) }{|x-y|^2+ \eps}.
\end{split}
\eeq  
and 
\beq \label{eq:S-E-diff-y}
\begin{split}
\nabla_\tau^y  S_{E,\eps}(x,y) &= \frac{\nabla_\tau^y \big((x-y) \cdot \nu_E(x)\big)}{|x-y|^2+ \eps} - \frac{(x-y) \cdot \nu_E(x)}{(|x-y|^2+ \eps)^2}  \nabla_\tau^y |x-y|^2\\
&= \frac{P_{\pa E}(y)\big( - \nu_E(x) + 2 S_{E,\eps}(x,y) (x-y)\big) }{|x-y|^2+ \eps}
\end{split}
\eeq
for every $(x,y) \in \pa E \times \pa E$.
We immediately obtain the following identities at critical points. 

\begin{lemma}
\label{lem:S-E-critical}
Let  $E \subset \R^{n+1}$ be a bounded and $C^2$-regular set.  Assume $(x,y) \in \pa E \times \pa E$ is a local maximum or a local minimum  point of $S_{E,\eps}$ defined in \eqref{def:S-E-eps}. Then it holds 
\begin{align}
 \label{eq:S-E-critical-1}
B_E(x)(x-y) &= 2  S_{E,\eps}(x,y) P_{\pa E}(x) (x-y) \ \ \text{and} \\
\label{eq:S-E-critical-2}
P_{\pa E}(y)\nu_E(x)  &=  2  S_{E,\eps}(x,y) P_{\pa E}(y) (x-y).
\end{align}
Moreover, the condition $r_E> \sqrt \eps$ implies
\beq \label{eq:S-E-critical-3}
\nu_E(y) = \frac{\nu_E(x)  - 2  S_{E,\eps}(x,y) (x-y)}{\big( \nu_E(x)  - 2  S_{E,\eps}(x,y) (x-y) \big) \cdot \nu_E(y)}.
\eeq
\end{lemma}
\begin{proof}
Since $(x,y)$ is a critical point for the functions $S_{E,\eps}( x , \ \cdot \ )$ and $S_{E,\eps}( \ \cdot \ , y )$, then the equality \eqref{eq:S-E-critical-1} follows from \eqref{eq:S-E-diff-x} and the equality \eqref{eq:S-E-critical-2} follows from \eqref{eq:S-E-diff-y}. Using $P_{\pa E}(y) = I - \nu_E(y) \otimes \nu_E(y)$ and  \eqref{eq:S-E-critical-2}  we have
\[
\nu_E(x)-2  S_{E,\eps}(x,y) (x-y) = \left[\left(\nu_E(x)-2  S_{E,\eps}(x,y) (x-y) \right) \cdot \nu_E(y)\right] \nu_E(y).
\]
The equality \eqref{eq:S-E-critical-3} thus follows once we show 
\beq \label{eq:S-E-critical-4}
\nu_E(x)-2  S_{E,\eps}(x,y) (x-y) \neq 0. 
\eeq
We argue by contradiction and assume $\nu_E(x)=2  S_{E,\eps}(x,y) (x-y)$. Then it holds $S_{E,\eps}(x,y) \neq 0$ and the definition of $ S_{E,\eps}(x,y)$ implies
\[
S_{E,\eps}(x,y) = \frac{(x-y)\cdot \nu_E(x) }{|x-y|^2+ \eps} = 2  S_{E,\eps}(x,y)\frac{|x-y|^2}{|x-y|^2+ \eps}. 
\] 
Therefore, we have $|x-y| = \sqrt{\eps}$. On the other hand,  the contradiction assumption, the definition of $S_{E,\eps}$ and Lemma \ref{lem:S-E} together yield
\[
1 = |\nu_E(x)| = 2 |S_{E,\eps}(x,y)|\,  |x-y| = 2 |S_{E,\eps}(x,y)| \sqrt{\eps} \leq 2 \| S_E\|_{L^\infty} \sqrt{\eps} = \frac{\sqrt\eps}{r_E},
\]
which is impossible by the assumption $r_E > \sqrt \eps$. 
\end{proof}

If $E$ enjoys higher regularity and $\eps$ is sufficiently small, we may naturally extract more information at local extreme points. 
Indeed, if $E$ is $C^3$-regular, then by maximum principle  at a local maximum (minimum)   point $(x,y) \in \pa E \times \pa E$ of $S_{E,\eps}$  it holds
\beq\label{eq:S-E-max-1}
\Delta_\tau^x S_{E,\eps}(x,y) +2 \diver_\tau^x \nabla_\tau^y S_{E,\eps}(x,y)  +  \Delta_\tau^y S_{E,\eps}(x,y)  \overset{(\geq)}{\leq} 0. 
\eeq
We calculate the LHS of \eqref{eq:S-E-max-1} in the next lemma. 

\begin{lemma}
\label{lem:S-E-max}
Let $E \subset \R^{n+1}$ be a bounded and $C^3$-regular set with $r_E > \sqrt \eps$.  At a local maximum (minimum)   point $(x,y) \in \pa E \times \pa E$ of $S_{E,\eps}$ it holds 
\[
\begin{split}
&\frac{\nabla_\tau H_E(x) \cdot (x-y)}{|x-y|^2 +\eps} + \frac{(\nu_E(x)\cdot \nu_E(y)) \, H_E(y)- H_E(x) }{|x-y|^2 +\eps} \\
&\overset{(\geq)}{\leq} |B_E(x)|^2 S_{E,\eps}(x,y) - 2H_E(x) S_{E,\eps}^2(x,y) - 2 H_E(y) S_{E,\eps}(y,x)S_{E,\eps}(x,y).
\end{split}
\]
\end{lemma}
\begin{proof}
First, we compute the terms on the LHS of \eqref{eq:S-E-max-1} by taking tangential divergences of \eqref{eq:S-E-diff-x} and \eqref{eq:S-E-diff-y} with respect to $x$ and $y$ -variables. In the computations, we use the identities \eqref{eq:Delta-id-normal} and the fact that the gradients $\nabla_\tau^x  S_{E,\eps}(x,y)$ and $\nabla_\tau^y S_{E,\eps}(x,y)$ vanish. 
Omitting all the details we obtain by straightforward calculation
\begin{align*}
\Delta_\tau^x S_{E,\eps}(x,y) &= \diver_\tau^x (\nabla_\tau^x S_{E,\eps}(x,y)  )\\
&= \diver_\tau^x\left( \frac{B_E(x)(x-y) -  2 S_{E,\eps}(x,y)\,  P_{\pa E}(x)(x-y) }{|x-y|^2+ \eps}\right)\\
&= \frac{\nabla_\tau H_E(x) \cdot (x-y)}{|x-y|^2 +\eps} +  \frac{H_E(x)}{|x-y|^2 +\eps} - 2 S_{E,\eps}(x,y) \frac{n}{|x-y|^2 +\eps} \\
&\,\,\,\,\,\,\,\,\,\, \,\,\,\,\,\,\,\,\,\,-|B_E|^2 S_{E,\eps}(x,y) + 2 S_{E,\eps}^2(x,y)\, H_E(x), 
\end{align*}
\begin{align*}
\Delta_\tau^y S_{E,\eps}(x,y) &= \diver_\tau^y (\nabla_\tau^y S_{E,\eps}(x,y)  )\\
&= \diver_\tau^y\left(- \frac{P_{\pa E}(y)\nu_E(x) + 2 S_{E,\eps}(x,y) \, P_{\pa E}(y) (x-y)}{|x-y|^2+ \eps}\right)\\
&=   \frac{(\nu_E(x)\cdot \nu_E(y)) \,  H_E(y)}{|x-y|^2 +\eps} - 2 S_{E,\eps}(x,y) \frac{n}{|x-y|^2 +\eps}\\
&\,\,\,\,\,\,\,\,\,\, \,\,\,\,\,\,\,\,\,\,+ 2 S_{E,\eps}(x,y) S_{E,\eps}(y,x)\,  H_E(y) 
\end{align*}
and
\begin{align*}
 \diver_\tau^x \nabla_\tau^y   S_{E,\eps}(x,y) &=\diver_\tau^x\left(- \frac{P_{\pa E}(y)\nu_E(x) + 2 S_{E,\eps}(x,y) \, P_{\pa E}(y) (x-y)}{|x-y|^2+ \eps}\right)\\
&= -\frac{H_E(x)}{|x-y|^2+ \eps} +  \frac{\big(B_E(x)\nu_E(y)\big) \cdot \nu_E(y)}{|x-y|^2+ \eps} \\
&\,\,\,\,\,\,\,\,\,\, \,\,\,\,\,\,\,\,\,\, +2 S_{E,\eps}(x,y) \, \frac{n}{|x-y|^2+ \eps} - 2 S_{E,\eps}(x,y) \, \frac{\big( P_{\pa E}(x) \nu_E(y)\big) \cdot \nu_E(y)}{|x-y|^2+ \eps}.
\end{align*}
Collecting the terms and applying the inequality \eqref{eq:S-E-max-1}, we obtain that at a local maximum (minimum) point it holds 
\[
\begin{split}
0 \overset{(\leq)}{\geq} &\frac{\nabla_\tau H_E(x) \cdot (x-y)}{|x-y|^2 +\eps} +\frac{(\nu_E(x)\cdot \nu_E(y)) \,  H_E(y)- H_E(x)}{|x-y|^2 +\eps}\\
&-|B_E|^2 S_{E,\eps}(x,y) +  2 S_{E,\eps}^2(x,y)\, H_E(x) + 2 S_{E,\eps}(x,y) S_{E,\eps}(y,x)\,  H_E(y) \\
&+  2\frac{\big(B_E(x)\nu_E(y)\big) \cdot \nu_E(y)}{|x-y|^2+ \eps}- 4 S_{E,\eps}(x,y)  \frac{\big( P_{\pa E}(x) \nu_E(y)\big) \cdot \nu_E(y)}{|x-y|^2+ \eps} .
\end{split}
\]
The claim follows once we show that the last line above vanishes, i.e., 
\beq \label{eq:S-E-max-2}
\big(B_E(x)\nu_E(y)\big) \cdot \nu_E(y) = 2 S_{E,\eps}(x,y) \big( P_{\pa E}(x) \nu_E(y)\big) \cdot \nu_E(y).
\eeq

Since $r_E > \sqrt \eps$, this follows by first applying the equalities \eqref{eq:S-E-critical-1} and \eqref{eq:S-E-critical-3} in Lemma \ref{lem:S-E-critical} and recalling $B_E(x) \nu_E (x)= 0$
\[
\begin{split}
B_E(x)\nu_E(y) &= - 2  S_{E,\eps}(x,y) \frac{B_E(x)(x-y)}{\big( \nu_E(x)  - 2  S_{E,\eps}(x,y) (x-y) \big) \cdot \nu_E(y)}\\
&=  - 4  S_{E,\eps}^2(x,y) \frac{P_{\pa E}(x)(x-y)}{\big( \nu_E(x)  - 2  S_{E,\eps}(x,y) (x-y) \big) \cdot \nu_E(y)}.
\end{split}
\]
Then we  use  \eqref{eq:S-E-critical-3} to deduce 
\[
P_{\pa E}(x) \nu_E(y) = - 2  S_{E,\eps}(x,y) \frac{P_{\pa E}(x) (x-y)}{\big( \nu_E(x)  - 2  S_{E,\eps}(x,y) (x-y) \big) \cdot \nu_E(y)}
\]
and \eqref{eq:S-E-max-2} follows. 
\end{proof}

In conclusion, by combining Lemma \ref{lem:S-E}  and Lemma \ref{lem:S-E-max}, we obtain that if a bounded $C^3$-regular set $E \subset \R^{n+1}$ satisfies $r_E > \sqrt \eps$, then at a local maximum (minimum) point $(x,y) \in \pa E \times \pa E$ of $S_{E,\eps}$ it holds
\beq
\label{est:3rdorderbound}
\substack{+ \\ (-)} \left(\frac{\nabla_\tau H_E(x) \cdot (x-y)}{|x-y|^2 +\eps} +\frac{(\nu_E(x)\cdot \nu_E(y)) \,  H_E(y)- H_E(x)}{|x-y|^2 +\eps}\right) \leq C_n \|S_E\|_{L^\infty}^3.
\eeq

\subsection{Short-time uniform ball estimate}
Let us turn our focus on how to prove \eqref{eq:rt}
for an approximative flat flow solution $(E^h_t)_{t\geq0}$ defined in  \eqref{app flow} when the initial set $E_0$ satisfies uniform ball condition with 
given a radius $r_0$. Assuming we may control the evolution of the quantity $\|S_{E^h_t}\|_{L^\infty}$, then thanks to Lemma \ref{lem:S-E} we also control (from below) the uniform ball condition for $E_t^h$.

 We motivate ourselves by consider this first in the continuous and embedded setting. Assume $(E_t)_t$ is a smooth flow and let $\nu_t$ and $V_t$ denote the outer
unit normal of $E_t$ and the normal velocity of the flow on $\pa E_t$ respectively. Then one may use the fact that for fixed $t$ there is a smooth \emph{normal parametrization} $(\Phi^t_s)_s$ of the flow such that 
$\Phi^t_0 = \id$ and $\pa_s \Phi^t_s = [V_s \, \nu_s] \circ \Phi^t_s$. This follows essentially from \cite[Thm 8]{AD}. 
It is straightforward to calculate that for such a parametrization
\beq 
\label{eq:flow-conti}
 \frac{\d}{\d s} \Phi^t_{t+s}  \ \bigg |_{s=0}   = V_t \, \nu_t \quad \text{and} \quad  \frac{\d}{\d s}  (\nu_{E_{t+s}} \circ \Phi^t_{t+s}) \ \bigg |_{s=0} = - \nabla_{\tau} V_t \quad \text{on} \ \ \pa E_t.
\eeq 
In the case of volume preserving mean curvature flow, we have $V_s = -(H_s - \bar {H_s})$, where $H_s$ is the scalar mean curvature on $\pa E_s$ and $\bar {H_s}$ its
integral average over $\pa E_s$. If $x$ and $y$ are distinct points on $\pa E_t$, then by using \eqref{eq:flow-conti} and 
the previous identity, we may compute
\beq
\label{eq:flow-conti-derivative}
\begin{split}
 \frac{\d}{\d s} S_{E_{t+s}} (\Phi^t_s(x),\Phi^t_s(y)) \bigg |_{s=0} = 
&\frac{\nabla_\tau H_E(x) \cdot (x-y)}{|x-y|^2} +\frac{(\nu_E(x)\cdot \nu_E(y)) \,  H_E(y)- H_E(x)}{|x-y|^2} \\
&\ \ \ \ + R_t(x,y),
\end{split}
\eeq
where the remainder term $R_t(x,y)$  has a bound $|R_t(x,y)| \leq C_n \|S_{E_t}\|_{L^\infty}^3$. Suppose that $\|S_{E_t}\|_{L^\infty} =  \pm S_{E_t}(x,y)$ and  the function $s \mapsto \|S_{E_{t+s}}\|_{L^\infty}$ is  differentiable at $s=0$, then  we deduce
\[
 \frac{\d}{\d s} \| S_{E_{t+s}}\|_{L^\infty} \bigg |_{s=0}  = \pm \frac{\d}{\d s} S_{E_{t+s}} (\Phi^t_s(x),\Phi^t_s(y)) \bigg |_{s=0}.
\]
Again, the estimate \eqref{est:3rdorderbound} also holds for $S_E$
when the points are distinct. Thus, by possibly increasing $C_n$, we infer from above and \eqref{eq:flow-conti-derivative}
\beq
\label{est:flow-conti-derivative-control}
\frac{\|S_{E_{t+s}}\|_{L^\infty} -  \|S_{E_t}\|_{L^\infty}} {s} \leq C_n \|S_{E_t}\|_{L^\infty}^3
\eeq
provided that $s \neq 0$ is sufficiently small.

The idea is to mimic the previous argument in the discrete setting for an approximative flat flow $(E^h_t)_{t\geq 0}$. To this end,
we need to approximate the two-point functional by its $\epsilon$-regularized version. We consider the element  $E^h_t$ and its consequent set  $E^h_{t+s}$. For sake of brevity, we use the shorthand 
notations $E_1 = E^h_t$ and $E_2 = E^h_{t+s}$ for the rest of the subsection. First, we want to find a discrete version of the equalities in \eqref{eq:flow-conti}. Suppose that an element $E_1$ satisfies uniform ball condition and $h$ is so small that
 by the discussion of the previous section we have that $E_2$ is $C^1$-regular set,
$\pa E_2 \subset \mathcal N_{r_{E_1}}(\pa E_1)$ and $\nabla d_{E_1} \cdot \nu_{E_2}>0$ on $\pa E_2$ are satisfied.

Then it is natural to project the boundary $\pa E_2$ to $\pa E_1$ by the projection 
$\pi_{\pa E_1}$ and, hence, using the identities in \eqref{def:proj} we have
\[
\frac{\id - \pi_{\pa E_1}}{h} =\frac{d_{E_1}}{h} (\nu_{ E_2} \circ \pi_{\pa E_1})  \quad \text{on $\pa E_2$}
\]
which can be seen as a discrete time counterpart of the first identity in \eqref{eq:flow-conti}. In the next simple but crucial lemma, we derive a relation between  $\nu_{E_2}$ and $\nu_{E_1} \circ \pi_{\pa E_1}$ for $x \in \pa E_2$.

\begin{lemma}
\label{lem:maaginen}
Assume that $E_1 \subset \R^{n+1}$ is open and  satisfies uniform ball condition, 
and $E_2$ is a $C^1$-regular set such that $\pa E_2 \subset \mathcal N_{r_{E_1}}(\pa E)$ and 
$ \nabla d_{E_1} \cdot \nu_{E_2}>0$ on $\pa E_2$. Then 
\[
\nu_{E_1} \circ \pi_{\pa E_1}  = \nabla_{\tau_2} d_{E_1}   + \sqrt{1 - |\nabla_{\tau_2} d_{E_1} |^2} \, \nu_{E_2} \ \ \text{on} \ \ \pa E_2.
\]
\end{lemma}
\begin{proof}
By using the second identity of \eqref{def:proj} for $d_{E_1}$ as well as the definition of tangential gradient the following holds on $\pa E_2$
\[
\begin{split}
\nu_{E_1} \circ \pi_{\pa E_1} = \nabla d_{E_1} =P_{\pa E_2}  \nabla d_{E_1}  +  ( \nabla d_{E_1} \cdot \nu_{E_2}) \nu_{E_2} =\nabla_{\tau_2}  d_{E_1}  +  ( \nabla d_{E_1} \cdot \nu_{E_2}) \nu_{E_2} 
\end{split}
\] 
Since $|\nu_{E_1} \circ \pi_{\pa E_1}|=1=|\nu_{E_2}|$ and $\nabla_{\tau_2}  d_{E_1} \cdot \nu_{E_2} = 0$, 
then the previous decomposition implies $| \nabla d_{E_1} \cdot \nu_{E_2}|= \sqrt{1-|\nabla_{\tau_2} d_{E_1}|^2}$.
Thus, the claim follows from the assumption $\nabla d_{E_1} \cdot \nu_{E_2}>0$ on $\pa E_2$.
\end{proof}

The equality in the statement of  Lemma \ref{lem:maaginen} gives us a discrete analog for the second equality in \eqref{eq:flow-conti} as 
\beq \label{eq:maagi1}
\nu_{E_2} - \nu_{E_1} \circ \pi_{\pa E_1}= -\nabla_{\tau_2} d_{E_1}  + \frac{|\nabla_{\tau_2} d_{E_1}|^2}{1+\sqrt{1 - |\nabla_{\tau_2} d_{E_1}|^2}} \, \nu_{E_2} \ \ \text{on} \ \ \pa E_2.
\eeq
or equivalently
\beq \label{eq:maagi2}
\begin{split}
\nu_{E_2} - \nu_{E_1} \circ \pi_{\pa E_1}  = &-\left( \frac{1}{\sqrt{1 - |\nabla_{\tau_2} d_{E_1}|^2}} \right) \nabla_{\tau_2} d_{E_1} \\
&+  \frac{|\nabla_{\tau_2} d_{E_1}|^2}{\sqrt{1 - |\nabla_{\tau_2} d_{E_1}|^2} + 1 - |\nabla_{\tau_2} d_{E_1}|^2}  \, \nu_{E_1} \circ \pi_{\pa E_1}  \ \ \text{on} \ \ \pa E_2
\end{split}
\eeq
which will be useful later. We need yet one technical lemma related to  the projection $\pi_{\pa E_1}$ on the consequent boundary $\pa E_2$.  
\begin{lemma}
\label{lem:distance-compa}
Let $E_1,E_2 \subset \R^{n+1}$ be open and  bounded sets satisfying uniform ball condition.
If $\pa E_2 \subset \mathcal N_{r_{E_1}/2}(\pa E)$, then for any  $x,y \in \pa F$ satisfying $\pi_{\pa E_1}(x) \neq \pi_{\pa E_1}(y)$
it holds
\[
\begin{split}
&\Big| |\pi_{\pa E_1}(x)-\pi_{\pa E_1}(y)|^2 - |x-y|^2\Big| \\
 \leq C_0\|d_{E_1}\|_{L^\infty(\pa E_2)} &\left(\|S_{E_1}\|_{L^\infty} +  \|S_{E_2}\|_{L^\infty} + \|d_{E_1}\|_{L^\infty(\pa E_2)} \|S_{E_2}\|^2_{L^\infty} \right)|x-y|^2,
\end{split}
\]
where $C_0\geq 1$ is a universal constant.
\end{lemma}
\begin{proof} 
First, we obtain from \eqref{def:proj} and the definition of $S_{E_1}$ that
\[
\begin{split}
|\pi_{\pa E_1}(x) -\pi_{\pa E_1}(y)|^2 - |x-y|^2 =&  -2 d_{E}(x) S_{E_1}(\pi_{\pa E_1}(x),\pi_{\pa E_1}(y))|\pi_{\pa E_1}(x) -\pi_{\pa E_1}(y)|^2 \\
&- 2 d_{E_1}(y) S_{E_1}(\pi_{\pa E_1}(y),\pi_{\pa E_1}(x))|\pi_{\pa E_1}(x) -\pi_{\pa E_1}(y)|^2  \\
&- \big|d_{E_1}(x) (\nu_{E_1}\circ\pi_{\pa E_1})(x)  - d_{E_1}(y) (\nu_{E_1} \circ \pi_{\pa E_1})(y) \big|^2. 
\end{split}
\]
Thus,
\begin{align*}
 \Big| |\pi_{\pa E_1}(x)-\pi_{\pa E_1}(y)|^2 - |x-y|^2\Big| 
\leq 4  &\|d_{E_1}\|_{L^\infty(\pa E_2)}  \|S_{E_1}\|_{L^\infty}|\pi_{\pa E_1}(x) -\pi_{\pa E_1}(y)|^2 \\
+ 2 &|d_{E_1}(x) |^2 |(\nu_{E_1} \circ \pi_{\pa E_1})(x)  - (\nu_{E_1} \circ \pi_{\pa E_1})(y) |^2 + 2 |d_{E_1}(x) -d_{E_1}(y)|^2 \\
\leq  4  &\|d_{E_1}\|_{L^\infty(\pa E_2)}  \|S_{E_1}\|_{L^\infty}|\pi_{\pa E_1}(x) -\pi_{\pa E_1}(y)|^2 \\
+ 2 &\|d_{E_1}\|_{L^\infty(\pa E_2)}^2 |(\nu_{E_1} \circ \pi_{\pa E_1})(x)  - (\nu_{E_1} \circ \pi_{\pa E_1})(y) |^2 \\
+ 2 &|d_{E_1}(x) -d_{E_1}(y)|^2.
\end{align*}
The normal $\nu_{E_1}$ is $1/r_{E_1}$-Lipschitz continuous by Proposition \ref{prop:unifball-reg} and $\pi_{\pa E_1}$ is 
$2$-Lipschitz continuous in $\mathcal N_{r_{E_1}/2}(\pa E_1)$ by Lemma \ref{lemma:globalLip}. On the other hand, recalling Lemma \ref{lem:S-E} we conclude 
$\|d_{E_1}\|_{L^\infty(\pa E_2)}\|S_{E_1}\|_{L^\infty} \leq 1/4$. Hence, we infer from previous estimate 
\beq
\label{distance-compa1}
\Big| |\pi_{\pa E_1}(x)-\pi_{\pa E_1}(y)|^2 - |x-y|^2\Big| \leq   24 \|d_{E_1}\|_{L^\infty(\pa E_2)} \|S_{E_1}\|_{L^\infty}\, |x-y|^2 
+2 |d_{E_1}(x) -d_{E_1}(y)|^2.
\eeq
Thus, we are remain to estimate the term $|d_{E_1}(x) -d_{E_1}(y)|^2$ on the boundary $\pa E_2$.
We divide this into two cases. First, suppose that $|x-y|\geq r_{E_2}/2$. Then using Lemma \ref{lem:S-E} we obtain
\beq
\label{distance-compa2}
|d_{E_1}(x)-d_{E_1}(y)|^2 \leq \frac{4 \|d_{E_1}\|_{L^\infty(\pa E_2)}^2}{r_{E_2}^2} |x-y|^2 \leq
16 \|d_{E_1}\|_{L^\infty(\pa E_2)}^2 \|S_{E_2}\|_{L^\infty}^2|x-y|^2.
\eeq
Suppose then $|x-y|<r_{E_1}/2$. We define a $C^1$-extension $\tilde d_{E_1} : \mathcal N_{r_{E_2}}(\pa E_2) \rightarrow \R$ of the restriction $d_{E_1}|_{\pa E_2}$ by setting $\tilde d_{E_1} = d_{E_1} \circ \pi_{\pa E_2}$. Then 
$\nabla \tilde d_{E_1} = \nabla \pi_{\pa E_2} \nabla_{\tau_2} d_{E_1} \circ \pi_{\pa E_2}$ and
by Lemma  \ref{lemma:globalLip} $|\nabla \pi_{\pa E_2}|_\op \leq 2$ in $\mathcal N_{r_{E_2}/2}(\pa E_2)$ so $|\nabla \tilde d_{E_1}| \leq 2 \|\nabla_{\tau_2} \tilde d_{E_1}\|_{L^\infty(\pa E_2)}$ there.
Since the line segment $J_{yx}$ belongs to $\mathcal N_{r_{E_2}/2}(\pa E_2)$, we have
\beq
\label{distance-compa3}
|d_{E_1}(x) - d_{E_1}(y)|^2 \leq 4 \|\nabla_{\tau_2} d_{E_1}\|_{L^\infty(\pa E_2)}^2|x-y|^2.
\eeq
By Lemma \ref{lemma:globalLip} we have $|\nabla^2 d_{E_1}|_\op \leq 2/r_{E_1}$ in $\mathcal N_{r_{E_1}}(\pa E_1)$. Therefore, by using 
Lemma \ref{lemma:simpleinterpolation} and Lemma \ref{lem:S-E} we get an estimate
\begin{align}
\notag
\|\nabla_{\tau_2} d_{E_1}\|_{L^\infty(\pa E_2)}^2 
&\leq 4\|d_{E_1}\|_{L^\infty(\pa E_2)}\left(\sup_{\pa E_2} |\nabla^2 d_{E_1}|_{\op} + \frac{\|\nabla_\tau d_{E_1}\|_{L^\infty(\pa E_2)}}{r_{E_2}}\right) \\
\label{gradinterpolation}
&\leq 16\|d_{E_1}\|_{L^\infty(\pa E_2)}\left( \|S_{E_1}\|_{L^\infty} + \|S_{E_2}\|_{L^\infty}\right).
\end{align}
Thus, we gather the estimate as claimed from \eqref{distance-compa1}, \eqref{distance-compa2}, \eqref{distance-compa3}
and the estimate above.
\end{proof}

We are now ready prove an analogous estimate to \eqref{est:flow-conti-derivative-control}  in the discrete setting.  

\begin{lemma}
\label{lem:2-point-arg}
Assume that $E_1 \subset \R^{n+1}$ is an open and bounded set, with $|E_1| = m_0$, which satisfies  uniform ball condition with
 radius $r_0 \in \R_+$. Let $E_2$ be any minimizer of the energy $\mathcal{F}_h( \ \cdot \ , E_1)$ defined in \eqref{min mov}. 
Then there is  $h_0=h_0(n,m_0,r_0)$ such that for $h \leq h_0$  $E_2$ is $C^3$-regular and 
\[
\frac{\|S_{E_2}\|_{L^\infty} -\|S_{E_1}\|_{L^\infty}}{h} \leq C_n \|S_{E_1}\|_{L^\infty}^3.
\]
If in addition $E_1$ is $C^k$-regular, then $E_2$ is $C^{k+2}$-regular.
\end{lemma}
\begin{proof} As previously, $C=C(n,m_0,r_0) >0$ may change from line to line. We find $h_0=h_0(n,m_0,r_0)\in \R_+$ such that assuming $h \leq h_0$ implies that the conclusions of
Proposition \ref{prop:distance-bound},  Lemma \ref{lemma:improved-regularity} and Remark \ref{rem:improved-regularity} are valid. Let us quickly summarize what we have achieved so far.
First, $E_2$ is open and bounded, $C^3$-regular set, or $C^{k+2}$-regular set provided that $E_1$ is $C^k$-regular, and it satisfies uniform ball condition with radius $c_0h^{1/3}$ for a constant $c_0=c_0(n,m_0,r_0) >0$.
Hence, by Lemma \ref{lem:S-E} we have apriori estimate
\beq
\label{2-point-arg-1}
\|S_{E_2}\|_{L^\infty} \leq C h^{-\frac13}.
\eeq
Second, $\pa E_2$ is ``close" to $\pa E_1$. To be more precise, we have $\|d_{E_1}\|_{L^\infty(E_2)} \leq C_n h /r_0$
and we may assume that $\pa E_2 \subset \mathcal N_{r_0/2}(\pa E_1)$.
Moreover, it holds that $\nabla d_{E_1} \cdot \nu_{E_2} > 0$ on $\pa E_2$ and $\pi_{\pa E_1}$ is injective on $\pa E_2$.
Third, we have the Euler-Lagrange equation \eqref{euler} on $\pa E_2$ in the classical sense.
 
Thus, we assume that $h \leq h_0$. We might need to shrink $h_0$ but always in a way that we preserve the dependency 
$h_0=h_0(n,m_0,r_0)$. By combining the estimate $\|d_{E_1}\|_{L^\infty(E_2)} \leq C_n h /r_0$ from Proposition \ref{prop:distance-bound} with Lemma \ref{lem:S-E} and \eqref{2-point-arg-1} and by possibly shrinking $h_0$ we obtain
\beq
\label{2-point-arg-2}
\frac{\|d_{E_1}\|_{L^\infty(E_2)}}{h} \leq C_n \|S_{E_1}\|_{L^\infty} \quad \text{and} \quad \|S_{E_2}\|_{L^\infty} \|d_{E_1}\|_{L^\infty(E_2)} \leq 1.
\eeq 
Then, by \eqref{euler}, Lemma \ref{lem:S-E} and the first estimate in \eqref{2-point-arg-2}, the Lagrange multiplier $\lambda^h$ can be controlled as 
\beq
\label{2-point-arg-2b}
|\lambda^h| \leq  \frac{\|d_{E_1}\|_{L^\infty(E_2)}}{h} + \|H_{E_2}\|_{L^\infty(\pa E_2)} \leq C_n( \|S_{E_1}\|_{L^\infty}+ \|S_{E_2}\|_{L^\infty}).
\eeq
The claim follows once we show
\beq
\label{2-point-arg-2c}
\frac{\|S_{E_2}\|_{L^\infty} -\|S_{E_1}\|_{L^\infty}}{h} \leq C_n\left(\|S_{E_1}\|_{L^\infty}^3+\|S_{E_2}\|_{L^\infty}^3\right).
\eeq
Indeed, assuming the above holds true we have by Lemma \ref{lem:S-E} and \eqref{2-point-arg-1}
\[
\|S_{E_2}\|_{L^\infty} -\|S_{E_1}\|_{L^\infty} \leq  C_n r_0^{-3} h + C h^\frac13 \|S_{E_2}\|_{L^\infty} 
\]
and, hence, recalling \eqref{eq:S-E-low} and shrinking $h_0$, if neccessary, we obtain 
$\|S_{E_2}\|_{L^\infty} \leq 2\|S_{E_1}\|_{L^\infty}$. Thus, reiterating the previous inequality via \eqref{2-point-arg-2c} yields the claim.

To prove \eqref{2-point-arg-2c}, we initially fix any $\eps < r_{E_2}^2$ and choose $(x,y) \in  \pa E_2 \times \pa E_2$
such that $|S_{E_2,\eps}(x,y)|=\|S_{E_2,\eps}\|_{L^\infty}$. 
Since $\|S_{E_2,\eps}\|_{L^\infty}>0$, then $x \neq y$ and, hence, the injectivity of $\pi_{\pa E_1}$ on $\pa E_2$ ensures
that $\pi_{\pa E_1}(x) \neq \pi_{\pa E_1}(x)$. In order to simplify our notations, we write 
$\pi=\pi_{\pa E_1}$ and $H_2 = H_{E_2}$ for short.
By using the definition in \eqref{def:S-E-eps}, the identities \eqref{def:proj} and \eqref{eq:maagi1} as well as the Euler-Lagrange equation we may decompose the difference quotient  as
\begin{align}
\notag
\frac{1}{h} \big(S_{E_2,\eps}&(x,y) -S_{E_1,\eps}(\pi(x),\pi(y))\big)  \\
\notag
&=\frac{(x-y) \cdot \nabla_{\tau_2} H_2(x)}{|x-y|^2+ \eps} + \frac{\left( \nu_{E_1} (x) \cdot \nu_{E_2}(y) \right)H_2(y) - H_2 (x)}{|x-y|^2+ \eps} \\
\label{2-point-arg-3}
 & \  +  \frac1h\frac{|\nabla_{\tau_2} d_{E_1}(x) |^2}{1+\sqrt{1 - |\nabla_{\tau_2} d_{E_1}(x)|^2}}S_{E_2,\eps}(x,y)+  \left( \lambda^h - \frac{d_{E_1}(y)}{2h}\right)\frac{|\nu_{E_1}(x)- \nu_{E_1}(y) |^2}{|x-y|^2+ \eps} \\
\notag
& \ + \frac{d_{E_1}(y)}{2h} \frac{|\nu_{E_1}(\pi(x))- \nu_{E_1}(\pi(y)) |^2}{|x-y|^2+ \eps}
+\frac1h\left(\frac{|\pi(x) -\pi(y)|^2- |x-y|^2}{|x-y|^2+\eps} \right) S_{E_1,\eps}(\pi(x),\pi(y)) .
\end{align}
Next, we estimate the last four terms on the RHS.
First, since  $\pa E_2 \subset \mathcal N_{r_0/2}(\pa E) \subset \mathcal N_{r_{E_1}/2}(\pa E_1)$, we have the estimate \eqref{gradinterpolation} for $\|\nabla_{\tau_2} d_{E_1}\|_{L^\infty(\pa E_2)}^2$ and, hence, recalling the first estimate in
\eqref{2-point-arg-2}
we have
\beq
\label{2-point-arg-4}
\begin{split}
\left|\frac1h\frac{|\nabla_{\tau_2} d_{E_1}(x) |^2}{1+\sqrt{1 - |\nabla_{\tau_2} d_{E_1}(x)|^2}}S_{E_2,\eps}(x,y)\right|
\leq \, & \frac{16\|d_{E_1}\|_{\pa E_2}}{h}\left( \|S_{E_1}\|_{L^\infty}\|S_{E_2}\|_{L^\infty} + \|S_{E_2}\|_{L^\infty}^2\right) \\
\leq \, & C_n\left(\|S_{E_1}\|_{L^\infty}^3 + \|S_{E_2}\|_{L^\infty}^3\right) .
\end{split}
\eeq
For the next term, we use  Lemma \ref{lem:S-E}, the first estimate in \eqref{2-point-arg-2} and \eqref{2-point-arg-2b} to obtain
\beq
\label{2-point-arg-5}
\begin{split}
\left| \left( \lambda^h - \frac{d_{E_1}(y)}{2h}\right)\frac{|\nu_{E_1}(x)- \nu_{E_1}(y) |^2}{|x-y|^2+ \eps}\right|
\leq \, & C_n\left(|\lambda^h| + \|S_{E_1}\|_{L^\infty(\pa E_1)}\right)  \|S_{E_2}\|^2_{L^\infty(\pa E_2)}\\
\leq \, & C_n\left(\|S_{E_1}\|_{L^\infty(\pa E_1)}^3 + \|S_{E_2}\|_{L^\infty(\pa E_2)}^3\right).
\end{split}
\eeq
By Proposition \ref{prop:unifball-reg}  $\nu_{E_1}$ is $1/r_0$-Lipschitz and by Lemma \ref{lemma:globalLip} $\pi$ is 2-Lipschitz continuous in $\mathcal N_{r_0/2}(\pa E_1)$. Thus, by Lemma \ref{lem:S-E} and  the first inequality in \eqref{2-point-arg-2}, we estimate the second last term as
\beq
\label{2-point-arg-6}
\begin{split}
\left| \frac{d_{E_1}(y)}{2h} \frac{|\nu_{E_1}(\pi(x))- \nu_{E_1}(\pi(y)) |^2}{|x-y|^2+ \eps}\right|
\leq \, & C_n \|S_{E_1}\|_{L^\infty} \frac{1}{r_0^2}  \frac{|\pi(x)- \pi(y) |^2}{|x-y|^2} \\
\leq \, & C_n \|S_{E_1}\|^3_{L^\infty}.
\end{split}
\eeq
Finally, by using Lemma \ref{lem:distance-compa} and the identities in \eqref{2-point-arg-2} we have
\beq
\label{2-point-arg-7}
\begin{split}
     &\left|\frac1h\left(\frac{|\pi(x) -\pi(y)|^2- |x-y|^2}{|x-y|^2+\eps} \right) S_{E_1,\eps}(\pi(x),\pi(y))\right| \\
\leq &C_n\frac{\|d_{E_1}\|_{L^\infty(\pa E_2)}}{h} \left(\|S_{E_1}\|_{L^\infty} +  \|S_{E_2}\|_{L^\infty} + \|d_{E_1}\|_{L^\infty(\pa E_2)} \|S_{E_2}\|^2_{L^\infty} \right)\|S_{E_1}\|_{L^\infty} \\
\leq &C_n\left(\|S_{E_1}\|_{L^\infty}^3 + \|S_{E_2}\|_{L^\infty}^3\right). 
\end{split}
\eeq
We infer from \eqref{2-point-arg-3},  \eqref{2-point-arg-4},  \eqref{2-point-arg-5},  \eqref{2-point-arg-6} and  \eqref{2-point-arg-7} the expression
\begin{align}
\notag
 \frac{S_{E_2,\eps}(x,y) -S_{E_1,\eps}(\pi(x),\pi(y))}{h} =\frac{(x-y) \cdot \nabla_{\tau_2} H_2(x)}{|x-y|^2+ \eps} + \frac{\left( \nu_{E_1} (x) \cdot \nu_{E_2}(y) \right)H_2(y) - H_2 (x)}{|x-y|^2+ \eps} + R,
\end{align}
where for the remainder term it holds $|R| \leq C_n\left( \|S_{E_1}\|_{L^\infty}^3 + \|S_{E_2}\|_{L^\infty}^3\right)$. Since $(x,y)$ is a maximum (or minimum) point for $S_{E_2,\eps}$, then we conclude from \eqref{est:3rdorderbound}
\[
\frac{\|S_{E_2,\eps}\|_{L^\infty} -\|S_{E_1,\eps}\|_{L^\infty}}{h} \leq C_n\left( \|S_{E_1}\|_{L^\infty}^3 + \|S_{E_2}\|_{L^\infty}^3\right).
\]
Since now $\|S_{E_i,\eps}\|_{L^\infty}  \uparrow \|S_{E_i,\eps}\|_{L^\infty}$ for $i=1,2$ as $\eps$ tends to zero,
the above yields  \eqref{2-point-arg-2c} and we conclude the proof.
\end{proof}

We may now prove the main result of this section which is the uniform ball condition estimate for the approximative flat flow. 
\begin{theorem}
\label{thm2}
Let $E_0 \subset \R^{n+1}$  be an open and bounded set which satisfies uniform ball condition with radius $r_0\in \R_+$ and
let $m_0$ denote its volume. There are $h_0=h_0(n,m_0,r_0) \in \R_+$ and  $T_0=T_0(n,r_0) \in \R_+$ such that if $h \leq h_0$, then any approximative flat flow  $(E_t^h)_{t \geq 0}$  of \eqref{eq:VMCF} starting from $E_0$  satisfies uniform ball condition with radius $r_0/2$ for all $t \leq T_0$. Moreover, $E^h_t$ is $C^{1+2\lfloor t/h\rfloor}$-regular for every $0 \leq t \leq T_0$.
\end{theorem}

\begin{proof} By a slight  abuse of notation, we set $h_0$ to be as in Lemma \ref{lem:2-point-arg} for the parameters $n$, $m_0$ and $r_0/2$.
Then we choose 
\beq
\label{eq:thm2-1}
T_0 = \frac{r_0^2}{4C_n}\, , 
\eeq
where the dimensional constant is the same as in Lemma \ref{lem:2-point-arg}. We assume that $h \leq h_0$ and consider
  an approximative flat flow  $(E_t^h)_{t \geq 0}$ starting from $E_0$ obtained via the minimizing movements scheme \eqref{app flow}. We may assume $h \leq T_0$, since otherwise the proof is trivial.
Since $E_0$ satisfies uniform ball condition with radius $r_0$, we have by Lemma \ref{lem:S-E} that $\|S_{E_0}\|_{L^\infty}
 =1/(2r_0)$. Then we set
\[
K=\sup \left\{ k \in \N : E_t^h \ \ \text{satisfies  uniform ball condition with} \ \ \|S_{E_{lh}^h}\|_{L^\infty}  \leq \frac{1}{r_0}
\ \ \text{for} \ \ 0 \leq l \leq k \right\}. 
\]  
Note that if $E_{kh}^k$ is a bounded set satisfying  uniform ball condition with $\|S_{E_k^h}\|_{L^\infty}  \leq 1 /r_0$, then
thanks to Lemma \ref{lem:S-E} we know that it satisfies uniform ball condition with  radius $r_0/2$. Thus, it follows from the construction of 
$(E^h_t)_{t \geq 0}$, the choice of $h_0$, and Lemma \ref{lem:2-point-arg} that we have $E_{(k+1)h}^h$ is a bounded $C^3$-regular set satisfying
\[
\|S_{E_{(k+1)h}^h}\|_{L^\infty} \leq  \|S_{E_k^h}\|_{L^\infty} + C_n h\|S_{E_k^h}\|_{L^\infty}^3 \leq  \|S_{E_k^h}\|_{L^\infty} + C_nr_0^{-3} h.
\] 
Since $h_0 \leq T_0$, then the choices in \eqref{eq:thm2-1} imply that $K$ is well-defined. By summing the above from $k=0$ to $k=K$ we obtaon
\[
\frac{1}{r_0} \leq \|S_{E_{(K+1)h}^h}\|_{L^\infty} \leq \|S_{E_0}\|_{L^\infty} + \frac{C_n}{r_0^3} (K+1)h = \frac{1}{2r_0} + \frac{C_n}{r_0^3} (K+1)h.
\]
This yields $K \geq \lfloor T_0/h \rfloor$ and, hence, it follows from the construction \eqref{app flow} that $E^h_t$ satisfies uniform ball condition with radius $r_0/2$ for every $0 \leq t \leq T_0$. The last claim then follows 
directly from Lemma \ref{lem:2-point-arg}.
\end{proof}


 \section{Higher regularity}

In this section we utilize the uniform ball condition (UBC) from previous section and prove the full regularity of the flat flow solution of \eqref{eq:VMCF}. It is well known that the classical solution for the mean curvature flow is well defined as long as the second fundamental form stays bounded \cite{MantegazzaBook}. For the volume preserving flow this is not enough as the flow may develop singularities even if it stays regular \cite{M, MaSim}. However, if the flow in addition satisfies UBC then these singularities do not occur.  In this section we show that the approximative flat flow becomes instantaneously smooth and stays smooth as long as it satisfies UBC. We will prove this via energy estimates.


Our starting point is the formula in Lemma \ref{lem:maaginen}, which for sets $E_1$ and $E_2$ as in the lemma, gives the formula which relates their normals as 
\[
\nu_{E_1} \circ \pi_{\pa E_1} = \nabla_{\tau_2} d_{E_1}  + \sqrt{1 - |\nabla_{\tau_2} d_{E_1}|^2} \, \nu_{E_2} \quad \text{on }\,  \pa E_2.
\]
Recall that $\nabla_{\tau_2}$ denotes the tangential gradient on $\pa E_2$. Assume now further that  $E_2$ is a minimizer of the functional  $\mathcal{F}_h(\ \cdot \ , E_1) $ defined in  \eqref{min mov}, we may use the Euler-Lagrange equation \eqref{euler} and have
\beq \label{eq:maagi-deri}
\nu_{E_1} \circ \pi_{\pa E_1} =  - h\, \nabla_{\tau_2} H_{E_2}  + \sqrt{1 - |\nabla_{\tau_2} d_{E_1}|^2} \, \nu_{E_2}  \quad \text{on }\,  \pa E_2.
\eeq
This identity is simple enough  for us to differentiate multiple times and this in turn gives us formula which is the discrete analog of the identity  (see e.g. \cite[Lemma 3.5]{Mantegazza2002})
\beq \label{eq:deri-high-conti}
\frac{\d}{\d t} \Delta^k H_{E_t} = \Delta^{k+1} H_{E_t} + \text{lower order terms} .
\eeq  

Let us, for the sake of clarification, show how we obtain the discrete version of \eqref{eq:deri-high-conti} for $k =0$ from \eqref{eq:maagi-deri}, which reads as follows
\beq \label{eq:high-deri-1}
\sqrt{1-|\nabla_{\tau_2} d_{E_1}|^2} H_{E_2}  - H_{E_1} \circ \pi_{\pa E_1}  =h\,   \Delta_{\tau_2} H_{E_2}  + h^2 \,  A_2(\cdot)  \nabla_{\tau_2} H_{E_2} \cdot \nabla_{\tau_2} H_{E_2}   + a_1(\cdot) d_{E_1} \ \ \text{on} \ \   \pa E_2,
\eeq
where the function  $a_1(\cdot)$  and  the matrix field $A_2(\cdot )$ depend smoothly on $d_{E_1}$, $\nu_{E_1}\circ \pi_{\pa E_1}$
 $\nu_{E_2}$, $B_{E_1}\circ \pi_{\pa E_1}$ and $B_{E_2}$. In particular,  since $E_1$ and $E_2$ satisfy uniform ball condition with radius $r_0/2$, then  $a_1(\cdot)$  and $A_2(\cdot )$ are uniformly bounded. 

Indeed, by applying the tangential divergence on \eqref{eq:maagi-deri} we have
\[
\diver_{\tau_2}\big(\nu_{E_1} \circ \pi_{\pa E_1}\big) = - h\, \Delta_{\tau_2} H_{E_2}  +  \sqrt{1 - |\nabla_{\tau_2} d_{E_1}|^2} \, H_{E_2}\quad \text{on }\,  \pa E_2.
\]
In order to calculate the LHS, we use \eqref{def:proj},  \eqref{eq:dist-22} and \eqref{eq:dist-3} to obtain 
\[
\begin{split}
\nabla \big(\nu_{E_1} \circ \pi_{\pa E_1}\big) = \nabla^2 d_{E_1} &=  B_{E_1}\circ \pi_{\pa E_1}(I + d_{E_1}B_{E_1}\circ \pi_{\pa E_1})^{-1} \\
& = B_{E_1}\circ \pi_{\pa E_1} - d_{E_1}\left(I + d_{E_1}B_{E_1}\circ \pi_{\pa E_1}\right)^{-1}(B_{E_1}\circ \pi_{\pa E_1})^2
\end{split}
\]
which holds in the tubular neighborhood  $\mathcal N_{r_0}(\pa E_1)$, where we also used the fact 
\[
(B_{E_1}\circ \pi_{\pa E_1})(I + d_{E_1}B_{E_1}\circ \pi_{\pa E_1})^{-1} =(I + d_{E_1}B_{E_1}\circ \pi_{\pa E_1})^{-1} (B_{E_1}\circ \pi_{\pa E_1}). 
\]
Agian from \eqref{eq:maagi-deri}  we have
\[
\nu_{E_2} = \frac{1}{\sqrt{1-|\nabla_{\tau_2} d_{E_1}|^2}} \left( h \, \nabla_{\tau_2} H_{E_2} + \nu_{E_1} \circ \pi_{\pa E_1} \right) 
\]
on $\pa E_2$. Using the above identities and the fact   $B_{E_1} \nu_{E_1}= 0$ on $\pa E_1$, we have the following equality on 
$\pa E_2$
\[
\begin{split}
\diver_{\tau_2}\big(\nu_{E_1} \circ \pi_{\pa E_1}\big)  &= \Tr \big((I- \nu_{E_2}\otimes \nu_{E_2}) \nabla^2 d_{E_1}   \big)\\
&= H_{E_1} \circ \pi_{\pa E_1} -  d_{E_1}\Tr \big( \left(I + d_{E_1}B_{E_1}\circ \pi_{\pa E_1}\right)^{-1}(B_{E_1}\circ \pi_{\pa E_1})^2 \big)\\
&\,\,\,\,\, -\frac{h^2}{1 - |\nabla_{\tau_2} d_{E_1}|^2}\big( \left(I + d_{E_1}B_{E_1}\circ \pi_{\pa E_1}\right)^{-1}(B_{E_1}\circ \pi_{\pa E_1}) \big) \nabla_{\tau_2} H_{E_2} \cdot \nabla_{\tau_2} H_{E_2}
\end{split}
\]
The equation \eqref{eq:high-deri-1}  then follows from the previous calculations and from the identity
\beq \label{eq:dist-norms}
(\nu_{E_1} \circ \pi_{\pa E_1})  \cdot \nu_{E_2} = \sqrt{1- |\nabla_{\tau_2} d_{E_1}|^2}  \quad \text{on } \, \pa E_2 
\eeq
which is a direct consequence of Lemma \ref{lem:maaginen}.

We may differentiate the equality \eqref{eq:high-deri-1} further and obtain a discrete version of \eqref{eq:deri-high-conti} for every order $k$. This will produce several nonlinear error terms which have rather complicated structure. However, by introducing sufficiently  efficient notation we are able to identify the structure of these error terms and   by using the uniform ball condition and the interpolation inequality from Proposition   \ref{prop:interpolation} we are able to reproduce the argument from \cite{FJM3D} in the discrete setting. 
The following proposition  is the core of the proof for the higher order regularity. 
\begin{proposition}
\label{prop:maaginen-deri-2}
Assume that $E_1 \subset \R^{n+1}$ is an open and bounded set, with $|E_1| = m_0$, which satisfies uniform ball condition with radius $r_0$ and let $E_2$ be any minimizer of $\mathcal{F}_h( \ \cdot \ , E_1) $ defined in  \eqref{min mov}.
There is $h_0=h_0(n,m_0,r_0)$ such that if $h \leq h_0$ and $E_1$ is $C^{2m+3}$-regular for $m =  0,1,2, \dots$ then 
\[
\begin{split}
\Delta_{\tau_2}^m H_{E_2}  - (\Delta_{\tau_1}^m H_{E_1}) \circ \pi_{\pa E_1}  &=h\,   \Delta_{\tau_2} ^{m+1} H_{E_2}  +  h\,  R_{2m} \quad \text{and} \\
 \nabla_{\tau_2} \Delta_{\tau_2}^m H_{E_2}  - (\nabla_{\tau_1}  \Delta_{\tau_1}^m H_{E_1}) \circ  \pi_{\pa E_1}  &=h\,  \nabla_{\tau_2}   \Delta_{\tau_2}^{m+1} H_{E_2} - \pa_{\nu_{E_2}} (\Delta_{\tau_1}^m H_{E_1} \circ  \pi_{\pa E_1})  \nu_{E_2} +h\,   R_{2m+1}
\end{split}
\]
on $\pa E_2$ and the error term $R_{l}$ for $l = 0,1,2, \dots$ satisfies the estimate 
\[
\| R_l \|_{L^2(\pa E_2)}^2 \leq C_l  \left(1+ \| B_{E_2}\|_{H^{l+1}(\pa E_2)}^2 +  \| B_{E_1}\|_{H^{l}(\pa E_1)}^2\right),
\]
where $C_l = C_l(l,n,m_0,r_0)$.
\end{proposition}

We note that so far we have not used any results  from differential geometry. In fact, we need the notation from geometry only to prove Proposition \ref{prop:maaginen-deri-2}. Therefore,   instead of giving the proof of Proposition \ref{prop:maaginen-deri-2}, which is technically challenging, we  show first  how we may use it  to obtain the regularity estimate \eqref{eq:smoothing} in the statement of Theorem \ref{thm1}. Here is the main result of this section. 
\begin{theorem}
\label{thm3} Let $E_0$ be an open and bounded set, with $|E_0|=m_0$, and let $(E_t^h)_{t \geq 0}$ be an approximative flat flow 
starting from $E_0$ defined  in \eqref{app flow}.
For given $r_0 \in \R_+$  there is $h_0=h_0(n,m_0,r_0) \in \R_+$
such that if $h \leq h_0$,  $E_t^h$ satisfies uniform ball condition with radius $r_0$ in $[0,T]$ and if $(l+2)h \leq T$ for a given $l \in \N \cup \{0\}$, then  we have 
\[
\sup_{t \in [(l+2) h,T]} \left( (t-lh)^l \|H_{E_t^h}\|_{H^{l}(\pa E_t^h)}^2 \right) + \int_{(l+2)h}^{T} (t-lh)^{l}  \|H_{E_t^h}\|_{H^{l+1}(\pa E_t^h)}^2 \, \d t  \leq C,
\]
for a constant $C = C(l,n,m_0,r_0,T)$. 
\end{theorem}

\begin{proof}
In the proof, $C$ and $C_m$ denote a positive real number
which may change their values but always in a manner that 
we have the dependencies $C=C(n,m_0,r_0)$ and $C_m=C_m(m,n,m_0,r_0, T)$.
We use the abbreviation $E_k = E^h_{kh}$ for $k=0,1,2,\ldots$

First, by Proposition \ref{prop:distance-bound},  Lemma \ref{lemma:improved-regularity}, Remark \ref{rem:improved-regularity} and Theorem \ref{thm2},  we find $h_0=h_0(n,m_0,r_0) >0$ such that if $h \leq h_0$ and $E_k$ is $C^{2k+1}$-regular, bounded set of volume $m_0$, which satisfies uniform ball condition with radius $r_0$, then  the consequent set $E_{k+1}$ is  $C^{2k+3}$-regular, bounded and of volume $m_0$,
with
\[
\|d_{E_k}\|_{L^\infty(\pa E_{k+1})}\leq Ch < r_0/2. 
\]
Moreover, $E_{k+1}$ satisfies uniform ball condition with radius $r_0/2$  and the projection $\pi_{\pa E_k} : \pa E_{k+1} \rightarrow \pa E_k$ is injective. 
We may then prove that, for $k\geq 1$, $\pi_{\pa E_k} : \pa E_{k+1} \rightarrow \pa E_k$ is a  diffeomorphism with
\beq \label{eq:thm3-0}
J_{\tau_{k+1}} \pi_{\pa E_k} \geq 1 - Ch > 0 \ \ \text{on} \ \ \pa E_{k+1},
\eeq
where the tangential Jacobian $J_{\tau_{k+1}}{\pi_{\pa E_k}}$ of $\pi_{\pa E_k}$ on $\pa E_{k+1}$ is defined  in \eqref{def:tJacob}. 
Indeed, since $\pa E_{k+1} \subset \mathcal N_{r_0/2}(\pa E_k)$, then $\pi_{\pa E_k}$ is $C^1$-regular map on $\pa E_{k+1}$. 
Recalling the injectivity of the projection we are remain to prove \eqref{eq:thm3-0}. By \eqref{eq:pro-1} we may write
\[
\nabla \pi_{\pa E_{k+1}} = I - \nabla d_{E_k} \otimes \nabla d_{E_k} - d_{E_k} \nabla^2 d_{E_k} \ \ \text{on} \ \ \pa E_{k+1}.
\]
Thus, it follows from the definition in \eqref{def:tJacob} and $\nabla^2 d_{E_k} \nabla d_{E_k}=0$ in $\mathcal N_{r_0}(\pa E_k)$
that for given a point $x \in \pa E_{k+1}$ there is an orthonormal basis $v_1, \dots, v_n$ of $G_x \pa E_{k+1}$ such that   
\begin{align*}
J_{\tau_{k+1}} \pi_{\pa E_k}(x) 
&= \prod_{i=1}^n \left| \left( I-  \nabla d_{E_k}(x) \otimes \nabla d_{E_k}(x) - d_{E_k}(x)\nabla^2 d_{E_k}(x)\right) v_i\right| \\
&= \prod_{i=1}^n \left(1 - (\nabla d_{E_k}(x) \cdot v_i)^2 - 2d_{E_k}(x)\nabla^2 d_{E_k}(x) v_i \cdot v_i  + |d_{E_k}(x)|^2|\nabla^2 d_{E_k}(x)v_i|^2 \right)^\frac12.
\end{align*}
Since $\pa E_{k+1} \subset \mathcal N_{r_0/2}(\pa E_k)$, then Lemma \ref{lemma:globalLip} yields $\sup_{\pa E_{k+1}} |\nabla^2 d_{E_k}|_{\op} \leq C$. 
Further, since $E_{k+1}$ satisfies uniform ball condition with radius $r_0/2$, then by Lemma \ref{lemma:simpleinterpolation} and by the previous estimates we deduce
\[
|\nabla d_{E_k}(x)\cdot v_i|^2 \leq |\nabla_{\tau_2} d_{E_k}(x)|^2 \leq 4\|d_{E_k}\|_{L^\infty(\pa E_{k+1})}\left(\sup_{\pa E_{k+1}}|\nabla^2 d_{E_k}|_{\op} + \frac{\|\nabla d_{E_k}\|_{L^\infty(\pa E_{k+1})}}{r_0/2}\right) \leq Ch .
\]
Therefore, by combining the previous observations and shrinking $h_0$, if needed, we obtain \eqref{eq:thm3-0}. 
Again, by possibly shrinking $h_0$, we may assume that the implications of Proposition \ref{prop:maaginen-deri-2} hold true
for the parameters $m_0$ and $r_0/2$.

Let us from now on assume that  the sets  $E^h_t$ satisfy uniform ball condition with radius $r_0$ 
for every $t \in [0,T]$. Let us denote $K = \lfloor T/h\rfloor$. Then the previous discussion holds for every $E_k$ and $k=0,1,2,\ldots,K$. For the sake of presentation, we use abbreviations $\| B_{E_k}\|_{L^2}=\| B_{E_k}\|_{L^2(\pa E_k)}$,
$ \| B_{E_k}\|_{H^{2m}} = \| B_{E_k}\|_{H^{2m}(\pa E_k)}$  etc.

After the initialization, we prove the claim by induction and to this aim we begin by  proving the main regularity estimates.
We claim that for every $m =0,1,2, \dots$, with $m  \leq K-2$, and every $k =m+1, m+2,  \dots, K$ it holds   
\beq \label{eq:thm3-1}
\|\Delta_{\tau_{k+1}}^m H_{E_{k+1}}\|_{L^2}^2  \leq (1+ C_mh)\|\Delta_{\tau_k}^m H_{E_{k}}\|_{L^2}^2  - h \|\nabla_{\tau_{k+1}}\Delta_{\tau_{k+1}}^m H_{E_{k+1}}\|_{L^2}^2  + C_m h 
\eeq
and
\beq \label{eq:thm3-2}
 \|\nabla_{\tau_{k+1}} \Delta_{\tau_{k+1}}^m H_{E_{k+1}}\|_{L^2}^2\leq (1+ C_mh) \|\nabla_{\tau_{k}} \Delta_{\tau_{k}}^m H_{E_{k}}\|_{L^2}^2- h \|\Delta_{\tau_{k+1}}^{m+1} H_{E_{k+1}}\|_{L^2}^2  + C_m h .
\eeq
We first prove \eqref{eq:thm3-1} and fix $m$. Recall that  for $k \geq m+1$ the set $E_k$ is $C^{2m+3}$-regular. Therefore by Proposition  \ref{prop:maaginen-deri-2} it holds for every $k =m+1, m+2,  \dots, K$
\[
\Delta_{\tau_{k+1}}^m H_{E_{k+1}}  - (\Delta_{\tau_k}^m H_{E_k}) \circ \pi_{\pa E_k}  =h\,   \Delta_{\tau_{k+1}} ^{m+1} H_{E_{k+1}}  +  h\,  R_{2m,k} \ \ \text{on} \ \ \pa E_{k+1},
\]
where the  remainder term $ R_{2m,k}$ satisfies 
\[
\| R_{2m,k}\|_{L^2}^2 \leq C_m \big(1+ \|B_{E_{k+1}}\|_{H^{2m+1}}^2+  \|B_{E_k}\|_{H^{2m}}^2 \big).
\] 
Again, since $E_k$ and $E_{k+1}$ satisfy uniform ball condition with radius $r_0/2$ and $|E_k|=m_0=|E_{k+1}|$,
then $\|B_{E_k}\|_{L^\infty},\|B_{E_k}\|_{L^\infty} \leq C$ by \eqref{eq:curva-bound-tri} and $P(E_k),P(E_{k+1}) \leq C$ by Lemma \ref{lemma:calibration}.
Therefore, we may use Proposition \ref{prop:mean-curv} and Young's inequality to deduce 
\begin{align}
\notag
\|B_{E_k}\|_{H^{2m}}^2 &\leq C_m\left( 1+  \|\Delta_{\tau_{k}}^m H_{E_k}\|_{L^{2}}^2\right) \quad \text{and}  \\
\label{eq:thm3-1b}
\|B_{E_{k+1}}\|_{H^{2m+1}}^2  &\leq C_m \left( 1+ \|\nabla_{\tau_{k+1}}\Delta_{\tau_{k+1}}^m H_{E_{k+1}} \|_{L^{2}}^2\right).
\end{align}
We also observe that $ \|\Delta_{\tau_{k+1}}^m H_{E_{k+1}}\|_{L^2}^2 \leq C_m \|B_{E_{k+1}}\|_{H^{2m}}$. 
Let $\eps \in (0,1)$ be a number which we will choose later.
By using the previous observations, \eqref{eq:thm3-0},  Young's inequality,
and integration by parts we estimate as follows
\[
\begin{split}
 &\|\Delta_{\tau_{k+1}}^m H_{E_{k+1}}\|_{L^2}^2 -  \|\Delta_{\tau_{k}}^m H_{E_{k}}\|_{L^2}^2   \\
&\leq    \int_{\pa E_{k+1}} |\Delta_{\tau_{k+1}}^m H_{E_{k+1}}|^2 - |\Delta_{\tau_k}^m H_{E_k}\circ \pi_{\pa E_k}|^2\, \d \Ha^n +  C h \, \int_{\pa E_{k+1}}|\Delta_{\tau_k}^m H_{E_k}\circ \pi_{\pa E_k}|^2\, \d \Ha^n  \\
&\leq   \int_{\pa E_{k+1}} |\Delta_{\tau_{k+1}}^m H_{E_{k+1}}|^2 - |\Delta_{\tau_k}^m H_{E_k}\circ \pi_{\pa E_k}|^2\, \d \Ha^n +  \frac{C h}{1-Ch} \,  \|\Delta_{\tau_{k}}^m H_{E_{k}}\|_{L^2}^2  \\
&\leq  2 \int_{\pa E_{k+1}} \Delta_{\tau_{k+1}}^m H_{E_{k+1}} ( \Delta_{\tau_{k+1}}^m H_{E_{k+1}}  - \Delta_{\tau_k}^m H_{E_k}\circ \pi_{\pa E_k})\, \d \Ha^n + Ch \, \|\Delta_{\tau_{k}}^m H_{E_{k}}\|_{L^2}^2\\
&= 2h\int_{\pa E_{k+1}} \Delta_{\tau_{k+1}}^m H_{E_{k+1}} (  \Delta_{\tau_{k+1}} ^{m+1} H_{E_{k+1}}  + R_{2m,k} )\, \d \Ha^n + Ch\, \|\Delta_{\tau_{k}}^m H_{E_{k}}\|_{L^2}^2 \\
&\leq - 2h  \|\nabla_{\tau_{k+1}}\Delta_{\tau_{k+1}}^m H_{E_{k+1}} \|_{L^{2}}^2  +  \eps h \, \| R_{2m,k}\|_{L^2}^2  +
\frac{h}\eps\,  \|\Delta_{\tau_{k+1}}^m H_{E_{k+1}}\|_{L^2}^2+ C h \, \|\Delta_{\tau_{k}}^m H_{E_{k}}\|_{L^2}^2 \\
&\leq - 2h  \|\nabla_{\tau_{k+1}}\Delta_{\tau_{k+1}}^m H_{E_{k+1}} \|_{L^{2}}^2  + 
 C_mh \, \left(\eps \|B_{E_{k+1}}\|_{H^{2m+1}}^2 + \frac{1}{\eps}\|B_{E_{k+1}}\|_{H^{2m}}^2\right) + C_m h \,\left(1+ \|\Delta_{\tau_{k}}^m H_{E_{k}}\|_{L^2}^2\right) \\
&\leq - 2h  \|\nabla_{\tau_{k+1}}\Delta_{\tau_{k+1}}^m H_{E_{k+1}} \|_{L^{2}}^2  + 
 C_mh \, \left(\eps \|\nabla_{\tau_{k+1}}\Delta_{\tau_{k+1}}^m H_{E_{k+1}} \|_{L^{2}}^2 + \frac{1}{\eps}\|B_{E_{k+1}}\|_{H^{2m}}^2\right) + C_m h \,\left(1+ \|\Delta_{\tau_{k}}^m H_{E_{k}}\|_{L^2}^2\right) .
\end{split}
\]
By choosing $\eps = (1+C_m)^{-1}/2$, the previous estimate yields
\beq
\label{eq:thm3-1c}
\begin{split}
\|\Delta_{\tau_{k+1}}^m &H_{E_{k+1}}\|_{L^2}^2-  \|\Delta_{\tau_{k}}^m H_{E_{k}}\|_{L^2}^2 \\
&\leq - \frac{3h}{2} \|\nabla_{\tau_{k+1}}\Delta_{\tau_{k+1}}^m H_{E_{k+1}}\|_{L^2}^2
+ C_m h \, \|B_{E_{k+1}}\|_{H^{2m}}^2 +  C_m h \,\left(1+ \|\Delta_{\tau_{k}}^m H_{E_{k}}\|_{L^2}^2\right).
\end{split}
\eeq

Since $\|B_{E_{k+1}}\|_{L^\infty}, P(E_{k+1}) \leq C$, we may use Proposition \ref{prop:interpolation} to 
find $\theta = \theta(m,n) \in (0,1)$ such that 
\[ 
 \|B_{E_{k+1}}\|_{H^{2m}}^2 \leq C_m \|B_{E_{k+1}}\|_{H^{2m+1}}^{2\theta} \|B_{E_{k+1}}\|_{L^\infty}^{2(1-\theta)} \leq  C_m \eps \, \|B_{E_{k+1}}\|_{H^{2m+1}}^2 +C_m \eps^{-\frac{\theta}{1-\theta}}
\]
for any $\eps \in (0,1)$, where the last inequality follows from Young's inequality and 
the curvature bound. Thus, by combing the above with \eqref{eq:thm3-1c} and
\eqref{eq:thm3-1b} the estimate \eqref{eq:thm3-1} follows with a suitable choice of $\eps$.

Let us then prove \eqref{eq:thm3-2}.  The argument is similar than above and we only point out the main differences. 
Now Proposition  \ref{prop:maaginen-deri-2} gives for every $k=m+1, m+2, \dots, K$ the formula 
\begin{align*}
&\nabla_{\tau_{k+1}} \Delta_{\tau_{k+1}}^m H_{E_{k+1}}  - (\nabla_{\tau_k}  \Delta_{\tau_k}^m H_{E_k}) \circ  \pi_{\pa E_k} \\ =&h\,  \nabla_{\tau_{k+1}}   \Delta_{\tau_{k+1}}^{m+1} H_{E_{k+1}} - \pa_{\nu_{E_{k+1}}} (\Delta_{\tau_k}^m H_{E_k} \circ  \pi_{\pa E_k})  \nu_{E_{k+1}} 
+h\,   R_{2m+1,k}\quad \text{on} \ \ \pa E_{k+1},
\end{align*}
where 
\[
\| R_{2m+1,k}\|_{L^2}^2 \leq C_m\left(1+ \|B_{E_{k+1}}\|_{H^{2m+2}}^2+  \|B_{E_k}\|_{H^{2m+1}}^2 \right),
\]
and, again, by using  Proposition \ref{prop:mean-curv} and Young's inequality we have estimates 
\begin{align*}
\|B_{E_k}\|_{H^{2m+1}}^2 &\leq C_m( 1+  \|\nabla_{\tau_k}\Delta_{\tau_{k}}^m H_{E_k}\|_{L^{2}}^2) \quad \text{and} \\
\|B_{E_{k+1}}\|_{H^{2m+2}}^2  &\leq  C_m\left( 1 +  \|\Delta_{\tau_{k+1}}^{m+1} H_{E_{k+1}} \|_{L^{2}}^2\right).
\end{align*}
We use the previous observations, the Cauchy-Schwarz inequality, the estimate $\|\nabla_{\tau_{k+1}}  \Delta_{\tau_{k+1}}^m H_{E_{k+1}}\|_{L^2} \leq  C_m \|B_{E_{k+1}}\|_{H^{2m+1}}^2$ 
and argue as in proving \eqref{eq:thm3-1}  to deduce
\[
\begin{split}
 &\|\nabla_{\tau_{k+1}} \Delta_{\tau_{k+1}}^m H_{E_{k}}\|_{L^2}^2 - \|\nabla_{\tau_{k}} \Delta_{\tau_{k}}^m H_{E_{k}}\|_{L^2}^2 \\
&\leq \int_{\pa E_{k+1}} |\nabla_{\tau_{k+1}} \Delta_{\tau_{k+1}}^m H_{E_{k+1}}|^2 - |\nabla_{\tau_k} \Delta_{\tau_k}^m H_{E_k}\circ \pi_{\pa E_k}|^2\, \d \Ha^n +  Ch \, \|\nabla_{\tau_{k}} \Delta_{\tau_{k}}^m H_{E_{k}}\|_{L^2}^2  \\
&\leq 2  \int_{\pa E_{k+1}} \nabla_{\tau_{k+1}}  \Delta_{\tau_{k+1}}^m H_{E_{k+1}}\cdot  ( \nabla_{\tau_{k+1}} \Delta_{\tau_{k+1}}^m H_{E_{k+1}}  - \nabla_{\tau_{k}} \Delta_{\tau_k}^m H_{E_k}\circ \pi_{\pa E_k})\, \d \Ha^n 
+Ch \, \|\nabla_{\tau_{k}} \Delta_{\tau_{k}}^m H_{E_{k}}\|_{L^2}^2 \\
&\leq - 2h \|\Delta_{\tau_{k+1}}^{m+1} H_{E_{k+1}}\|_{L^2}
+\eps h \, \| R_{2m+1,k}\|_{L^2(\pa E_{k+1})}^2 + \frac{C_m}\eps h\, \|B_{E_{k+1}}\|_{H^{2m+1}}^2
+Ch \, \|\nabla_{\tau_{k}} \Delta_{\tau_{k}}^m H_{E_{k}}\|_{L^2}^2 \\
&\leq - 2h \|\Delta_{\tau_{k+1}}^{m+1} H_{E_{k+1}}\|_{L^2} + 
C_mh \, \left(\eps \,  \|\Delta_{\tau_{k+1}}^{m+1} H_{E_{k+1}} \|_{L^{2}}^2 + \frac{1}{\eps}\|B_{E_{k+1}}\|_{H^{2m+1}}^2\right) + C_m h \,\left(1+ \|\nabla_{\tau_k}\Delta_{\tau_{k}}^m H_{E_{k}}\|_{L^2}^2\right) \\
&\leq - \frac{3h}{2} \|\Delta_{\tau_{k+1}}^{m+1} H_{E_{k+1}}\|_{L^2} + 
C_m h \, \|B_{E_{k+1}}\|_{H^{2m+1}}^2 + C_m h \,\left(1+ \|\nabla_{\tau_k}\Delta_{\tau_{k}}^m H_{E_{k}}\|_{L^2}^2\right).
\end{split}
\]
Again,  Proposition \ref{prop:interpolation} implies that there is $\theta = \theta(m,n) \in (0,1)$ such that
\[
\|B_{E_{k+1}}\|_{H^{2m+1}}^2 \leq C_m \|B_{E_{k+1}}\|_{H^{2m+2}}^{2\theta} \|B_{E_{k+1}}\|_{L^\infty}^{2(1-\theta)}
\]
and we may proceed as previously to obtain   \eqref{eq:thm3-2}.

Let us then prove the claim by induction. To be more precise, under the assumption $h \leq h_0$, we claim the following. 
For every $l \in \N \cup \{0\}$ it holds 

\beq 
\label{eq:thm3-3}
\max_{l+2\leq k \leq K} \big((k-(l+1))h\big)^l \|H_{E_k} \|_{H^l}^2  + \sum_{k=l+2}^K  h \big((k-(l+1))h\big)^l \|H_{E_k}\|_{H^{l+1}}^2 \, \d t  \leq C_l
\eeq
for $C_l = C_l(l,n,m_0,r_0,T)$, provided that $ (l+2) h \leq T$. 
Since $t - lh \leq 3 \lfloor t / h\rfloor h -3(l+1)h $ for every $t \geq (l+2)h$, then by multiplying \eqref{eq:thm3-3} by $3^l$ and recalling the definition for
the approximative solution in \eqref{app flow}, we obtain the statement of the theorem.

Let us  consider first the case $l =0$. 
Since $P(E_k),\|B_{E_k}\|_{L^\infty}\leq C$, then 
$\|H_{E_k}\|_{L^2} \leq C$ for every $k=0,1,\ldots,K$.
By combining this with \eqref{eq:thm3-1} gives us that for every $k=1,2,\ldots,K-1$
\[
\| H_{E_{k+1}}\|_{L^2}^2  - \|H_{E_{k}}\|_{L^2}^2 \leq  - h \|\nabla_{\tau_{k+1}}  H_{E_{k+1}}\|_{L^2}^2  + C h .
\]
We sum over $k=1,2,\dots , K-1$ and use $\|H_{E_k}\|_{L^2} \leq C$ as well as $K h \leq T$ to obtain
\[
 \| H_{E_{K}}\|_{L^2}^2 + \sum_{k=1}^{K-1}  h \|\nabla_{\tau_{k+1}}  H_{E_{k+1}}\|_{L^2}^2  \leq \| H_{E_{1}}\|_{L^2}^2 +   CKh \leq CT.
\] 
Thus, we conclude that  \eqref{eq:thm3-3} holds in the case $l=0$.

Let us then assume that \eqref{eq:thm3-3} holds for $l-1$, where $l \in \N$. We assume that $(l+2)h\leq T$ and prove \eqref{eq:thm3-3} for $l$. To this aim, we denote $K' = K - l$ and $E'_k = E_{k+l}$.
Again, let $\tau_k$ denote the tangential differentiation along $\pa E'_k$. 
Thus, the induction assumption reads as
\beq
\label{eq:thm3-4}
\max_{1 \leq k \leq K'} (kh)^{l-1} \|H_{E'_k} \|_{H^{l-1}}^2  + \sum_{k=1}^{K'} h (kh)^{l-1} \|H_{E'_k}\|_{H^{l}(\pa E_t^h)}^2  \leq C_{l-1}.
\eeq
 We divide the argument into  two cases depending whether $l$ is even or odd. 

Let us first assume that $l$ is even and thus is of the form 
$l = 2m$ for $m=1,2,\dots$. By binomial expansion it holds $(k+1)^{2m} - k^{2m} \leq 2m(k+1)^{2m-1}$. Therefore, by multiplying \eqref{eq:thm3-1} by $k^{2m} h^{2m}$ we deduce for every
$k=0,1,2,\ldots, K'$
\[
\begin{split}
 (k+1)^{2m}& h^{2m} \|\Delta_{\tau_{k+1}}^m H_{E'_{k+1}}\|_{L^2}^2 -  k^{2m} h^{2m}  \|\Delta_{\tau_k}^m H_{E'_{k}}\|_{L^2}^2  \\
&= \big((k+1)^{2m} - k^{2m}\big) h^{2m} \|\Delta_{\tau_{k+1}}^m H_{E'_{k+1}}\|_{L^2}^2 +  k^{2m} h^{2m}\big(\|\Delta_{\tau_{k+1}}^m H_{E'_{k+1}}\|_{L^2}^2-\|\Delta_{\tau_{k}}^m H_{E'_{k}}\|_{L^2}^2 \big)\\
&\leq 2m  (k+1)^{2m-1} h^{2m}    \|\Delta_{\tau_{k+1}}^m H_{E'_{k+1}}\|_{L^2}^2  +  C_m k^{2m} h^{2m+1}\left(1+  \|\Delta_{\tau_{k}}^m H_{E'_{k}}\|_{L^2}^2 \right)\\
&\,\,\,\,\,\,\,\,\,\,\,\,  -  k^{2m} h^{2m+1} \|\nabla_{\tau_{k+1}}\Delta_{\tau_{k+1}}^m H_{E'_{k+1}}\|_{L^2}^2 .
\end{split}
\]
Fix any $j=2,\ldots K'$. Summing the previous estimate from $k =0$ to $k = j-1$ and using the fact $K'h \leq T$ yields
\[
\begin{split}
 &j^{2m} h^{2m}  \|\Delta_{\tau_j}^m H_{E'_{j}}\|_{L^2}^2 \\
&\leq  C_m \sum_{k=0}^{j-1}  (k+1)^{2m-1} h^{2m}    \|\Delta_{\tau_{k+1}}^m H_{E'_{k+1}}\|_{L^2}^2  +  C_m\sum_{k=0}^{j-1}(kh) k^{2m-1} h^{2m} \|\Delta_{\tau_{k}}^m H_{E'_{k}}\|_{L^2}^2  + C_m\sum_{k=0}^{j-1} h \, k^{2m} h^{2m}\\
&\,\,\,\,\,\,\,\,\,\,\,\,  - \sum_{k=0}^{j-1}  k^{2m} h^{2m+1} \|\nabla_{\tau_{k+1}}\Delta_\tau^m H_{E'_{k+1}}\|_{L^2}^2  \\
&\leq C_m(1+T) \sum_{k=1}^{j} k^{2m-1} h^{2m}    \|\Delta_{\tau_{k}}^m H_{E'_{k}}\|_{L^2}^2  +  C_m \int_{0}^{K'h} s^{2m} \, \d s  
-   \sum_{k=1}^{j} h \,  (k-1)^{2m} h^{2m} \|\nabla_{\tau_k}\Delta_{\tau_k}^m H_{E'_k}\|_{L^2}^2  \\
&\leq C_m(1+T) \sum_{k=1}^{K'} h \,  k^{2m-1} h^{2m-1} \|\Delta_{\tau_{k}}^m H_{E'_{k}}\|_{L^2}^2  +  C_m  T^{2m+1}  - \sum_{k=1}^{j} h \,  (k-1)^{2m} h^{2m} \|\nabla_{\tau_k}\Delta_{\tau_k}^m H_{E'_k}\|_{L^2}^2.
\end{split}
\]
Thus, reordering the previous estimate and using the induction assumption \eqref{eq:thm3-4}
gives us
\begin{align*}
(j-1)^{2m} h^{2m} \|\Delta_{\tau_j}^m H_{E_j'}\|_{L^2}^2& +\sum_{k=1}^{j} h \,  (k-1)^{2m} h^{2m} \|\nabla_{\tau_k}\Delta_{\tau_k}^m
H_{E_k'}\|_{L^2}^2 \\
\leq & C_m(1+T) \sum_{k=1}^{K'}  k^{2m-1} h^{2m}  \|\Delta_{\tau_{k}}^m H_{E_k'}\|_{L^2}^2  + C_mT^{2m+1}\\
\leq & C_m(1+T) \sum_{k=1}^{K'} h (kh)^{l-1} \| H_{E_{k}'}\|_{H^l}^2  + C_mT^{2m+1} \\
\leq & C_m C_{l-1}+ C_m T^{2m+1}.
\end{align*}

After substituting $E'_k=E_{k+l}$ and reindexing we have for every $j=l+2,\ldots,K$

\[
\big((j-(l+1))h\big)^{l} \|\Delta_{\tau_j}^m H_{E_k}\|_{L^2}^2 + \sum_{k=l+1}^{j} h \,  \big((k-(l+1))h\big)^{l} \|\nabla_{\tau_k}\Delta_{\tau_k}^m H_{E_k}\|_{L^2}^2 \leq  C_m C_{l-1}+ C_m T^{2m+1}.
\]
Since we have $\|B_{E_k}\|_{L^\infty},P(E_k) \leq C$ for every $k \leq K$, then by combining the estimates of Proposition \ref{prop:mean-curv}
with the previous estimate and using $Kh \leq T$ we obtain \eqref{eq:thm3-3}.

The case when $l$ is odd is similar. In this case, we have $l = 2m+1$ for some $m \in \N \cup \{0\}$.
Thus, by using \eqref{eq:thm3-2} in the place of \eqref{eq:thm3-1} we may proceed as in the previous case.
\end{proof}

Let us then focus on Proposition \ref{prop:maaginen-deri-2}. We will begin by proving two technical lemmas which involve  high order derivatives of $d_E$ and $\pi_{\pa E}$. To overcome the technicalities  we adopt the notation where $A_i$ denotes a generic  tensor field, which depends on  the distance function, the normal and the second fundamental form in a smooth way, i.e., 
\beq \label{def:A-i}
A_i = A_i(d_E, \nu_E \circ \pi_{\pa E}, B_E \circ \pi_{\pa E}) \ \ \text{in} \ \ \mathcal{N}_{r/2}(\pa E).
\eeq
 We also adopt here  the notation   $S \star T$  to denote a tensor formed by contraction on some indexes of  tensors $S$ and $T$. If the set $E$ satisfies uniform ball condition, then the quantities $d_E, \nu_E$ and $ B_E \circ \pi_{\pa E}$ are uniformly bounded in $\mathcal{N}_{r/2}(\pa E)$, we may treat $A_i$ in \eqref{def:A-i} as a bounded coefficient. 

It is immediate that  it holds for $x \in \pa E$ and $u  \in C^2(\pa E)$
\[
\nabla (u \circ \pi_{\pa E})(x) = \nabla_\tau u(x) \quad \text{and} \quad \Delta_{\R^{n+1}} (u \circ \pi_{\pa E})(x) = \Delta_{\tau} u(x).
\]
Let us then derive related formulas for points $x \in \mathcal{N}_{r/2}(\pa E)$ outside $\pa E$.

\begin{lemma}
\label{lem:away-bdr}
Assume $E \subset \R^{n+1}$, with $\Sigma = \pa E$, is bounded and $C^3$-regular set which satisfies uniform ball condition with radius $r$. Then   it holds for $u  \in C^2(\pa E)$  in  $ \mathcal{N}_{r/2}(\pa E)$
 \[
\begin{split}
\nabla (u\circ \pi_{\pa E}) &= \nabla (u\circ \pi_{\pa E}) \circ  \pi_{\pa E}    - d_E \nabla^2d_E  \nabla (u\circ \pi_{\pa E}) \circ \pi_{\pa E} 
\end{split}
\]
and 
\[
\begin{split}
\nabla^2 (u\circ \pi_{\pa E}) = &(P_{\pa E} \circ \pi_{\pa E}) (\nabla^2 (u\circ \pi_{\pa E}) \circ  \pi_{\pa E}   )
- \nabla d_{E} \otimes\nabla^2d_{E}   \nabla (u\circ \pi_{\pa E}) \circ  \pi_{\pa E} \\
&+ d_{E} \, A_1 \star \nabla^2 (u\circ \pi_{\pa E}) \circ \pi_{\pa E} +  d_{E} \, A_2 \star \nabla (B_{E} \circ \pi_{\pa E}) \circ \pi_{\pa E} \star \nabla (u\circ \pi_{\pa E}) \circ \pi_{\pa E} 
\end{split}
\]
where  $A_1, A_2$ are tensor fields  as in \eqref{def:A-i}.  
Moreover, if $\Sigma$ is in addition $C^{k+2}$-regular and $u \in C^k(\Sigma)$  for $k \in \N$, then  for all $x \in  \mathcal{N}_{r/2}(\pa E)$  we may estimate
\[
|\nabla^k (u\circ \pi_{\pa E})(x) | \leq C_k \sum_{|\alpha|\leq k} \big(1+|\tilde \nabla_{\Sigma}^{\alpha_1}  B_{E}(\pi_{\pa E}(x)) | \cdots  |\tilde \nabla_{\Sigma}^{\alpha_{k-1}} B_{E}(\pi_{\pa E}(x)) |\big) \,|\tilde \nabla_{\Sigma}^{\alpha_{k}} u (\pi_{\pa E}(x)) |.  
\]
Here $\tilde \nabla_{\Sigma}$ denotes the covariant derivative on $\Sigma$.  
\end{lemma}

\begin{proof}
Let us  denote $\hat u = u\circ \pi_{\pa E}$ and $\pi =  \pi_{\pa E}$ for short. Since $\pi$ is projection it holds 
\[
\hat u (x) = \hat u( \pi(x))
\]
for all $x \in \mathcal{N}_{r/2}(\pa E)$. By differentiating this we obtain
\[
\nabla \hat u(x) =\nabla \pi(x)   \nabla \hat u ( \pi(x)) . 
\]
The first claim then follows from \eqref{eq:pro-1} and from $\nabla \hat u \cdot (\nu_E \circ \pi)  = 0$. The second claim follows by differentiating the first and by writing $\nabla^2 d_E(x), \nabla^3 d_E(x)$ and $\nabla \pi$ in 
a geometric way by using \eqref{eq:dist-3} and \eqref{eq:pro-2}.

In order to prove the third claim we  observe that we may write the second equality simply as 
\[
\nabla^2 \hat u (x) =  A_1(x) \star  \nabla^2\hat u \big( \pi(x)\big) + A_2(x) \star  \nabla  (B_{E}\circ \pi)(\pi_{\pa E}(x))\star  \nabla \hat u ( \pi(x)) .
\]
By differentiating this $(k-2)$-times and by using \eqref{eq:leibniz}  and \eqref{eq:pro-2} we deduce
\[
|\nabla^k \hat u (x)| \leq C_k \sum_{|\alpha|\leq k} C\big( 1 + | \nabla^{\alpha_1}    (B_{E}\circ \pi)(\pi(x))|\cdots  | \nabla^{\alpha_{k-1}}    (B_{E}\circ \pi)(\pi(x))| \big) |  \nabla^{\alpha_k} \hat u ( \pi(x))|.
\] 
The claim follows once we show that for all $y \in \Sigma$ it holds 
\beq \label{away-bdr-1}
|  \nabla^l \hat u  (y)| \leq  C_l  \sum_{|\beta|\leq l }\big( 1+|\tilde \nabla^{\beta_1}  B_{E}(y) |\cdots  |\tilde \nabla^{\beta_l}  B_{E}(y)|\big) | \bar  \nabla^{\beta_{l+1} } u (y)|,
\eeq
which is the opposite estimate as Lemma \ref{lem:cova-eucl}. 

We argue as in the proof of Lemma \ref{lem:cova-eucl} and assume $y = 0$, $\nu_E(0) = e_{n+1}$  and  write the surface  $\Sigma$ locally as a graph of $f$, i.e., $\Sigma \cap B_r \subset \{ (x',f(x')) : x' \in \R^n\}$ and extended $f$ to $\R^{n+1}$ trivially as $f(x', x_{n+1}) = f(x')$. We may then write the metric tensor and the Christoffel symbols in coordinates as
\[
g_{ij}(x') = \delta_{ij} + \pa_i f(x') \pa_j f(x') \quad \text{and} \quad \Gamma_{jk}^i(x') = g^{il}(x')\pa_{jk}^2 f(x') \pa_l f(x').
\] 
Since $\nu_{E}  = \frac{(-\nabla_{\R^n} f,1)}{\sqrt{1 + |\nabla_{\R^n}f |^2}}$  and $\nabla \hat u \cdot (\nu_E \circ \pi)  = 0$, we have 
\beq \label{away-bdr-2}
\pa_{n+1}  \hat u(y) = \sum_{i =1}^n \pa_i f(\pi(y)) \cdot \pa_i  \hat u(y).
\eeq
Let us denote the $l$th order differential  of the function $x' \to \hat u(x',0)$  as $\nabla_{\R^n}^l \hat u$. Then by applying first \eqref{away-bdr-2} and \eqref{eq:leibniz},  and  then \eqref{eq:cova-eucl} we deduce that 
\[
\begin{split}
|  \nabla^l \hat u  (0)| &\leq  C_l  \sum_{|\beta|\leq l -1}\big( 1+|\nabla^{\beta_1} (\nabla f \circ \pi) (0) |\cdots  |\nabla^{\beta_{l-1}} (\nabla f \circ \pi)  (0) |\big) | \nabla_{\R^n}^{1+\beta_{l} } \hat u (0)|\\
&\leq C_l  \sum_{|\gamma|\leq l-1}\big( 1+|\tilde \nabla^{\gamma_1}  B_E(0)) |\cdots  |\tilde \nabla^{\gamma_{l-1}}  B_E(0)|\big) | \nabla_{\R^n}^{1+\gamma_{l} } \hat u (0)|.
\end{split}
\]

Denote the local chart given by the coordinate parametrization by $\Phi$, i.e.,  $\Phi^{-1}(x') = (x', f(x'))$ and note that $\hat u (\Phi^{-1}(x')) =  u(\Phi^{-1}(x'))$. Fix an index    vector $\beta = (\beta_1, \dots, \beta_n , 0) $ with $|\beta| = m$. Then by \eqref{eq:leibniz} and \eqref{eq:cova-eucl} we obtain after straightforward calculations
\[
\begin{split}
|\nabla^{\beta} \,  (u \circ \Phi^{-1})(0)| &\geq  |\nabla^{\beta} \,  \hat u (0)| -  C_m \sum_{|\gamma|\leq m -1}\big(1+ |\nabla^{1 + \gamma_1} f (0) |\cdots  |\nabla^{1 +\gamma_{m-1}}  f (0)|\big)||\nabla^{\gamma_{m}} \, \hat  u(0)|\\
&\geq  |\nabla^{\beta} \,  \hat u (0)| -  C_m \sum_{|\gamma|\leq m -1}\big(1+ |\tilde \nabla^{\gamma_1}  B_E(0)|\cdots  |\tilde \nabla^{\gamma_{m-1}}  B_E(0)|\big)||\nabla^{\gamma_{m}} \, \hat  u(0)|.
\end{split}
\]
From here we deduce by an inductive argument that 
\[
|\nabla^{\beta} \,  \hat u (0)| \leq C_m \sum_{|\gamma|\leq m}\big(1+ |\tilde \nabla^{\gamma_1} B_E(0)|\cdots  |\tilde \nabla^{\gamma_{m}}  B_E(0)|\big)||\nabla^{\gamma_{m+1}} \,  (u \circ \Phi^{-1})(0)|.
\]
Finally using the definition of the covariant derivative and the expression of the Christoffel symbols we obtain arguing  as in the proof of Lemma \ref{lem:cova-eucl}  that
\[
|\nabla^{m} \,  (u \circ \Phi^{-1})(0)|\leq  C_m \sum_{|\gamma|\leq m}\big(1+ |\tilde \nabla^{\gamma_1}  B_E(0)|\cdots  |\tilde \nabla^{\gamma_{m}}  B_E(0)|\big)||\tilde \nabla^{\gamma_{m+1}} \, u(0)|.
\]
Hence, we have \eqref{away-bdr-1} and the third claim follows.
\end{proof}

Let us from now on  assume $E_1, E_2\subset \R^{n+1}$ are as in  Proposition \ref{prop:maaginen-deri-2}. We write  the equality \eqref{eq:high-deri-1} by using the Euler-Lagrange equation \eqref{euler} as 
\beq \label{eq:deriva-1}
H_{E_2}  - H_{E_1} \circ \pi_{\pa E_1}  =h\,   \Delta_{\tau_2} H_{E_2} + h \, \rho_0(\cdot ) 
\eeq
on  $ \pa E_2$, where the error function is of the form 
\beq \label{def:rho-0}
\rho_0(x) = A_1(x)  + h \, A_2(x) \star \nabla_{\tau_2} H_{E_2}(x) \star \nabla_{\tau_2} H_{E_2}(x) .
\eeq
Here and in the rest of the section $A_i(\cdot)$ denotes  a tensor field which depends smoothly on $d_{E_1}, \nu_{E_1} \circ \pi_{\pa E_1}$, 
 $\nu_{E_2}$, $B_{E_1} \circ \pi_{\pa E_1}$ and on $B_{E_2}$. i.e., 
\beq \label{def:tensor12}
A_i(x) = A_i\big(d_{E_1}(x),\nu_{E_1}(\pi_{\pa E_1}(x)), \nu_{E_2}(x),  B_{E_1}(\pi_{\pa E_1}(x)), B_{E_2}(x) \big).
\eeq

The following lemma  is a consequence of  Lemma \ref{lem:away-bdr}. 
\begin{lemma}
\label{lem:laplace12}
Assume that the sets $E_1, E_2\subset \R^{n+1}$ are as in  Proposition \ref{prop:maaginen-deri-2}. Then  it holds for $u \in C^2(\pa E_1)$ on $ \pa E_2 $  
\[
\begin{split}
\Delta_{\tau_2} (u\circ \pi_{\pa E_1}) = &\Delta_{\tau_1}  u \circ \pi_{\pa E_1}    + h\,  A_1 \star \nabla^2(u\circ \pi_{\pa E_1}) \circ \pi_{\pa E_1}    \\
&+ h^2 \,  A_2 \star   \nabla^2(u\circ \pi_{\pa E_1}) \circ \pi_{\pa E_1} \star  \nabla_{\tau_2} H_{E_2} \star \nabla_{\tau_2} H_{E_2} \\
&+ h\, A_3 \star \nabla (B_{E_1}\circ \pi_{\pa E_1}) \circ \pi_{\pa E_1}   \star \nabla (u\circ \pi_{\pa E_1}) \circ \pi_{\pa E_1}\\
&+ h \,  A_4 \star \nabla_{\tau_2} B_{E_2}   \star \nabla (u\circ \pi_{\pa E_1}) \circ \pi_{\pa E_1}. 
\end{split}
\]
\end{lemma}

\begin{proof}
Let us denote $\hat u = u\circ \pi_{\pa E_1}$ and $\pi = \pi_{\pa E_1}$ for short. Recall that we may write the Laplace-Beltrami on $\pa E_2$ as
\beq \label{eq:laplace12}
\Delta_{\tau_2} \hat u = \Delta_{\R^{n+1}} \hat u - (\nabla^2 \hat u \,\nu_{E_2} \cdot \nu_{E_2}) - H_{E_2}\pa_{\nu_{E_2}} \hat u,  
\eeq
where $\Delta_{\R^{n+1}} \hat u = \text{Tr}(\nabla^2 \hat u)$ denotes the Euclidian Laplacian. Recall that  $P_{\pa E_1}= I - \nu_{E_1} \otimes \nu_{E_1}$ stands for the projection on the (geometric) tangent space. We deduce by applying the trace on  the second equality in  Lemma \ref{lem:away-bdr}, by $\nabla^2 d_{E_1} \nabla d_{E_1} = 0$, and by the Euler-Lagrange equation \eqref{euler} that it holds on $\pa E_2$
\beq \label{eq:laplace12-1}
\begin{split}
 \Delta_{\R^{n+1}} \hat u   =\text{Tr}(\nabla^2 \hat u) = \Delta_{\tau_1}  u \circ \pi   &+ h\,   A_1 \star (\nabla^2 \hat u \circ \pi )    \\
&+ h\,  A_3 \star \nabla (B_{E_1}\circ \pi) \circ \pi   \star (\nabla \hat u \circ \pi ).
\end{split}
\eeq
Similarly we have 
\beq \label{eq:laplace12-2}
\begin{split}
(\nabla^2 \hat u \,\nu_{E_2} )\cdot \nu_{E_2}  = \, &\big((P_{\pa E_1} \circ \pi) (\nabla^2 \hat u \circ \pi) \,\nu_{E_2} \big)  \cdot \nu_{E_2}     \\
&- (\nabla d_{E_1}\cdot \nu_{E_2})  \big( \nabla^2 d_{E_1}   (\nabla \hat u \circ \pi) \cdot \nu_{E_2} \big)\\
&+ h\,  \tilde A_1\star (\nabla^2 \hat u \circ \pi )  + h\, \tilde  A_3 \star \nabla (B_{E_1}\circ \pi) \circ \pi  \star (\nabla \hat u \circ \pi ) .
\end{split}
\eeq
We write 
\[
\big((P_{\pa E_1} \circ \pi) (\nabla^2 \hat u \circ \pi) \,\nu_{E_2} \big)  \cdot \nu_{E_2}  =  \big((P_{\pa E_1} \circ \pi) (\nabla^2 \hat u \circ \pi) \,(\nu_{E_2}- \nu_{E_1} \circ \pi) \big)  \cdot (\nu_{E_2}- \nu_{E_1} \circ \pi)
\]
and 
\[
(\nabla d_{E_1}\cdot \nu_{E_2})  \big( \nabla^2 d_{E_1}   (\nabla \hat u \circ \pi) \cdot \nu_{E_2} \big)    = (\nabla d_{E_1}\cdot \nu_{E_2})  \big( \nabla^2 d_{E_1}   (\nabla \hat u \circ \pi) \cdot (\nu_{E_2}- \nu_{E_1} \circ \pi) \big).  
\]
We then use \eqref{eq:maagi2} to write $\nu_{E_2}- \nu_{E_1} \circ \pi $ as
\[
\nu_{E_2}- \nu_{E_1} \circ \pi = a_1  \nabla_{\tau_2} d_{E_1}  + a_2 \, (\nu_{E_1}\circ \pi) 
\]
for functions $a_1$ and $a_2$ which depend on $|\nabla_{\tau_2} d_{E_1}(x)|^2$. Therefore we may write \eqref{eq:laplace12-2} by the Euler-Lagrange equation \eqref{euler} as
\beq \label{eq:laplace12-3}
\begin{split}
(\nabla^2 \hat u \,\nu_{E_2} )\cdot \nu_{E_2}   &= h\,  A_1 \star (\nabla^2 \hat u \circ \pi ) \\
&+h^2 \,  A_2 \star  (\nabla^2 \hat u \circ \pi) \star  \nabla_{\tau_2} H_{E_2} \star \nabla_{\tau_2} H_{E_2}     \\
&+ h\,  A_3 \star \nabla (B_{E_1}\circ \pi) \circ \pi    \star (\nabla \hat u \circ \pi )\\
&+ h\,  A_4  \star \nabla_{\tau_2} H_{E_2}   \star (\nabla \hat u \circ \pi).
\end{split}
\eeq

We use  the first equality in  Lemma \ref{lem:away-bdr},  \eqref{eq:maagi2} and the Euler-Lagrange equation \eqref{euler} to write  on $\pa E_2$
\beq \label{eq:laplace12-4}
\begin{split}
\pa_{\nu_{E_2}} \hat u &= (\nabla  \hat u \circ  \pi)\cdot  \nu_{E_2}+ h \, A_3    \star (\nabla \hat u \circ \pi )\\
&=  (\nabla  \hat u \circ  \pi) \cdot  (\nu_{E_2} -\nu_{E_1}\circ \pi) + h \, A_3  \star (\nabla \hat u \circ \pi )\\
&= h\,  A_4 \star \nabla_{\tau_2} H_{E_2}   \star (\nabla  \hat u \circ  \pi) +  h \, A_3  \star (\nabla \hat u \circ \pi ).
\end{split}
\eeq
The claim then follows from \eqref{eq:laplace12},  \eqref{eq:laplace12-1}, \eqref{eq:laplace12-3} and \eqref{eq:laplace12-4}.
\end{proof}

We may now prove Proposition \ref{prop:maaginen-deri-2}.
\begin{proof}[\textbf{Proof of Proposition \ref{prop:maaginen-deri-2}}]
We prove only the first equality since the second follows by differentiating the first. We point out that since $E_1$ is $C^{2m+3}$-regular, then by Lemma \ref{lemma:improved-regularity} the set $E_2$ is $C^{2m+5}$-regular. In particular, we have the necessary regularity for the proceeding calculations.  To that aim we recall that by \eqref{eq:deriva-1} it holds 
\beq \label{eq:maagi-deri-1}
H_{E_2}  - H_{E_1} \circ \pi_{\pa E_1}   =h\,   \Delta_{\tau_2} H_{E_2} + h \, \rho_0 \qquad \text{on }  \pa E_2,
\eeq
where 
\[
\rho_0(x) = A_1(x)  + h \, A_2(x) \star \nabla_{\tau_2} H_{E_2}(x) \star  \nabla_{\tau_2} H_{E_2}(x) .
\]
We  differentiate \eqref{eq:maagi-deri-1},  use Lemma \ref{lem:laplace12} and have  on $\pa E_2$
\[
\Delta_{\tau_2} H_{E_2}(x)  - \Delta_{\tau_1} H_{E_1}  \circ \pi_{\pa E_1}  =h\,   \Delta_{\tau_2}^2 H_{E_2} + h \rho_2 +  h \, \Delta_{\tau_2} \rho_0,
\]
where 
\[
\begin{split}
\rho_2 = &A_1 \star \nabla^2(H_{E_1}\circ \pi_{\pa E_1}) \circ \pi_{\pa E_1}    \\
&+ h\,  A_2 \star  \nabla^2(H_{E_1}\circ \pi_{\pa E_1}) \circ \pi_{\pa E_1}  \star  \nabla_{\tau_2} H_{E_2} \star \nabla_{\tau_2} H_{E_2} \\
&+  A_3 \star \nabla (B_{E_1}\circ \pi_{\pa E_1}) \circ \pi_{\pa E_1}  \star \nabla (H_{E_1}\circ \pi_{\pa E_1}) \circ \pi_{\pa E_1}\\
&+   A_4 \star \nabla_{\tau_2} B_{E_2}  \star \nabla (H_{E_1} \circ \pi_{\pa E_1}) \circ \pi_{\pa E_1}. 
\end{split}
\]
We  continue and deduce  by an iterative argument  that it holds on $\pa E_2$
\[
\Delta_{\tau_2}^m H_{E_2}  - \Delta_{\tau_1} H_{E_1}^m  \circ \pi_{\pa E_1}  =h\,   \Delta_{\tau_2}^{m+1} H_{E_2} + h \sum_{k=0}^m \Delta_{\tau_2}^{m -k}  \rho_{2k},
\]
where $\rho_0$ is defined in \eqref{def:rho-0} and $\rho_{2k}$ for $k \geq 1$ is 
\[
\begin{split}
\rho_{2k} = &A_1 \star \nabla^2(\Delta_{\tau_1}^{k-1} H_{E_1}\circ \pi_{\pa E_1}) \circ \pi_{\pa E_1}    \\
&+ h\,  A_2 \star  \nabla^2(\Delta_{\tau_1}^{k-1}  H_{E_1}\circ \pi_{\pa E_1})\circ \pi_{\pa E_1}  \star  \nabla_{\tau_2} H_{E_2} \star \nabla_{\tau_2} H_{E_2} \\
&+  A_3 \star \nabla (B_{E_1}\circ \pi_{\pa E_1})\circ \pi_{\pa E_1}   \star \nabla (\Delta_{\tau_1}^{k-1} H_{E_1}\circ \pi_{\pa E_1})\circ \pi_{\pa E_1} \\
&+   A_4 \star \nabla_{\tau_2} B_{E_2} \star \nabla (\Delta_{\tau_1}^{k-1} H_{E_1} \circ \pi_{\pa E_1})\circ \pi_{\pa E_1}. 
\end{split}
\]
We have thus derived a formula for the error terms in the statement of Proposition \ref{prop:maaginen-deri-2}, i.e., we have 
\[
R_{2m}(x) = \sum_{k=0}^m \Delta_{\tau_2}^{m -k}  \rho_{2k}(x).
\]
We need to estimate the norm $\|R_{2m}\|_{L^2(\Sigma_2)}$, where $\Sigma_2 = \pa E_2$. The idea is that  the total amount of derivatives acting on the curvature terms in  $\Delta_{\tau_2}^{m -k}  \rho_{2k}$ is for most of the terms at most  $2m$. The only difference is the second row in the definition of $\rho_{2k}$, which total amount of derivatives is higher but it has an extra $h$ as a coefficient. Therefore we need to treat this term more carefully. 

 Recall that the tensor fields $A_i(\cdot)$ depend on $d_{E_1}, \nu_{E_1} \circ \pi_{\pa E_1}$, 
 $\nu_{E_2}$, $B_{E_1} \circ \pi_{\pa E_1}$ and on $B_{E_2}$  as stated in   \eqref {def:tensor12}.  Denote $\pi = \pi_{\pa E_1}$ for short.  We use repeatedly  \eqref{eq:leibniz},  Lemma \ref{lem:cova-eucl} and  the last inequality in Lemma \ref{lem:away-bdr}  and obtain after long but straightforward calculations  the following  pointwise  estimate   for all $x \in \pa E_2$ 
\beq \label{eq:maagi-deri-2}
\begin{split}
\big| \Delta_{\tau_2}^{m -k}  \rho_{2k}(x) \big| \leq C &+ C \sum_{|\alpha|\leq 2m} |\tilde \nabla^{\alpha_1} B_{\Sigma_2}(x)| \cdots  |\tilde \nabla^{\alpha_{2m}} B_{\Sigma_2}(x)|  \\
&+ C \sum_{|\alpha|\leq 2m} |\tilde \nabla^{\alpha_1} B_{\Sigma_1}(\pi(x))| \cdots  |\tilde \nabla^{\alpha_{2m}} B_{\Sigma_1}(\pi(x))|\\
&+  C h \sum_{|\alpha|\leq 2m} (|\tilde \nabla^{\alpha_1} B_{\Sigma_2}(x)| + |\tilde \nabla^{\alpha_1} B_{\Sigma_1}(\pi(x))|)  \cdots   (|\tilde \nabla^{\alpha_{2m}} B_{\Sigma_2}(x)| + |\tilde \nabla^{\alpha_{2m}} B_{\Sigma_1}(\pi(x))|) \cdots \\
&\,\,\,\,\,\,\,\,\,\,\,\,\,\,\,\,\,\,\,\,\,\,\,\,\,\,\,\,\,\,\,\,\,\,\, \cdots |\tilde \nabla^{1+\alpha_{2m+1}} H_{E_2}(x)| \,  |\tilde \nabla^{1+\alpha_{2m+2}} H_{E_2}(x)|.
\end{split}
\eeq 
We use the uniform curvature bounds $\| B_{\Sigma_1}\|_{L^\infty}, \| B_{\Sigma_2}\|_{L^\infty} \leq C$ and Proposition \ref{prop:kato-ponce} to estimate 
\[
\sum_{|\alpha|\leq 2m} \| |\tilde \nabla^{\alpha_1} B_{\Sigma_2}(x)| \cdots  |\tilde \nabla^{\alpha_{2m}} B_{\Sigma_2}(x)|\|_{L^2(\Sigma_2)} \leq C \|B_{\Sigma_2}\|_{H^{2m}(\Sigma_2)}
\]
and
\[
\begin{split}
\sum_{|\alpha|\leq 2m} &\| |\tilde \nabla^{\alpha_1} B_{\Sigma_1}(\pi(x))| \cdots  |\tilde \nabla^{\alpha_{2m}} B_{\Sigma_1}(\pi(x))|\|_{L^2(\Sigma_2)} \\
&\leq C \sum_{|\alpha|\leq 2m} \| |\tilde \nabla^{\alpha_1} B_{\Sigma_1}(y)| \cdots  |\tilde \nabla^{\alpha_{2m}} B_{\Sigma_1}(y)|\|_{L^2(\Sigma_1)}  \leq C \|B_{\Sigma_1}\|_{H^{2m}(\Sigma_1)}.
\end{split}
\]
We are left with the last term in \eqref{eq:maagi-deri-2}. As we already mentioned, this term has different scaling with respect to  $h$.  We use the Euler-Lagrange equation \eqref{euler},  \eqref{gradinterpolation}  and $\|d_{E_1}\|_{L^\infty(\pa E_2)} \leq Ch$ from Proposition \ref{prop:distance-bound} to deduce that 
\[
\|\tilde \nabla H_{E_2}\|_{L^\infty(\Sigma_2)}^2 \leq \frac{C}{h}. 
\]
Therefore we have by Proposition \ref{prop:kato-ponce}  
\[
\begin{split}
h &\sum_{|\alpha|\leq 2m} \|(|\tilde \nabla^{\alpha_1} B_{\Sigma_2}(x)| + |\tilde \nabla^{\alpha_1} B_{\Sigma_1}(\pi(x))|)  \cdots   (|\tilde \nabla^{\alpha_{2m}} B_{\Sigma_2}(x)| + |\tilde \nabla^{\alpha_{2m}} B_{\Sigma_1}(\pi(x))|)  \cdots \\
&\,\,\,\,\,\,\,\,\,\,\,\,\,\,\,\,\,\,\,\,\,\,\,\,\,\,\,\,\,\,\,\,\,\,\, \cdots  |\tilde \nabla^{1+\alpha_{2m+1}} H_{E_2}(x)| \,  |\tilde \nabla^{1+\alpha_{2m+2}} H_{E_2}(x)|\|_{L^2(\Sigma_2)}\\
&\leq Ch \|\tilde \nabla H_{E_2}\|_{L^\infty(\Sigma_2)}^2 \big(\|B_{\Sigma_1}\|_{H^{2m}(\Sigma_1)} +\|B_{\Sigma_2}\|_{H^{2m}(\Sigma_2)}\big)  + Ch \|\tilde \nabla H_{E_2}\|_{L^\infty(\Sigma_2)}  \|H_{E_2}\|_{H^{2m+1}(\Sigma_2)} \\
&\leq  C \|B_{\Sigma_1}\|_{H^{2m}(\Sigma_1)} + C \|B_{\Sigma_2}\|_{H^{2m}(\Sigma_2)}   + C \sqrt{h} \|H_{E_2}\|_{H^{2m+1}(\Sigma_2)} \\
&\leq C \|B_{\Sigma_1}\|_{H^{2m}(\Sigma_1)}  + C \|B_{\Sigma_2}\|_{H^{2m+1}(\Sigma_2)} 
\end{split}
\]
when $h \leq 1$, and the claim follows. 
\end{proof}

Let us conclude this section by discussing briefly how we obtain Theorem \ref{thm1} and Corollary \ref{coro}  from the results in Sections 4 and 5. 
We obtain first from Lemma \ref{lem:2-point-arg} and from Theorem \ref{thm2} that the approximative flow  $(E_t^h)_k$ satisfies  uniform ball condition with radius $r_0/2$ for $t \leq T_0$ and we have 
\begin{equation}
\label{eq:final-1}
\frac{\|S_{E_{t+h}^h}\|_{L^\infty} -\|S_{E_t^h}\|_{L^\infty}}{h} \leq C_n \|S_{E_t^h}\|_{L^\infty}^3.
\end{equation}
Then we use Theorem  \ref{thm3}  to deduce that for $t \in [\delta,T_0]$ the sets  $E_t^h$ are uniformly $C^3$-regular when $h$ is small enough. By Ascoli-Arzela theorem we may pass the estimate \eqref{eq:final-1} to the limit as $h \to 0$ and conclude that the function $t \mapsto \sup_{s \leq t}  \|S_{E_{s}}\|_{L^\infty}$  is locally  Lipschitz continuous and satisfies
\beq \label{eq:smoothing-flat} 
\frac{\d}{\d t} \big(\sup_{s \leq t}  \|S_{E_{s}}\|_{L^\infty}\big) \leq C_n \big(\sup_{s \leq t}  \|S_{E_{s}}\|_{L^\infty}\big)^3
\eeq
for almost every $t \geq 0$ as long as $\sup_{s \leq t}  \|S_{E_{s}}\|_{L^\infty}$ remains bounded. The inequality \eqref{eq:smoothing-flat}  implies that the uniform ball condition is an open condition. To be more precise  if the flat flow $(E_t^h)_t$, starting from $E_0$, satisfies $\sup_{t \leq T}  \|S_{E_{t}}\|_{L^\infty} \leq C$, then by \eqref{eq:smoothing-flat} there is $\delta >0$  such that 
\[
\sup_{t \leq T+\delta}  \|S_{E_{t}}\|_{L^\infty} \leq 2C.
\]
This together  with the estimate in Theorem  \ref{thm3} implies Theorem \ref{thm1}. 

The consistency principle follows from the regularity in a rather straightforward way. Indeed, we obtain by the uniform regularity of the approximate flat flow $(E_t^h)_{t\in [0,T]}$ and by  the Euler-Lagrange equation \eqref{euler}  that the signed distance function satifies
\[
\partial_t d_{E_t}(x) = \Delta_{\R^{n+1}} d_{E_t}(\pi_{\pa E_t}(x)) + f(t)
\]
for $t \leq T$ and for $x $ in a neighborhood of $\pa E_t$, where $f(t)$ is a  bounded function ot time. From here we may conclude that the flat flow satisfies
\[
V_t = - H_{E_t} + f(t). 
\]
Since the flat flow preserves the volume then necessarily $f(t) = \fint_{\pa E_t} H_{E_t} \ \d \Ha^n$ and thus it is a solution to  \eqref{eq:VMCF}.


\section*{Acknowledgments}
V.J. was supported by the Academy of Finland grant 314227. J.N. was partially supported by ERC-CZ grant LL2105 and the University Centre UNCE/SCI/023 of Charles University.


\end{document}